\documentclass[a4paper,10pt]{amsart}

\usepackage{amsmath}
\usepackage{amsthm}
\usepackage{amssymb}
\usepackage{amsfonts}
\usepackage{enumitem} % more enumerate/itemize options
\usepackage{mathrsfs} % for \mathscr
\usepackage{xcolor}

\usepackage{marginnote} % add margin notes

\usepackage{xr-hyper}

\usepackage[bookmarks=true,hyperindex,pdftex,colorlinks,citecolor=cyan,urlcolor=cyan]{hyperref}

\usepackage[a4paper,lmargin=3cm,rmargin=3cm,tmargin=4cm,bmargin=4cm,marginparwidth=2.8cm,marginparsep=1mm]{geometry} % reduce side margins a bit

\usepackage{soul} % for overstriking (\textst) and highlighting (\texthl)

\allowdisplaybreaks                                     % allows breaking of long align environments so that we don't get underfull pages

\makeatletter
\renewcommand{\tocsection}[3]{%
  \indentlabel{\@ifnotempty{#2}{\bfseries\ignorespaces#1 #2\quad}}\bfseries#3}
\renewcommand{\tocsubsection}[3]{%
  \indentlabel{\@ifnotempty{#2}{\ignorespaces#1 #2\quad}}#3}
\newcommand\@dotsep{4.5}
\def\@tocline#1#2#3#4#5#6#7{\relax
  \ifnum #1>\c@tocdepth % then omit
  \else
    \par \addpenalty\@secpenalty\addvspace{#2}%
    \begingroup \hyphenpenalty\@M
    \@ifempty{#4}{%
      \@tempdima\csname r@tocindent\number#1\endcsname\relax
    }{%
      \@tempdima#4\relax
    }%
    \parindent\z@ \leftskip#3\relax \advance\leftskip\@tempdima\relax
    \rightskip\@pnumwidth plus1em \parfillskip-\@pnumwidth
    #5\leavevmode\hskip-\@tempdima{#6}\nobreak
    \leaders\hbox{$\m@th\mkern \@dotsep mu\hbox{.}\mkern \@dotsep mu$}\hfill
    \nobreak
    \hbox to\@pnumwidth{\@tocpagenum{\ifnum#1=1\bfseries\fi#7}}\par% <-- \bfseries for \section page
    \nobreak
    \endgroup
  \fi}
\AtBeginDocument{%
\expandafter\renewcommand\csname r@tocindent0\endcsname{0pt}
}
\def\l@subsection{\@tocline{2}{0pt}{2.5pc}{5pc}{}}
\makeatother

\DeclareMathOperator{\lspan}{span}                          % linear span
\DeclareMathOperator{\supp}{supp}                           % support
\DeclareMathOperator{\Lip}{Lip}                             % Lipschitz functions
\DeclareMathOperator{\lip}{lip}                             % little Lipschitz functions

\newcommand{\NN}{\mathbb{N}}                                % natural numbers
\newcommand{\ZZ}{\mathbb{Z}}                                % integer numbers
\newcommand{\QQ}{\mathbb{Q}}                                % rational numbers
\newcommand{\RR}{\mathbb{R}}                                % real numbers
\newcommand{\abs}[1]{\left|{#1}\right|}                     % absolute value
\newcommand{\esupp}[1]{\mathcal{S}(#1)}                     % extended support in the bidual
\newcommand{\pare}[1]{\left({#1}\right)}                    % parentheses
\newcommand{\set}[1]{\left\{{#1}\right\}}                   % set by extension
\newcommand{\norm}[1]{\left\|{#1}\right\|}                  % norm
\newcommand{\dual}[1]{{#1}^\ast}                            % dual or adjoint
\newcommand{\ddual}[1]{{#1}^{\ast\ast}}                     % double dual
\newcommand{\duality}[1]{\left<{#1}\right>}                 % dual action
\newcommand{\cl}[1]{\overline{#1}}                          % closure
\newcommand{\wscl}[1]{\overline{#1}^{\dual{w}}}             % weak* closure
\newcommand{\wsconv}{\stackrel{\dual{w}}{\longrightarrow}}  % weak* convergence
\newcommand{\restrict}{\mathord{\upharpoonright}}           % restriction symbol

\newcommand{\lipfree}[1]{\mathcal{F}({#1})}                 % Lipschitz-free space
\newcommand{\lipnorm}[1]{\norm{#1}_L}                       % Lipschitz norm/constant
\newcommand{\meas}[1]{\mathcal{M}({#1})}                    % Radon measure space
\newcommand{\wt}[1]{\widetilde{#1}}                         % shorthand for the square de Leeuw domain
\newcommand{\bwt}[1]{\beta\wt{#1}}                          % shorthand for the full de Leeuw domain
\newcommand{\pp}{\mathfrak{p}}                              % projection map in de Leeuw domain
\newcommand{\rr}{\mathfrak{r}}                              % reflection map in de Leeuw domain
\newcommand{\opr}[1]{\mathcal{M}_{\mathrm{op}}(#1)}         % optimal positive representations in the de Leeuw domain
\newcommand{\norming}[1]{\mathcal{N}(#1)}                   % set of norming functions

\newcommand{\rcomp}[1]{#1^\mathcal{R}}                      % Lipschitz realcompactification
\newcommand{\ucomp}[1]{#1^\mathcal{U}}                      % uniform compactification
\newcommand{\ucl}[1]{\overline{#1}^{\mathcal{U}}}           % uniform closure
\newcommand{\rcl}[1]{\overline{#1}^{\mathcal{R}}}           % closure in Lipschitz realcompactification
\newcommand{\bwtf}{\mathsf{R}}                              % "finite points" in \bwt{M}
\newcommand{\bwtfp}{\mathsf{R}^+}                           % same but excluding the "diagonal"
\newcommand{\opt}{\mathsf{O}}                               % optimal points in \bwt{M}

\renewcommand{\geq}{\geqslant}
\renewcommand{\leq}{\leqslant}

\makeatletter
\newcommand{\labeltext}[2]{%
  \@bsphack
  \csname phantomsection\endcsname % in case hyperref is used
  \def\@currentlabel{#1}{\label{#2}}%
  \@esphack
}
\makeatother

\theoremstyle{plain}
\newtheorem{theorem}{Theorem}[section]
\newtheorem{lemma}[theorem]{Lemma}
\newtheorem{corollary}[theorem]{Corollary}
\newtheorem{proposition}[theorem]{Proposition}
\newtheorem*{claim}{Claim}

\theoremstyle{definition}
\newtheorem*{definition*}{Definition}
\newtheorem{definition}[theorem]{Definition}
\newtheorem{example}[theorem]{Example}

\newtheorem{problem}[theorem]{Problem}

\theoremstyle{remark}
\newtheorem{remark}[theorem]{Remark}

\numberwithin{equation}{section}

\begin{document}

\title{De Leeuw representations of functionals on Lipschitz spaces}

\author[R. J. Aliaga]{Ram\'on J. Aliaga}
\address[R. J. Aliaga]{Instituto Universitario de Matem\'atica Pura y Aplicada,
Universitat Polit\`ecnica de Val\`encia,
Camino de Vera S/N,
46022 Valencia, Spain}
\email{raalva@upvnet.upv.es}

\author[E. Perneck\'a]{Eva Perneck\'a}
\address[E. Perneck\'a]{Faculty of Information Technology, Czech Technical University in Prague, Th\'akurova 9, 160 00, Prague 6, Czech Republic}
\email{perneeva@fit.cvut.cz}

\author[R. J. Smith]{Richard J. Smith}
\address[R. J. Smith]{School of Mathematics and Statistics, University College Dublin, Belfield, Dublin 4, Ireland}
\email{richard.smith@maths.ucd.ie}

\date{}

% ABSTRACT

\begin{abstract}
Let $\mathrm{Lip}_0(M)$ be the space of Lipschitz functions on a complete metric space $(M,d)$ that vanish at a point $0\in M$. We investigate its dual $\mathrm{Lip}_0(M)^*$ using the De Leeuw transform, which allows representing each functional on $\mathrm{Lip}_0(M)$ as a (non-unique) measure on $\beta\widetilde{M}$, where $\widetilde{M}$ is the space of pairs $(x,y)\in M\times M$, $x\neq y$. We distinguish a set of points of $\beta\widetilde{M}$ that are ``away from infinity'', which can be assigned coordinates belonging to the Lipschitz realcompactification $M^{\mathcal{R}}$ of $M$. We define a natural metric $\bar{d}$ on $M^{\mathcal{R}}$ extending $d$ and we show that optimal (i.e. positive and norm-minimal) De Leeuw representations of well-behaved functionals are characterised by $\bar{d}$-cyclical monotonicity of their support, extending known results for functionals in $\mathcal{F}(M)$, the predual of $\mathrm{Lip}_0(M)$. We also extend the Kantorovich-Rubinstein theorem to normal Hausdorff spaces, in particular to $M^{\mathcal{R}}$, and use this to characterise measure-induced and majorisable functionals in $\mathrm{Lip}_0(M)^*$ as those admitting optimal representations with additional finiteness properties. Finally, we use De Leeuw representations to define a natural L-projection of $\mathrm{Lip}_0(M)^*$ onto $\mathcal{F}(M)$ under some conditions on $M$. % 175 words
\end{abstract}

% KEYWORDS

%46B20 Geometry and structure of normed linear spaces
%46E15 Banach spaces of continuous, differentiable or analytic functions
%49Q22 Optimal transportation
%54D35 Extensions of spaces (compactifications, supercompactifications, completions, etc.)
%54E50 Complete metric spaces


\maketitle

% MAIN DOCUMENT

\tableofcontents

\section{Introduction}\label{sect_introduction}

\subsection{Background and preliminaries}\label{subsect_notation_prelims}

We begin by recalling the definitions of Lipschitz and Lipschitz-free spaces. Throughout, $(M,d)$ will be a complete metric space with a fixed base point $0\in M$. Given a Lipschitz function $f$ on $M$ and taking values in another metric space, we denote by $\Lip(f)$ its (optimal) Lipschitz constant. When $f$ takes values in a normed space $(X,\norm{\cdot})$ we write
\[
\lipnorm{f} = \Lip(f) = \sup\set{\frac{\norm{f(x)-f(y)}}{d(x,y)} \,:\, x\neq y\in M}.
\]
We write $\Lip(M)$ for the space of all real-valued Lipschitz functions on $M$, and $\Lip_0(M)$ for the subspace consisting of all $f\in\Lip(M)$ such that $f(0)=0$ (hopefully this notation will not lead to confusion with $\Lip(f)$). The function $\norm{\cdot}_L$ is a seminorm on $\Lip(M)$ and a norm on $\Lip_0(M)$, and it is straightforward to check that $(\Lip_0(M),\lipnorm{\cdot})$ is complete. We shall call the Banach space $\Lip_0(M)$ and its dual space $\Lip_0(M)^*$ the \emph{Lipschitz space} and \emph{dual Lipschitz space} over $M$, respectively.

The Lipschitz space $\Lip_0(M)$ is a dual Banach space. For every $x\in M$, we define the evaluation functional $\delta(x)$ by $\duality{f,\delta(x)}=f(x)$, $f \in \Lip_0(M)$; then $\delta:M\to\dual{\Lip_0(M)}$ is an isometric embedding. The \emph{Lipschitz-free space} over $M$ is defined to be
\[
\lipfree{M} = \cl{\lspan}\,\delta(M) \subset \dual{\Lip_0(M)},
\]
and is a natural isometric predual of $\Lip_0(M)$. As well as this, $\lipfree{M}$ possesses a key universal property: any Lipschitz map $f$ from $M$ to a real Banach space $X$ satisfying $f(0)=0$ can be extended to a unique bounded linear operator $T:\lipfree{M}\to X$, in the sense that $f=T \circ \delta$ and $\norm{T}=\norm{f}_L$. While these spaces have been known since the 1950s \cite{AE56}, the papers \cite{GK,Kalton} of Godefroy and Kalton in the early 2000s demonstrated to a new generation of researchers in functional analysis their key role in the linear and non-linear theory of Banach spaces, and ever since then they have been a popular topic of study within the community. These spaces also share an intimate connection with optimal transport theory, as described for example in \cite[Section 1]{APS23}.

The most important elements of $\lipfree{M}$ are arguably the so-called \emph{(elementary) molecules}
\[
m_{xy} = \frac{\delta(x)-\delta(y)}{d(x,y)} \in S_{\lipfree{M}}
\]
for pairs of points $x\neq y\in M$. They represent the action of taking incremental quotients of functions $f\in\Lip_0(M)$. It is easy to check that the set of all elementary molecules on $M$ is symmetric and norming for $\Lip_0(M)$, and so its closed convex hull is the whole unit ball $B_{\lipfree{M}}$. 

It is possible to represent general elements of $\dual{\Lip_0(M)}$ (including of course all elements of $\lipfree{M}$) using a key construction of K. de Leeuw \cite{deLeeuw}. We consider the completely metrisable space
\[
\wt{M} = \set{(x,y)\in M\times M \,:\, x\neq y},
\]
together with its Stone-\v Cech compactification $\bwt{M}$ and the space $\meas{\bwt{M}}$ of signed Radon measures on $\bwt{M}$. The \emph{De Leeuw transform} is the linear isometric embedding $\Phi:\Lip_0(M)\to C(\bwt{M})$ given by
\[
(\Phi f)(x,y) = \frac{f(x)-f(y)}{d(x,y)}, \quad (x,y) \in \wt{M}
\]
and extending continuously to $\bwt{M}$.

The adjoint operator $\Phi^*:\meas{\bwt{M}}\to\dual{\Lip_0(M)}$, which we will also refer to as the De Leeuw transform, is a quotient map. Hence, given $\phi\in\dual{\Lip_0(M)}$ there exists $\mu\in\meas{\bwt{M}}$ such that $\dual{\Phi}\mu=\phi$, meaning that
$$
\duality{f,\phi} = \int_{\bwt{M}}(\Phi f)\,d\mu, \quad f\in\Lip_0(M).
$$
Any measure $\mu$ of this sort will be called a \emph{De Leeuw representation} of $\phi$. The easiest such examples are measures $\delta_{(x,y)}$, $(x,y)\in\wt{M}$, which represent molecules $m_{xy}$ as $(\Phi f)(x,y)=\duality{f,m_{xy}}$. Here, we denote by $\delta_\zeta$ the Dirac measure associated with $\zeta \in \bwt{M}$. The symbol $\delta$ was already used above to denote the evaluation map $\delta:M\to\Lip_0(M)^*$, but we trust that this second use will cause no confusion.

The fundamental theory of Lipschitz and Lipschitz-free spaces is presented in \cite{Weaver2}, wherein the De Leeuw transforms (and variants thereof) are used extensively (we point out that, in this book, Lipschitz-free spaces are called ``Arens-Eells spaces'' after their original inventors, and $\lipfree{M}$ is denoted \AE$(M)$). The present authors initiated a programme to study systematically the De Leeuw transforms and, as a consequence, glean information about the structure of $\lipfree{M}$ and $\Lip_0(M)^*$. The first paper resulting from this programme is \cite{APS23}; this paper is the second and is intended to be a self-contained sequel. In the paper, we extend many of the results in \cite{APS23} that apply to elements of $\lipfree{M}$ to those of $\Lip_0(M)^*$ or, in some cases, the subspace of well-behaved ``functionals that avoid infinity''. Readers of this paper will benefit from having prior knowledge of the contents of \cite{APS23}, but it is not essential; indeed, the key concepts and results from it will be repeated later in this subsection and at various points thereafter.

We observe that most of the contemporary work of functional analysts in this direction has focussed on $\lipfree{M}$, and that the structure of its bidual $\Lip_0(M)^*$ is still relatively poorly understood. In particular, to date, there is no general ``representation theorem'' for $\Lip_0(M)^*$, should one exist (though isometric representations of $\Lip_0(M)^*$ exist in some cases, such as when $M$ is a convex domain in finite-dimensional space \cite{CKK19} or a closed subset of an $\RR$-tree \cite{Godard_2010}). The structure of $\Lip_0(M)^*$ has been investigated in \cite{AP_measures}; there, the authors note that a better understanding of this space could conceivably lead to the resolution of several important open problems in the theory of $\lipfree{M}$, such as when this space is complemented or even $L$-embedded in its bidual. In the absence of a general representation theorem for $\Lip_0(M)^*$, the De Leeuw transform provides an important and useful substitute which we exploit throughout this paper.

We spend the rest of this subsection introducing notation and some key notions and findings from \cite{APS23}. Notation pertaining to Banach space theory, and Lipschitz and Lipschitz-free spaces in particular, is standard. We repeat here only concepts of direct relevance to the paper, and refer the reader to \cite{APS23} and references therein for everything else.

Many of our arguments and results will involve measures on topological spaces. All topological spaces in this paper are assumed to be Hausdorff. By a measure on a Hausdorff space $X$ we always mean a (signed) Borel, countably additive measure, and by a \emph{Radon measure} we mean one that is regular, tight, and finite (i.e. has finite total variation).
By $\meas{X}$ we denote the Banach space of all Radon measures on $X$ endowed with the total variation norm.
Recall that $\meas{X}=C_0(X)^*$ when $X$ is locally compact.
The \emph{support} $\supp(\mu)$ of a measure $\mu$ is defined as the (closed) set of all $x\in X$ such that $\abs{\mu}(U)>0$ for all open neighbourhoods $U$ of $x$.
If $\mu$ is Radon then it is concentrated on its support, and $\abs{\mu}(\supp(\mu))=\norm{\mu}$.
If $f:X\to Y$ is a continuous map between Hausdorff spaces and $\mu\in\meas{X}$, then the \emph{pushforward measure} $f_\sharp\mu\in\meas{Y}$ is given by $f_\sharp\mu(E)=\mu(f^{-1}(E))$ for Borel $E\subset Y$.
It has the property that $\int_Y g\,d(f_\sharp\mu)=\int_X (g\circ f)\,d\mu$ for every Borel function $g:Y\to\RR$ such that $g\circ f$ is $\mu$-integrable.
We refer to \cite{Bogachev} for any undefined measure-theoretical notion.

Since the De Leeuw transform $\dual{\Phi}$ is non-expansive, any De Leeuw representation $\mu\in\meas{\bwt{M}}$ of a functional $\phi\in\dual{\Lip_0(M)}$ satisfies $\norm{\mu}\geq\norm{\phi}$. By the Hahn-Banach theorem, every functional admits representations $\mu$ such that $\norm{\mu}=\norm{\phi}$, and it is not hard to see that they can be chosen to be positive \cite[Proposition 3]{Aliaga}. We call such measures \emph{optimal (De Leeuw) representations} and denote by
\[
\opr{\bwt{M}} = \set{\mu\in\meas{\bwt{M}} \,:\, \mu\geq 0 \text{ and } \norm{\mu}=\norm{\dual{\Phi}\mu}}
\]
the set of such optimal elements.

We will refer often to results from \cite{Aliaga, APS23} in this paper. For the reader's convenience, we finish this subsection by listing three of them. The first one gathers some elementary properties of the set of optimal representations.

\begin{proposition}[{\cite[Proposition 2.1]{APS23}}]
\label{pr:opr_facts}
\
\begin{enumerate}[label={\upshape{(\alph*)}}]
\item For any $\phi\in\dual{\Lip_0(M)}$ there is $\mu\in\opr{\bwt{M}}$ such that $\dual{\Phi}\mu=\phi$.
\item If $\mu\in\opr{\bwt{M}}$ then $c\cdot\mu\in\opr{\bwt{M}}$ for every $c\geq 0$.
\item If $\mu\in\opr{\bwt{M}}$ and $\lambda\in\meas{\bwt{M}}$ satisfies $0\leq\lambda\leq\mu$, then $\lambda\in\opr{\bwt{M}}$.
\item If $\mu\in\opr{\bwt{M}}$ and $E$ is a Borel subset of $\bwt{M}$ then $\mu\restrict_E\in\opr{\bwt{M}}$.
\end{enumerate}
\end{proposition}

The second result shows that De Leeuw representations concentrated on $\wt{M}$ correspond to expressions of $\dual{\Phi}\mu$ as an integral of molecules in a literal sense.

\begin{proposition}[{\cite[Proposition 2.6]{APS23}}]
\label{pr:wt_bochner}
If $\mu\in\meas{\bwt{M}}$ then
$$
\dual{\Phi}(\mu\restrict_{\wt{M}}) = \int_{\wt{M}}m_{xy}\,d\mu(x,y)
$$
as a Bochner integral in $\lipfree{M}$.
\end{proposition}

Motivated by this result, we call an element of $\lipfree{M}$ a \emph{convex integral of molecules} if it can be written as $\Phi^*\mu$ for some $\mu\in\opr{\bwt{M}}$ that is concentrated on $\wt{M}$ \cite[Definition 2.7]{APS23}, and a convex \emph{series} of molecules if $\mu$ is moreover discrete.

For the third result, let us fix some additional notation. Given $\phi\in \Lip_0(M)^*$, denote by
$$
\norming{\phi} = \set{f\in S_{\Lip_0(M)} \,:\, \duality{f,\phi}=\norm{\phi}}
$$
the set of norming functions for $\phi$. This set may be empty in general, but it never is when $\phi\in\lipfree{M}$. There is a close relationship between norming functions and optimal De Leeuw representations.

\begin{lemma}[{\cite[Lemma 2.3]{APS23}}]
\label{lm:norming_phi1}
Let $\mu\in \meas{\bwt{M}}$ be a positive De Leeuw representation of $\phi\in\Lip_0(M)^*$. Then
\[
\set{f\in S_{\Lip_0(M)} \,:\, \Phi f(\zeta)=1 \textup{ for all }\zeta\in \supp(\mu)} \subset \norming{\phi},
\]
and the two sets are equal if $\mu$ is optimal. If the set on the left-hand side is non-empty then $\mu$ is optimal.
\end{lemma}

A few more results from \cite{APS23} will be quoted in later sections.

\subsection{Plan of the paper}

In Section \ref{sect_top_and_metric_prelims}, we lay down the necessary topological and metric preliminaries required throughout the rest of the paper, including a coordinate system for $\bwt{M}$; the uniform compactification $\ucomp{M}$ and Lipschitz realcompactification $\rcomp{M}$ of $M$; a natural embedding of $\rcomp{M}$ in $\Lip_0(M)^*$; a natural metric extension of $d$ to $\rcomp{M}$; a useful result about extending Lipschitz functions; and a well-behaved class of elements of $\Lip_0(M)^*$ (including all elements of $\lipfree{M}$) called \emph{functionals that avoid infinity}. Then, in Section \ref{sect_points_in_bwtM}, we classify points $\zeta \in \bwt{M}$ according to the value of $\norm{\Phi^*\delta_\zeta}$ (namely $1$, $0$ or strictly between the two), and expose some of their properties.

Sections \ref{sect_optimal_rep_and_optimal_trans} and \ref{sect_relations_between_supports} are dedicated to generalising several previous results from \cite{APS23} and \cite{Aliaga}, respectively, that applied to elements of $\lipfree{M}$ and subsets of $M$, to elements of $\Lip_0(M)^*$ and subsets of $\ucomp{M}$. In Section \ref{sect_optimal_rep_and_optimal_trans}, we relate De Leeuw representations to optimal transport. Optimal representations of functionals that avoid infinity are characterised in terms of cyclical monotonicity. We also prove a general version of the classical Kantorovich-Rubinstein theorem that applies to normal Hausdorff spaces, in particular to $\rcomp{M}$, and use it to characterise measure-induced and majorisable functionals in $\Lip_0(M)^*$ in terms of their optimal representations. Section \ref{sect_relations_between_supports} concerns the relationship between the ``extended support'' of a functional in $\Lip_0(M)^*$ and the supports of its representing measures.

Finally, in Section \ref{sect_representations_of_F(M)}, we focus on the representation of elements of $\lipfree{M}$ and show that if $M$ is proper and purely 1-unrectifiable, then $\lipfree{M}$ is L-embedded in its bidual. This is a known result, but here we provide an explicit and natural example of an L-projection of $\lipfree{M}^{**}$ onto $\lipfree{M}$. We also provide an explicit representation of $\Lip_0(M)^*$ for compact and countable $M$.

\section{Topological and metric preliminaries}\label{sect_top_and_metric_prelims}

\subsection{\texorpdfstring{The ``coordinates'' of elements of $\bwt{M}$}{Coordinates}}

It will be necessary to consider ``coordinates'' of elements of $\bwt{M}$. The natural way to define them is to extend the usual coordinate projections from $\wt{M}$ onto $M$ continuously to $\bwt{M}$. Of course, we have to determine the appropriate (compact) codomain of these extensions. A standard choice is $\beta M$, and indeed this space was used for the coordinate system in \cite{APS23}. However, this compactification of $M$ has the drawback that Lipschitz functions on $M$ cannot separate points of $\beta M$ in general (see e.g. \cite[Example 2.5]{AP_measures}). For this reason, it is sometimes more appropriate to consider instead the so-called \emph{uniform (or Samuel) compactification} of $M$, which we denote by $\ucomp{M}$. This compactification is studied systematically in \cite{Woods}, and features in \cite[Chapter 7]{Weaver2} and extensively in \cite{AP_measures}. We will make use of the following characterisation.

\begin{proposition}[{\cite[Corollary 2.4]{Woods}}]\label{pr:woodsuniform}
There is a compactification $\ucomp{M}$ of $M$, unique up to homeomorphism, that satisfies
\begin{enumerate}[label={\upshape{(\roman*)}}]
 \item every bounded uniformly continuous function $f:M \to \RR$ can be extended uniquely and continuously to $\ucomp{f}:\ucomp{M} \to \RR$, and
 \item\label{ucomp_sep} given $A,B \subset M$, their closures in $\ucomp{M}$ are disjoint if and only if $d(A,B)>0$.
\end{enumerate}
\end{proposition}

In particular, disjoint closed subsets of $\ucomp{M}$ can be separated by (continuous extensions of) Lipschitz functions on $M$, which implies that such functions separate points of $\ucomp{M}$. For clarity, the closure of $A\subset\ucomp{M}$ in $\ucomp{M}$ will be denoted $\ucl{A}$.

When $M$ is \emph{uniformly discrete}, that is, when non-zero distances in $M$ have a positive lower bound, we have $\ucomp{M}=\beta M$ by \cite[Theorem 3.4]{Woods}.

With this in mind, we denote by $\pp_i:\bwt{M} \to \ucomp{M}$, $i=1,2$, the unique continuous extension of the first and second coordinate projection of elements of $\wt{M}$, respectively. Furthermore, we define the coordinate map $\pp:\bwt{M}\to \ucomp{M} \times \ucomp{M}$ by $\pp(\zeta)=(\pp_1(\zeta),\pp_2(\zeta))$. We note that if $\zeta \in \bwt{M}$ and $\pp(\zeta)=(x,y) \in \wt{M}$, then $\zeta = (x,y)$, but in general $\pp$ is not injective. It will help us to know the range of $\pp$.

\begin{proposition}\label{prop_p_range}
We have
\[
\pp(\bwt{M}) \;=\; (\ucomp{M} \times \ucomp{M})\setminus\set{(x,x)\,:\,x \in M \text{ is isolated}}.
\]
\end{proposition}

\begin{proof}
Let $(\xi,\eta) \in \ucomp{M} \times \ucomp{M}$, with $(\xi,\eta) \neq (x,x)$ for all isolated $x \in M$. Then either $\xi \neq \eta$ or $\xi=\eta$ is an accumulation point of $\ucomp{M}$ (as an isolated point of $\ucomp{M}$ must belong to $M$). It follows that given open $U \ni \xi$ and $V \ni \eta$, we have $(U \times V) \cap \wt{M} \neq \varnothing$. Thus by compactness there exists $\zeta$ in the set
\[
\bigcap_{U,V} \cl{((U \times V) \cap \wt{M})}^{\bwt{M}},
\]
where the intersection is taken over all such pairs of open sets as above. Then $\pp(\zeta) = (\xi,\eta)$. Indeed, given open $U,V$ as above, let open $U' \ni \xi$, $V' \ni \eta$ satisfy $\cl{U'} \times \cl{V'} \subset U \times V$. Then $\zeta \in \cl{((U' \times V') \cap \wt{M})}^{\bwt{M}}$, so by continuity
\[
\pp(\zeta) \in \cl{U'} \times \cl{V'} \;\subset\; U \times V.
\]
Because $\ucomp{M}$ is Hausdorff, we must have $\pp(\zeta)=(\xi,\eta)$.

Conversely, suppose that $x \in M$ is isolated. Then $x$ is isolated in $\ucomp{M}$, whence $\pp^{-1}(x,x)$ is an open subset of $\bwt{M}$; by density of $\wt{M}$, we have $\pp^{-1}(x,x)=\varnothing$.
\end{proof}

\subsection{The Lipschitz realcompactification of \texorpdfstring{$M$}{M} and extended metric}\label{subsec_realcompact}

Regarding extensions of unbounded functions, we have that every Lipschitz function $f:M \to \RR$ can be extended uniquely and continuously to $\ucomp{f}:\ucomp{M} \to [-\infty,\infty]$ \cite[Proposition 2.9]{AP_measures}. To avoid problems with infinity, most of the time we will restrict our attention to the set
$$
\rcomp{M}:=\set{\xi \in \ucomp{M} \,:\, \ucomp{\rho}(\xi) < \infty} ,
$$
where $\rho:M\to \RR$ is given by
\begin{equation}\label{eq:rho}
\rho(x)=d(x,0) .
\end{equation}
Given the maximality of $\rho$ in $B_{\Lip_0(M)}$, it is clear that $\xi \in \ucomp{M}$ belongs to $\rcomp{M}$ if and only if $\abs{\ucomp{f}(\xi)}<\infty$ for all Lipschitz functions $f:M \to \RR$; equivalently, if $\xi$ is the limit of a bounded net in $M$. The set $\rcomp{M}$, which is dense and open in $\ucomp{M}$, is called the \emph{Lipschitz realcompactification} of $M$ in \cite{GarridoMerono}. Evidently, $\rcomp{M}=\ucomp{M}$ if and only if $M$ is bounded, and $\rcomp{M} = M$ if and only if $M$ is \emph{proper}, i.e. its closed balls are compact. Given a Lipschitz function $f:M \to \RR$, we set $\rcomp{f} = \ucomp{f}\restrict_{\rcomp{M}}$; since $M$ is dense in $\rcomp{M}$, this function (whose range is a subset of $\RR$) is the unique continuous extension of $f$ to $\rcomp{M}$. Similarly to above, we denote by $\rcl{A}$ the closure of a set $A \subset \rcomp{M}$ in $\rcomp{M}$.

Let $\tau$ denote the subspace topology on $\rcomp{M}$. We make some further observations which lead us to an extension result. Define $\delta: \rcomp{M} \to \Lip_0(M)^*$ by $\duality{f,\delta(\xi)} = \rcomp{f}(\xi)$, $f \in \Lip_0(M)$. This map naturally extends the standard isometric embedding $\delta$ of $M$ into $\lipfree{M}$ given in Section \ref{subsect_notation_prelims}.

\begin{proposition}\label{prop_realcompact_embedding}
The map $\delta$ is a $\tau$-$w^*$ homeomorphism onto $\delta(\rcomp{M})$, which is $w^*$-closed in $\Lip_0(M)^*$.
\end{proposition}

\begin{proof}
 It is evident that $\delta$ is continuous. It is injective because points of $\rcomp{M}$ are separated by Lipschitz functions. To see that it is an open map, we first define for $n \in \NN$ the subsets $U_n = (\rcomp{\rho})^{-1}([0,n))$ and $K_n = (\rcomp{\rho})^{-1}([0,n])$ of $\rcomp{M}$, and $V_n=\set{\phi \in \Lip_0(M)^* \,:\, \duality{\rho,\phi} < n}$, where $\rho$ is as in \eqref{eq:rho}. We observe that $U_n$ and $K_n$ are open and compact in $\rcomp{M}$, respectively, and that $\delta(U_n) = V_n \cap \delta(\rcomp{M})$ is $w^*$-open in $\delta(\rcomp{M})$. Given open $U \subset \rcomp{M}$, $\delta(U \cap U_n)$ is a $w^*$-open subset of $\delta(K_n)$ because $\delta\restrict_{K_n}$ is a homeomorphism onto $\delta(K_n)$ by compactness. As $\delta(U_n)$ is $w^*$-open in $\delta(\rcomp{M})$, it follows that $\delta(U \cap U_n)$ is also $w^*$-open in $\delta(\rcomp{M})$. As this holds for all $n \in \NN$, we conclude that the same holds for $\delta(U)$. This proves the first part.
 
 To see that $\delta(\rcomp{M})$ is $w^*$-closed in $\Lip_0(M)^*$, let $\phi \in \cl{\delta(\rcomp{M})}^{w^*}$ and choose $n>\duality{\rho,\phi}$. Then there exists a net $(\xi_i) \subset K_n$ such that $\delta(\xi_i) \to \phi$ in the $w^*$-topology; given a $\tau$-accumulation point $\xi \in K_n$ of $(\xi_i)$, it follows that $\phi = \delta(\xi) \in \delta(\rcomp{M})$ by continuity of $\delta$.
\end{proof}

Next, we consider a natural metric on $\rcomp{M}$. The above embedding allows us to easily define a natural metric extension $\bar{d}$ of $d$ to $\rcomp{M}$; given $\xi,\eta \in \rcomp{M}$, simply set
\begin{equation}
\bar{d}(\xi,\eta) = \norm{\delta(\xi)-\delta(\eta)}_{\Lip_0(M)^*} .
\end{equation}
Given the maximality of $\rho$ in $B_{\Lip_0(M)}$, we have $\bar{d}(\xi,0) = \|\delta(\xi)\| = \rcomp{\rho}(\xi)$ for all $\xi \in \rcomp{M}$. We collect some more observations about $\rcomp{M}$ in the next proposition.

\begin{proposition}\label{prop_rcomp_observations}~
\begin{enumerate}[label={\upshape{(\roman*)}}]
\item The metric $\bar{d}$ is $\tau$-lower semicontinuous and finer than $\tau$;
 \item\label{bar-d-balls_compact} the $\bar{d}$-closed balls $B_{\bar{d}}(\xi,r):= \set{\eta \in \rcomp{M}\,:\,\bar{d}(\xi,\eta) \leq r}$, $r>0$, are $\tau$-compact;
 \item\label{compact_bounded} $\tau$-compact subsets of $\rcomp{M}$ are $\bar{d}$-bounded;
 \item the metric space $(\rcomp{M},\bar{d})$ is complete;
 \item\label{rcomp_normal} $\rcomp{M}$ is normal;
 \item $M$ is $\bar{d}$-closed in $\rcomp{M}$.
\end{enumerate}
\end{proposition}

Statements (i) to (v) follow immediately from Proposition \ref{prop_realcompact_embedding} and the corresponding properties of norm and weak$^*$ topologies in dual Banach spaces, and (vi) holds because $M$ is complete and $d=\bar{d}$ on $M$.

A further proposition yields additional information about the continuity properties of $\bar{d}$ and its associated distance maps, together with a quantitative version of Proposition \ref{prop_rcomp_observations} \ref{rcomp_normal}.

\begin{proposition}\label{pr:bar-d_cont}~
 \begin{enumerate}[label={\upshape{(\roman*)}}]
  \item\label{bar-d_cont_1} Given a $\tau$-closed non-empty set $A \subset \rcomp{M}$, the function $\bar{d}(\cdot,A)$ is $\tau$-lower semicontinuous.
  \item\label{bar-d_cont_2} Given a non-empty set $A \subset M$, $\bar{d}(\cdot,\rcomp{\cl{A}})=\rcomp{d}(\cdot,A)$, and is therefore $\tau$-continuous.
  \item\label{bar-d_cont_3} Given $\xi \in \rcomp{M}$, $\bar{d}(\cdot,\xi)$ is $\tau$-continuous if and only if $\xi \in M$, i.e.~$\bar{d}$ is separately $\tau$-continuous if and only if the fixed coordinate belongs to $M$.
  \item\label{bar-d_cont_4} Given $\tau$-closed non-empty sets $A,B \subset \rcomp{M}$, 
  \[
   \bar{d}(A,B) = \sup\set{d(U \cap M,V \cap M) \,:\, \text{$U \supset A$, $V \supset B$ are $\tau$-open subsets of $\rcomp{M}$}}.
  \]
 \end{enumerate}
\end{proposition}

\begin{proof}
 To prove \ref{bar-d_cont_1}, first fix a $\tau$-closed non-empty set $A \subset \rcomp{M}$ and let $r > 0$. Define $C=\set{\xi \in \rcomp{M} \,:\, \bar{d}(\xi,A) \leq r}$. We will show that $C$ is $\tau$-closed, which will prove \ref{bar-d_cont_1}. The set $\delta(A)$ is $w^*$-closed in $\Lip_0(M)^*$ by Proposition \ref{prop_realcompact_embedding}, thus $\delta(A) + rB_{\Lip_0(M)^*}$ is also $w^*$-closed because $B_{\Lip_0(M)^*}$ is $w^*$-compact. We claim that $C=\delta^{-1}(\delta(A) + rB_{\Lip_0(M)^*})$. Inclusion $\supset$ is clear by the definition of $\bar{d}$. For the other inclusion, let $\xi \in C$. There exist $\eta_n \in A$, $n \in \NN$, such that $\lim_n \bar{d}(\xi,\eta_n) \leq r$. By Proposition \ref{prop_rcomp_observations} \ref{bar-d-balls_compact}, the $\eta_n$ have a $\tau$-accumulation point $\eta$ which, by lower semicontinuity of $\bar{d}$, satisfies $\bar{d}(\xi,\eta) \leq r$.
 As $A$ is $\tau$-closed we have $\eta \in A$, hence $\delta(\xi) \in \delta(A)+rB_{\Lip_0(M)^*}$.
 
 Part \ref{bar-d_cont_2} follows partially from \ref{bar-d_cont_1}. Let $A \subset M$ be non-empty. Consider the function $f \in B_{\Lip_0(M)}$, given by $f(x)=d(x,A)-d(0,A)$, $x \in M$, and its $\tau$-continuous extension $\rcomp{f}$. It is clear that $\rcomp{f}(\xi)+d(0,A)=\rcomp{d}(\xi,A)$ for $\xi \in \rcomp{M}$. Now we show that $\bar{d}(\xi,\rcl{A})=\rcomp{f}(\xi)+d(0,A)$ for all such $\xi$. Let $\xi \in \rcomp{M}$ and $\eta \in \rcl{A}$. Given a net $(x_i)$ of points in $A$ converging to $\eta$, we have $\rcomp{f}(\eta)=\lim_i f(x_i) = -d(0,A)$, hence
 \[
 \bar{d}(\xi,\eta) \geq \duality{f,\delta(\xi)-\delta(\eta)} = \rcomp{f}(\xi)-\rcomp{f}(\eta) = \rcomp{f}(\xi) + d(0,A).
 \]
 As this holds for all such $\eta$, $\bar{d}(\xi,\rcl{A}) \geq \rcomp{f}(\xi) + d(0,A)$. Now suppose that $(y_i)$ is a net of points in $M$ converging to $\xi$. Then by \ref{bar-d_cont_1} and the fact that $\bar{d}$ extends $d$,
 \[
  \bar{d}(\xi,\rcl{A})\leq \liminf_i \bar{d}(y_i,\rcl{A}) \leq \liminf_i d(y_i,A) = \rcomp{f}(\xi)+d(0,A).
 \]
 
 One implication in \ref{bar-d_cont_3} follows trivially from \ref{bar-d_cont_2}. Now suppose that $\bar{d}(\cdot,\xi)$ is $\tau$-continuous. Let $(x_i)$ be a net of points in $M$ converging to $\xi$. By $\tau$-continuity, $\lim_i\bar{d}(x_i,\xi) = 0$, meaning that $(x_i)$ is a $d$-Cauchy net in $M$; by completeness there exists $x \in M$ such that $\lim_i d(x_i,x)=0$, and we must conclude that $\xi=x \in M$.
 
 To prove \ref{bar-d_cont_4}, let $A,B \subset \rcomp{M}$ be non-empty and $\tau$-closed. Let $U \supset A$, $V \supset B$ be $\tau$-open and define $f \in B_{\Lip_0(M)}$ by $f(x)=d(x,V \cap M)-d(0,V \cap M)$. Given $\xi \in A$ and $\eta \in B$, let $(x_i,y_i)$ be a net in $M\times M$ that converges to $(\xi,\eta)\in\rcomp{M}\times\rcomp{M}$. Then $x_i \in U \cap M$ and $y_i \in V \cap M$ for all $i$ large enough, so
 \[
  \bar{d}(\xi,\eta) \geq \rcomp{f}(\xi)- \rcomp{f}(\eta) = \lim_i f(x_i) - f(y_i) = \lim_i d(x_i,V \cap M) \geq d(U \cap M,V \cap M),
 \]
 giving $\bar{d}(A,B) \geq d(U \cap M,V \cap M)$. Conversely, suppose that $\bar{d}(A,B) > \alpha$. Set $r=\frac{1}{2}(\bar{d}(A,B)+\alpha)$ and define $U'=\set{\xi \in \rcomp{M} \,:\, \bar{d}(\xi,B) > r}$, which includes $A$ by hypothesis and is $\tau$-open by \ref{bar-d_cont_1}. Using Proposition \ref{prop_rcomp_observations} \ref{rcomp_normal}, let $U$ be a $\tau$-open set satisfying $A \subset U \subset \rcl{U} \subset U'$. Then $\bar{d}(\rcl{U},B) \geq \bar{d}(U',B) \geq r > \alpha$. Repeating this process with $B$ and $\rcl{U}$ substituted for $A$ and $B$, respectively, yields a $\tau$-open set $V \supset B$ satisfying $d(U \cap M,V \cap M) \geq \bar{d}(\rcl{U},\rcl{V}) > \alpha$. 
\end{proof}

We remark that, as a consequence of Proposition \ref{pr:bar-d_cont} \ref{bar-d_cont_3}, $\bar{d}$ is not compatible with $\tau$ unless $\rcomp{M}=M$, i.e.~$M$ is proper. Finally, we state an extension result which follows from a theorem of Matou\v skov\'a \cite{Matouskova}.

\begin{proposition}\label{pr:Mat_app}
Let $A \subset \rcomp{M}$ be $\tau$-closed and let $\psi:A \to \RR$ be $\tau$-continuous, bounded and $1$-$\bar{d}$-Lipschitz. Then there exists a bounded 1-Lipschitz function $f:M \to \RR$ such that $\rcomp{f}\restrict_A = \psi$ and $\inf\psi\leq f\leq\sup\psi$.
\end{proposition}

\begin{proof}
By Proposition \ref{prop_realcompact_embedding}, the map $\psi \circ \delta^{-1}$ is a $w^*$-continuous, bounded and $1$-Lipschitz function on $\delta(A)$, which is $w^*$-closed in $\Lip_0(M)^*$. By \cite[Corollary 2.6]{Matouskova}, there exists a $w^*$-continuous, bounded and $1$-Lipschitz function $g$ on $\Lip_0(M)^*$ such that $g\restrict_{\delta(A)}=\psi \circ \delta^{-1}$ and $\inf\psi\leq g\leq\sup\psi$. The function $f = g \circ \delta\restrict_M$ is then $1$-Lipschitz on $M$, and by the uniqueness of continuous extensions of Lipschitz functions to $\rcomp{M}$ we have $\rcomp{f} = g \circ \delta$, giving $\rcomp{f}\restrict_A = g \circ \delta\restrict_A = \psi$. 
\end{proof}

\begin{remark}\label{rm:lipfree_rcomp}
Since $(\rcomp{M},\bar{d})$ is a complete metric space, we may consider its Lipschitz-free space $\lipfree{\rcomp{M},\bar{d}}$, which contains $\lipfree{M,d}$ as a subspace. Given a finitely supported element $\phi\in\lspan\,\delta(\rcomp{M})$, its norms in $\lipfree{\rcomp{M},\bar{d}}$ and in $\Lip_0(M)^*$ agree; this follows from applying Proposition \ref{pr:Mat_app} to the finite set $A=\supp(\phi)\cup\set{0}$, as it shows that any 1-$\bar{d}$-Lipschitz function on $A$ is the restriction of $\rcomp{f}$ for some $f\in B_{\Lip(M)}$. By density, we get the isometric identification
$$
\lipfree{\rcomp{M},\bar{d}} = \cl{\lspan}\,\delta(\rcomp{M}) \subset \Lip_0(M)^* .
$$
On the other hand, $\Lip_0(M,d)$ is isometric to the subspace $\Lip_0(\rcomp{M},\bar{d}) \cap C(\rcomp{M},\tau)$ via the extension map $f \mapsto \rcomp{f}$, which in turn is $1$-complemented in $\Lip_0(\rcomp{M},\bar{d})$ via restriction followed by extension: $f \mapsto \rcomp{(f\restrict_M)}$. It is easy to see that the extension map is not $w^*$-$w^*$-continuous in general, so it is not clear to us whether $\lipfree{M,d}$ is always a complemented subspace of $\lipfree{\rcomp{M},\bar{d}}$.
\end{remark}

At this point, we have two natural extensions of $d$: the metric $\bar{d}$ on $\rcomp{M}$, and the continuous extension of $d:\wt{M}\to\RR$ to $\bwt{M}$, which we will also denote by $d$ and can now take the value $\infty$. The next proposition establishes the relationship between the two extensions via the De Leeuw transform $\Phi^*$. Hereafter, we will fix the set
\begin{equation}\label{eqn_S}
\bwtf=\pp^{-1}(\rcomp{M}\times \rcomp{M}).
\end{equation}
This is an open subset of $\bwt{M}$ which includes $\wt{M}$, so it is also dense. It will serve as the arena for much of what we do in this paper. It is easily checked that $d(\zeta)<\infty$ for $\zeta\in\bwtf$. Indeed, given a net $(x_i,y_i)$ in $\wt{M}$ converging to such $\zeta$, we have
$$
d(\zeta) = \lim_i d(x_i,y_i) \leq \lim_i \rho(x_i)+\rho(y_i) = \rcomp{\rho}(\pp_1(\zeta))+\rcomp{\rho}(\pp_2(\zeta)) < \infty .
$$

\begin{proposition}\label{prop_de_Leeuw_relation}
Let $\zeta \in \bwtf$. Then
\begin{enumerate}[label={\upshape{(\roman*)}}]
 \item\label{delta_diff} $d(\zeta)\Phi^*\delta_\zeta = \delta(\pp_1(\zeta)) - \delta(\pp_2(\zeta))$;
 \item\label{phi_delta_norm} $d(\zeta)\|\Phi^*\delta_\zeta\| = \bar{d}(\pp(\zeta))$;
 \item\label{d_bar_less_than_d} $\bar{d}(\pp(\zeta)) \leq d(\zeta)$, and we have equality if and only if $\|\Phi^*\delta_\zeta\|=1$ or $d(\zeta)=0$.
\end{enumerate}
\end{proposition}

\begin{proof}
 To prove \ref{delta_diff}, let $(x_i,y_i) \subset \wt{M}$ be a net converging to $\zeta$ in $\bwt{M}$. Then by continuity of $\pp$ and the fact that $\pp(x,y)=(x,y)$ whenever $(x,y) \in \wt{M}$, we have $(x_i,y_i) \to \pp(\zeta)$ in $\rcomp{M} \times \rcomp{M}$. Given $f \in \Lip_0(M)$,
 \[
  d(x_i,y_i)\cdot\duality{f,\Phi^*\delta_{(x_i,y_i)}} = f(x_i) - f(y_i)
 \]
for all $i$, so taking limits of both sides and using the fact that $d(\zeta)<\infty$ yields
\[
 d(\zeta)\cdot\duality{f,\Phi^*\delta_\zeta} = \rcomp{f}(\pp_1(\zeta)) - \rcomp{f}(\pp_2(\zeta)) = \duality{f,\delta(\pp_1(\zeta)) - \delta(\pp_2(\zeta))},
\]
which gives \ref{delta_diff}. Then \ref{phi_delta_norm} and \ref{d_bar_less_than_d} follow trivially from \ref{delta_diff} and \ref{phi_delta_norm}, respectively.
\end{proof}

We conclude this subsection by introducing extended metric segments in $\rcomp{M}$. Recall that the \emph{metric segment} between $x,y\in M$ is the set of $p \in M$ that lie ``between'' $x$ and $y$, in the sense that the triangle inequality becomes an equality:
\[
[x,y] = \set{p\in M \,:\, d(x,p)+d(p,y)=d(x,y)}.
\]
We replicate this in $(\rcomp{M},\bar{d})$: given $\xi,\eta \in \rcomp{M}$, their \emph{extended metric segment} is the set
\[
\rcomp{[\xi,\eta]} \;=\; \set{\alpha \in \rcomp{M}\,:\, \bar{d}(\xi,\alpha) + \bar{d}(\alpha,\eta) = \bar{d}(\xi,\eta)}.
\]
Evidently this extends the notion of metric segments in $M$, with $[x,y]=\rcomp{[x,y]} \cap M$ whenever $x,y \in M$. The extended metric segments in $\rcomp{M}$ are $\tau$-compact. Indeed, given a net $(\alpha_i) \subset \rcomp{[\xi,\eta]}$ converging to $\alpha \in \rcomp{M}$, by the lower semicontinuity of $\bar{d}$ we have
\[
\bar{d}(\xi,\alpha) + \bar{d}(\alpha,\eta) \leq \liminf_i \bar{d}(\xi,\alpha_i) + \bar{d}(\alpha_i,\eta) = \bar{d}(\xi,\eta),
\]
so $\alpha \in \rcomp{[\xi,\eta]}$. Hence $\rcomp{[\xi,\eta]}$ is $\tau$-closed. That it is $\tau$-compact follows from Proposition \ref{prop_rcomp_observations} \ref{bar-d-balls_compact}.

\subsection{Functionals that avoid infinity}\label{functionals_avoid_infinity}

In this final subsection, we introduce a few special classes of elements of $\Lip_0(M)^*$ and see how they relate to the set $\bwtf$ defined in \eqref{eqn_S}. They arise from a natural decomposition of $\Lip_0(M)^*$ \cite[Section 3.1]{AP_measures}, which is related to the annular decomposition of $\lipfree{M}$ constructed by Kalton \cite[Section 4]{Kalton}. To define them, we need to introduce weighting operators and apply them to a sequence of Lipschitz maps.

Given $h\in\Lip(M)$, we can formally define a weighting operator $W_h$ on $\Lip_0(M)$ by
$$
W_h(f)=f\cdot h ,\quad f \in \Lip_0(M).
$$
If $h$ has bounded support, then this really defines a bounded operator $W_h:\Lip_0(M)\to\Lip_0(M)$ which is moreover $w^*$-$w^*$-continuous, hence $W_h^*$ maps $\lipfree{M}$ into $\lipfree{M}$ \cite[Lemma 2.3]{APPP_2020}. We will consider a canonical sequence of weighting functions $h$ with increasing supports. For $n\in\ZZ$, define $\daleth_n:M \to \RR$ by
\begin{equation}\label{eq:daleth}
\daleth_n(x) = \begin{cases}
1 &\text{if } \rho(x) \leq 2^n, \\
2-2^{-n}d(x,0) &\text{if } 2^n \leq \rho(x) \leq 2^{n+1}, \\
0 &\text{if } 2^{n+1} \leq \rho(x) .
\end{cases}
\end{equation}
We have $\norm{W_{\daleth_n}}\leq 3$ for all $n$ and, given $\phi\in\Lip_0(M)^*$, it is not hard to check that the sequence $W^*_{\daleth_n}(\phi)$ is always norm-convergent as $n\to +\infty$ (see \cite[Section 3.1]{AP_measures} where $W_{\daleth_n}$ are denoted by $\mathcal{T}_{H_n}$).

\begin{definition}[{\cite[Definition 3.1]{AP_measures}}]
Let $\phi \in \Lip_0(M)^*$. We say that $\phi$ \emph{avoids infinity} if
$$
\lim_{n\to +\infty} W^*_{\daleth_n}(\phi) = \phi ,
$$
and we say that $\phi$ is \emph{concentrated at infinity} if $W^*_{\daleth_n}(\phi)=0$ for all $n$.
\end{definition}

Let $\Lambda_0(M)$ and $\Lambda_\infty(M)$ denote the subspaces of functionals on $\Lip_0(M)$ that avoid infinity and are concentrated at infinity, respectively. We have $\Lambda_0(M)=\Lip_0(M)^*$ if and only if $M$ is bounded. By \cite[Corollary 3.4]{AP_measures} we have the decomposition
\begin{equation}\label{eq:decomp_bidual}
\Lip_0(M)^* = \Lambda_0(M) \oplus_1 \Lambda_\infty(M)
\end{equation}
where the L-projection onto $\Lambda_0(M)$ is given by $\phi\mapsto\lim_n W^*_{\daleth_n}(\phi)$. Recall that a projection $P:X\to Y\subset X$ is called an \emph{L-projection} if $\norm{x}=\norm{Px}+\norm{x-Px}$ for all $x\in X$.

While functionals that avoid infinity (which include $\lipfree{M}$) have particularly good behaviour, the definition of the concept and the description of the associated projection depend on technical details like the functions $\daleth_n$, which are by no means unique. Functionals concentrated at infinity are precisely those that vanish on all $f\in\Lip_0(M)$ with bounded support \cite[Proposition 3.2]{AP_measures}, but there is no similar characterisation for functionals that avoid infinity.

We will now obtain a much more satisfying characterisation of these concepts, based on De Leeuw representations and the set $\bwtf$. The key fact is provided by the next proposition.

\begin{proposition}\label{prop_proj_avoid_infty}
Let $\phi\in\Lip_0(M)^*$. If $\mu$ is a De Leeuw representation of $\phi$, then $\mu\restrict_\bwtf$ is a De Leeuw representation of $\lim_n W^*_{\daleth_n}(\phi)$.
\end{proposition}

Before proving the result, we state and prove a lemma that is modelled closely on \cite[Lemma 6]{Aliaga}.

\begin{lemma}\label{lm:phi_bounded_support}
If $f\in\Lip_0(M)$ has bounded support then $\Phi f$ vanishes outside of $\bwtf$.
\end{lemma}

\begin{proof}
Fix $\zeta\in\bwt{M}\setminus\bwtf$ and let $(x_i,y_i)$ be a net in $\wt{M}$ converging to $\zeta$. If $d(\zeta)=\infty$ then the bound $\abs{\Phi f(x_i,y_i)}\leq 2\norm{f}_\infty/d(x_i,y_i)$ implies $\Phi f(\zeta)=0$. Otherwise both coordinates of $\zeta$ are outside of $\rcomp{M}$. Indeed, $\ucomp{\rho}(\pp_2(\zeta)) \leq \ucomp{\rho}(\pp_1(\zeta))+d(\zeta)$, thus if $\pp_1(\zeta)$ belongs to $\rcomp{M}$ then so does $\pp_2(\zeta)$, giving $\zeta \in \bwtf$, which isn't the case; likewise if $\pp_2(\zeta) \in \rcomp{M}$. But then $f(x_i)=f(y_i)=0$ eventually and thus $\Phi f(\zeta)=0$.
\end{proof}

\begin{proof}[Proof of Proposition \ref{prop_proj_avoid_infty}]
The argument is a generalization of \cite[Proposition 7]{Aliaga}. Let $\phi$ and $\mu$ be as above, and fix $f\in\Lip_0(M)$. First, we claim that
$$
\lim_{n\to\infty} \Phi(W_{\daleth_n}(f))(\zeta) = \begin{cases}
\Phi f(\zeta) &\text{if $\zeta\in\bwtf$,} \\
0 &\text{if $\zeta\in\bwt{M}\setminus\bwtf$.}
\end{cases}
$$
The case $\zeta\notin\bwtf$ is obvious from Lemma \ref{lm:phi_bounded_support}. To see the other case, fix $\zeta\in\bwtf$. If $n$ is so large that $2^n > \max\set{\rcomp{\rho}(\pp_1(\zeta)),\rcomp{\rho}(\pp_2(\zeta))}$, then $\ucomp{\daleth_n}=1$ on neighborhoods of $\pp_1(\zeta)$ and $\pp_2(\zeta)$, therefore $\Phi(f\cdot\daleth_n)(x_i,y_i)=\Phi f(x_i,y_i)$ eventually, and taking limits with respect to $i$ yields $\Phi(f\cdot\daleth_n)(\zeta)=\Phi f(\zeta)$.

Because $\lipnorm{W_{\daleth_n}(f)}\leq 3\lipnorm{f}$ for all $n$, Lebesgue's dominated convergence theorem yields
$$
\lim_{n\to\infty} \duality{f,W^*_{\daleth_n}(\phi)} = \lim_{n\to\infty}\int_{\bwt{M}}\Phi(W_{\daleth_n}(f))\,d\mu = \int_{\bwt{M}}\lim_{n\to\infty}\Phi(W_{\daleth_n}(f))\,d\mu = \int_\bwtf \Phi f \,d\mu
$$
for every $f\in\Lip_0(M)$, that is $W^*_{\daleth_n}(\phi)\wsconv\Phi^*(\mu\restrict_\bwtf)$. The result now follows by recalling that $(W^*_{\daleth_n}(\phi))$ converges in norm.
\end{proof}

The characterisations we sought are now immediate.

\begin{corollary}\label{cor_avoid_infinity}
Let $\phi\in \Lip_0(M)^*$. Then the following are equivalent:
\begin{enumerate}[label={\upshape{(\roman*)}}]
\item $\phi$ avoids infinity,
\item every optimal De Leeuw representation of $\phi$ is concentrated on $\bwtf$,
\item $\phi$ has a De Leeuw representation that is concentrated on $\bwtf$.
\end{enumerate}
\end{corollary}

\begin{corollary}\label{cor_conc_infinity}
Let $\phi\in \Lip_0(M)^*$. Then the following are equivalent:
\begin{enumerate}[label={\upshape{(\roman*)}}]
\item $\phi$ is concentrated at infinity,
\item every optimal De Leeuw representation of $\phi$ is concentrated on $\bwt{M}\setminus\bwtf$,
\item $\phi$ has a De Leeuw representation that is concentrated on $\bwt{M}\setminus\bwtf$,
\item $\duality{f,\phi}=0$ for all $f\in\Lip_0(M)$ with bounded support.
\end{enumerate}
\end{corollary}

\begin{proof}
Corollary \ref{cor_avoid_infinity} follows immediately from Proposition \ref{prop_proj_avoid_infty}. In Corollary \ref{cor_conc_infinity}, (i)$\Rightarrow$(ii)$\Rightarrow$(iii) are obtained similarly, implication (iii)$\Rightarrow$(iv) follows from Lemma \ref{lm:phi_bounded_support}, and the equivalence (i)$\Leftrightarrow$(iv) is \cite[Proposition 3.2(a)]{AP_measures}.
\end{proof}

\begin{remark}\label{remark:decomp_bidual}
As another consequence, we also obtain the decomposition \eqref{eq:decomp_bidual}, with the projection onto $\Lambda_0(M)$ given by $\Phi^*\mu \mapsto \Phi^*(\mu\restrict_\bwtf)$. Indeed, Proposition \ref{prop_proj_avoid_infty} ensures that this operator is well defined (and agrees with the one given in \cite{AP_measures}), and Corollaries \ref{cor_avoid_infinity} and \ref{cor_conc_infinity} ensure that it is a projection onto $\Lambda_0(M)$ with kernel $\Lambda_\infty(M)$. To see that it is an L-projection just notice that $\norm{\mu}=\norm{\mu\restrict_\bwtf}+\norm{\mu-\mu\restrict_\bwtf}$ for any $\mu\in\meas{\bwt{M}}$ and apply it to optimal representations. See Proposition \ref{pr:p1_l_embedded} below for a similarly defined L-projection onto $\lipfree{M}$ that is valid under certain hypotheses.
\end{remark}

We shall occasionally consider another concept that is dual to the avoidance of infinity. In this case, we employ the weighting functions $1-\daleth_{-n}$ instead of $\daleth_n$. These have increasing unbounded supports, but still define valid weighting operators as $W_{1-\daleth_n}=I-W_{\daleth_n}$.

\begin{definition}[{\cite[Definition 3.1]{AP_measures}}]
Let $\phi \in \Lip_0(M)^*$. We say that $\phi$ \emph{avoids $0$} if
$$
\lim_{n\to +\infty} W^*_{1-\daleth_{-n}}(\phi) = \phi .
$$
\end{definition}

Both concepts can be combined, and it is easy to see that $\phi$ avoids both $0$ and infinity if and only if $\lim_{n\to\infty} W^*_{\Pi_n}(\phi) = \phi$, where
\begin{equation}\label{eq:Pi}
\Pi_n = \daleth_n - \daleth_{-n} = \daleth_n \cdot (1 - \daleth_{-n})
\end{equation}
for $n\in\NN$. We also get a corresponding characterisation in terms of De Leeuw representations.

\begin{proposition}\label{pr:avoid_0}
Let $\phi\in\Lip_0(M)^*$. Then the following are equivalent:
\begin{enumerate}[label={\upshape{(\roman*)}}]
\item $\phi$ avoids 0,
\item every optimal De Leeuw representation $\mu$ of $\phi$ satisfies $\mu(\pp^{-1}(0,0))=0$,
\item $\phi$ has a De Leeuw representation $\mu$ such that $\abs{\mu}(\pp^{-1}(0,0))=0$.
\end{enumerate}
\end{proposition}

\begin{proof}
We only sketch the proof, which is analogous to that of Proposition \ref{prop_proj_avoid_infty}, using functions $1-\daleth_{-n}$ in place of $\daleth_n$. For every $f\in\Lip_0(M)$ we have
$$
\lim_{n\to\infty} \Phi(W_{1-\daleth_{-n}}(f))(\zeta) = \begin{cases}
\Phi f(\zeta) &\text{if $\pp(\zeta)\neq (0,0)$,} \\
0 &\text{if $\pp(\zeta)=(0,0)$,}
\end{cases}
$$
and the dominated convergence theorem implies that
$$
\Phi^*(\mu\restrict_{\bwt{M}\setminus\pp^{-1}(0,0)})=\lim_{n\to\infty} W_{1-\daleth_{-n}}^*(\Phi^*\mu)
$$
for any $\mu\in\meas{\bwt{M}}$. The equivalence of (i)-(iii) now follows easily as in Corollary \ref{cor_avoid_infinity}.
\end{proof}

\section{Optimal points, vanishing points and those that live in between}\label{sect_points_in_bwtM}

\subsection{Optimal points}\label{subsect_optimal}

It turns out that, in practice, when focussing on optimal De Leeuw representations, we need only consider those $\zeta \in \bwt{M}$ for which $\|\Phi^*\delta_\zeta\|=1$. We shall call such elements \emph{optimal points}. Let $\delta:\bwt{M}\to\meas{\bwt{M}}$ denote the Dirac evaluation map (not to be confused with the mapping $\delta:\rcomp{M}\to\Lip_0(M)^*$ introduced in Section \ref{sect_top_and_metric_prelims}). This is a homeomorphic embedding into $\meas{\bwt{M}}$ with the $w^*$-topology. As the map $\Phi^* \circ \delta:\bwt{M}\to(B_{\Lip_0(M)^*},w^*)$ is continuous and the dual norm is $w^*$-lower semicontinuous, the sets
\[
F_n :=\set{\zeta \in \bwt{M}\,:\,\|\Phi^*\delta_\zeta\| \leq 1-2^{-n}}, \qquad n \in \NN
\]
are closed. Moreover, we have $F_n \cap \wt{M} = \varnothing$ for all $n$, whence the set
\[
\opt:=\bwt{M}\setminus \bigcup_{n=1}^\infty F_n = \set{\zeta \in \bwt{M}\,:\,\|\Phi^*\delta_\zeta\| = 1}, 
\]
of optimal points is a dense $G_\delta$ that includes $\wt{M}$. For the purposes of studying optimal representations, it means that we can disregard non-optimal points.

\begin{proposition}\label{prop_conc_on_O}
If $\mu \in \opr{\bwt{M}}$ then $\mu$ is concentrated on $\opt$. 
\end{proposition}

\begin{proof}
Let $\mu \in \meas{\bwt{M}}$ satisfy $\|\mu\|>\|\mu\restrict_\opt\|$. Then there exists $n \in \NN$ such that $\|\mu\restrict_{F_n}\| > 0$. By the Krein-Milman Theorem, it follows that $\|\Phi^*\mu\restrict_{F_n}\| \leq (1-2^{-n})\|\mu\restrict_{F_n}\|$. Therefore
 \begin{align*}
  \|\Phi^*\mu\| \leq \|\Phi^*\mu\restrict_{F_n}\| + \|\Phi^*\mu\restrict_{\bwt{M}\setminus F_n}\| &\leq (1-2^{-n})\|\mu\restrict_{F_n}\| + \|\mu\restrict_{\bwt{M}\setminus F_n}\|\\
  &= \|\mu\| - 2^{-n}\|\mu\restrict_{F_n}\| < \|\mu\|.
 \end{align*}
 Consequently $\mu \notin \opr{\bwt{M}}$.
\end{proof}

\begin{remark}\label{rem_bar_d_d_equality}
Another consequence of Propositions \ref{prop_conc_on_O} and \ref{prop_de_Leeuw_relation} is that if $\mu \in \opr{\bwt{M}}$ is concentrated on $\bwtf$, then it is concentrated on $\opt \cap \bwtf$, on which we have $\bar{d} \circ \pp \equiv d$. Thus, when considering optimal representations of functionals avoiding infinity, we may assume that ``$d$ and $\bar{d}$ coincide''.
\end{remark}

As an aside, we deduce that if $\opt=\wt{M}$ then $M$ must be finite. Indeed, if $\opt=\wt{M}$ then with $\mu$ as above we have $\Phi^*\mu\in\lipfree{M}$ by Proposition \ref{pr:wt_bochner}, whence $\lipfree{M}$ is reflexive. This is true only if $M$ is finite \cite[Corollary 3.46]{Weaver2}.

Optimal points are plentiful in more subtle ways, besides the fact that they form a dense $G_\delta$. Not only do they contain $\wt{M}$, but we also have $\zeta \in \opt$ whenever $\zeta \in \bwtf$ and only one of the coordinates of $\zeta$ belongs to $M$. Indeed, suppose $\pp(\zeta)=(x,\xi)$ with $x\in M$, $\xi\in\rcomp{M}\setminus M$, and let $(x_i,y_i)$ be a net in $\wt{M}$ converging to $\zeta$. Then
$$
|d(x_i,y_i)-\bar{d}(x,\xi)| \leq |d(x_i,y_i)-d(x,y_i)| +|d(x,y_i)-\bar{d}(x,\xi)| \leq d(x,x_i)+|d(x,y_i)-\bar{d}(x,\xi)|
$$
converges to $0$ by Proposition \ref{pr:bar-d_cont} \ref{bar-d_cont_3}, hence $\zeta \in \opt$ by Proposition \ref{prop_de_Leeuw_relation} \ref{d_bar_less_than_d}.

While we cannot guarantee that all points in $\bwtf$ are optimal, we can at least find optimal points with any prescribed different coordinates. The situation when both coordinates are equal will be handled in Proposition \ref{pr:optimals_everywhere_2}. First we need a lemma.

\begin{lemma}
Fix $(\xi,\eta) \in \pp(\bwt{M})$. Then
 \begin{equation}\label{eqn:optimals_everywhere_1}
  d(\pp^{-1}(\xi,\eta)) \;=\; \bigcap_{U,V} \cl{d\pare{(U \times V) \cap \wt{M}}},
 \end{equation}
where $U \ni \xi$ and $V \ni \eta$ are open subsets of $\ucomp{M}$. In particular, we have
 \begin{equation}\label{eqn:optimals_everywhere_2}
 \min d(\pp^{-1}(\xi,\eta)) = \sup_{U,V} \inf d\pare{(U \times V) \cap \wt{M}}.
\end{equation}
\end{lemma}

\begin{proof}
Let $t \in \bigcap_{U,V} \cl{d((U \times V) \cap \wt{M})}$. Suppose that $t<\infty$; the argument for $t=\infty$ is similar. Then given open $U \ni \xi$ and $V \ni \eta$, and $n \in \NN$, the set
 \[
  \set{(x,y) \in (U \times V) \cap \wt{M}\,:\,|d(x,y) - t| < 2^{-n}}
 \]
is non-empty. Therefore by compactness there exists 
\[
 \zeta \in \bigcap_{U,V,n} \cl{\set{(x,y) \in (U \times V) \cap \wt{M}\,:\,|d(x,y) - t| < 2^{-n}}}^{\bwt{M}} .
\]
As in the proof of Proposition \ref{prop_p_range} we have $\pp(\zeta)=(\xi,\eta)$. It is straightforward to check that $d(\zeta)=t$.

Conversely, let $\pp(\zeta)=(\xi,\eta)$ and assume $d(\zeta)<\infty$; again, the case $d(\zeta)=\infty$ admits a similar treatment. Let $U \ni \xi$ and $V \ni \eta$ be open in $\ucomp{M}$, and let $n \in \NN$. Take open $W \ni \zeta$ such that $|d(x,y) - d(\zeta)| < 2^{-n}$ whenever $(x,y) \in W \cap \wt{M}$. By density of $\wt{M}$, there exists $(x,y) \in W \cap \pp^{-1}(U \times V) \cap \wt{M}$. Hence $|d(x,y) - d(\zeta)| < 2^{-n}$ and $(x,y) \in (U \times V) \cap \wt{M}$. As we can find such a pair $(x,y)$ for each $n \in \NN$, we obtain
\[
d(\zeta) \in \cl{d\pare{(U \times V) \cap \wt{M}}}.
\]
Thus \eqref{eqn:optimals_everywhere_1} has been proved, and from this \eqref{eqn:optimals_everywhere_2} follows immediately.
\end{proof}

\begin{proposition}\label{pr:optimals_everywhere}
If $\xi, \eta \in \rcomp{M}$ are distinct then $\opt \cap \pp^{-1}(\xi,\eta)$ is compact and non-empty.
\end{proposition}

\begin{proof}
By Proposition \ref{prop_de_Leeuw_relation} \ref{d_bar_less_than_d}, every $\zeta \in \pp^{-1}(\xi,\eta)$ satisfies $d(\zeta) \geq \bar{d}(\xi,\eta)>0$. Now define $t=\min d(\pp^{-1}(\xi,\eta))$. By \eqref{eqn:optimals_everywhere_2} and Proposition \ref{pr:bar-d_cont} \ref{bar-d_cont_4}, $t = \bar{d}(\xi,\eta)$. Consequently, again by Proposition \ref{prop_de_Leeuw_relation} \ref{d_bar_less_than_d}, $\opt \cap \pp^{-1}(\xi,\eta) = d^{-1}(t) \cap \pp^{-1}(\xi,\eta)$ is compact and non-empty.
\end{proof}

The next proposition is a quantitative enhancement of \cite[Lemma 4.2]{GRZ}. This proposition will allow us to find examples of metric spaces $M$ for which $\opt=\bwt{M}$ or not. It will also be used in the next subsection, where we will see more precisely how it relates to the aforementioned lemma. To state the result, first we define a modulus
\[
 \alpha(M) = \inf\set{\|\Phi^*\delta_\zeta\| \,:\, \zeta \in \bwt{M}} \in [0,1].
\]
This infimum is attained by compactness and the lower semicontinuity of the map $\zeta \mapsto \norm{\Phi^*\delta_\zeta}$. Clearly $\opt=\bwt{M}$ if and only if $\alpha(M)=1$. We are able to express $\alpha(M)$ in terms of the distortions of bi-Lipschitz embeddings of $M$ into $\ell_\infty^n$, $n \in\NN$, (should there be any) in the following sense. Since $M$ does not admit such bi-Lipschitz embeddings in general, not every point will be optimal for all $M$. We will consider non-optimal points in more depth later in the section.

\begin{proposition}\label{pr:biLip_embed}
 We have
 \[
  \alpha(M) = \sup_F \inf_{(x,y)\in\wt{M}} \max_{f\in F} \abs{\Phi f(x,y)}
 \]
 where $F$ runs over the non-empty finite subsets of $B_{\Lip_0(M)}$, and
\[
 \frac{1}{\alpha(M)} = \inf\set{ \Lip(h)\Lip(h^{-1}) \,:\, n \in \NN \text{ and } h:M \to \ell_\infty^n \text{ is a bi-Lipschitz embedding}},
\]
with the understanding that $\alpha(M)=0$ if and only if there are no such embeddings for any $n$.
\end{proposition}

\begin{proof}
 First we introduce some notation. Given a non-empty finite set $F \subset B_{\Lip_0(M)}$, define $\Psi_F \in B_{C(\bwt{M})}$ by 
 $\Psi_F(\zeta)= \max\set{|(\Phi f)(\zeta)| \,:\, f \in F}$, $\zeta \in \bwt{M}$.
 Then set $\alpha_F = \min \Psi_F$, which equals $\inf \Psi_F(\wt{M})$ by the continuity of $\Psi_F$ and the density of $\wt{M}$. Finally, set $\alpha=\sup_F \alpha_F$. Our first task is thus to prove that $\alpha(M)=\alpha$.
 
 To this end, let $\zeta \in \bwt{M}$ and $\varepsilon>0$, and pick non-empty finite $F \subset B_{\Lip_0(M)}$. Clearly $\norm{\Phi^*\delta_\zeta} \geq \Psi_F(\zeta)$, so by density and continuity there exists $(x,y) \in \wt{M}$ such that $\norm{\Phi^*\delta_\zeta} \geq \Psi_F(x,y)-\varepsilon \geq \alpha_F - \varepsilon$. As this holds for all such $F$, it follows that $\norm{\Phi^*\delta_\zeta} \geq \alpha - \varepsilon$. Consequently $\alpha(M) \geq \alpha$. Conversely, given non-empty finite $F \subset B_{\Lip_0(M)}$ and $n \in \NN$, the compact sets $\Psi_F^{-1}([-1,\alpha+2^{-n}])$ are non-empty, so by compactness there exists
\[
 \zeta \in \bigcap_{F,n} \Psi_F^{-1}([-1,\alpha+2^{-n}]) .
\]
We claim that $\norm{\Phi^*\delta_\zeta} \leq \alpha$. Given $n \in \NN$ and $f \in B_{\Lip_0(M)}$, $\abs{\duality{f,\Phi^*\delta_\zeta}} = \abs{\Phi f(\zeta)} = \Psi_{\set{f}}(\zeta) \leq \alpha + 2^{-n}$. Therefore $\norm{\Phi^*\delta_\zeta} \leq \alpha$, as required.

To establish the second equality, assume that $n \in \NN$ and $h:M \to \ell_\infty^n$ is a bi-Lipschitz embedding. If we define $k=h/\Lip(h)$, then evidently
\[
 \frac{1}{\Lip(h)\Lip(h^{-1})}d(x,y) \leq \norm{k(x)-k(y)}_\infty \leq d(x,y), \quad x,y \in M.
\]
Writing $k(x)=(k_1(x),\ldots,k_n(x))$, $x \in M$, we see that each $k_i$ belongs to $B_{\Lip_0(M)}$, and if we set $F=\set{k_1,\ldots,k_n}$, then $(x,y) \in \wt{M}$ implies
\[
 \Psi_F(x,y) = \frac{\norm{k(x)-k(y)}_\infty}{d(x,y)} \geq \frac{1}{\Lip(h)\Lip(h^{-1})},
\]
whence $\alpha(M)>0$ and $1/\alpha(M) \leq \Lip(h)\Lip(h^{-1})$. Conversely, suppose $\alpha(M)>0$ and let $\varepsilon \in (0,1)$. Set $\delta=\frac{1}{2}\alpha(M)^2\varepsilon$ and pick a non-empty finite set $F \subset B_{\Lip_0(M)}$ satisfying $\alpha_F \geq \alpha(M)-\delta > 0$ (recall that $\alpha(M) \leq 1$). Write $F=\set{h_1,\ldots,h_n}$ and define $h:M \to \ell_\infty^n$ by $h(x)=(h_1(x),\ldots,h_n(x))$, $x \in M$. Clearly $\Lip(h) \leq 1$, and for $(x,y) \in \wt{M}$
\[
 \frac{\norm{h(x)-h(y)}_\infty}{d(x,y)} = \Psi_F(x,y) \geq \alpha_F \geq \alpha(M)-\delta.
\]
Hence $h$ is bi-Lipschitz and $\Lip(h^{-1}) \leq 1/(\alpha(M)-\delta)$. Consequently,
\[
 \Lip(h)\Lip(h^{-1}) \leq \frac{1}{\alpha(M)-\delta} < \frac{1}{\alpha(M)} + \varepsilon. \qedhere
\]
\end{proof}

\begin{example}\label{ex:all_optimal}
 Let $M=X$ be a finite-dimensional Banach space. Then $\opt=\bwt{M}$. Indeed, given $\varepsilon > 0$, let $F=\set{f_1,\ldots,f_n} \subset S_{X^*}$ be a finite $\varepsilon$-net and define the linear embedding $h:M \to \ell_\infty^n$ by $h(x)=(f_1(x),\ldots,f_n(x))$; clearly $\Lip(h)\Lip(h^{-1}) \leq 1/(1-\varepsilon)$, so from Proposition \ref{pr:biLip_embed}, $\alpha(M) \geq 1-\varepsilon$. 
\end{example}

\subsection{Vanishing points}\label{subsec_vanishing}

On the other extreme we have those $\zeta \in \bwt{M}$ for which $\Phi^*\delta_\zeta = 0$. We will call these elements \emph{vanishing points}. These are extreme examples of the failure of the converse of Proposition \ref{pr:wt_bochner} (a less extreme example is given in \cite[Example 5]{Aliaga}). To see that they can exist, let $M$ be infinite, bounded and uniformly discrete and let $\xi \in \rcomp{M}\setminus M$. By Proposition \ref{prop_p_range} there exists $\zeta \in \bwt{M}$ such that $\pp(\zeta)=(\xi,\xi)$. As $M$ is uniformly discrete, we have $d(\zeta)>0$. Thus by Proposition \ref{prop_de_Leeuw_relation}, $\Phi^*\delta_\zeta=0$.

The question of whether vanishing points can exist relates to \cite[Theorem 3.43]{Weaver2}. In the statement of this theorem, Weaver makes the assumption that $\zeta\in\bwt{X}$ is not a vanishing point of $X$, and remarks on p.~114 that he doesn't know whether it is necessary to do so, i.e.~whether vanishing points can exist at all. From above, we see that this assumption is indeed necessary. We should point out that the existence of vanishing points had already been established essentially via \cite[Lemma 4.2]{GRZ} by Garc\'ia-Lirola and Rueda Zoca (who in turn credit Lancien for the argument).

Using Proposition \ref{pr:biLip_embed}, we can characterise precisely when vanishing points exist in the next result. Observe that the equivalence of statements \ref{vanishing_points_3} and \ref{vanishing_points_4} below is precisely what is established in \cite[Lemma 4.2]{GRZ}. 

\begin{proposition}[{cf. \cite[Lemma 4.2]{GRZ}}]\label{pr:vanishing_points} The following are equivalent.
\begin{enumerate}[label={\upshape{(\roman*)}}]
 \item\label{vanishing_points_1} There exists a vanishing point in $\bwt{M}$;
 \item\label{vanishing_points_2} $\alpha(M)=0$;
 \item\label{vanishing_points_3} $M$ cannot be bi-Lipschitz embedded into Euclidean space;
 \item\label{vanishing_points_4} $0 \in \cl{\set{m_{xy} \,:\, (x,y) \in \wt{M}}}^{(\lipfree{M},w)}$.
\end{enumerate}
\end{proposition}

\begin{proof}
 The equivalence of \ref{vanishing_points_1} -- \ref{vanishing_points_3} is immediate from the definitions and Proposition \ref{pr:biLip_embed}. As noted above, \ref{vanishing_points_3} $\Leftrightarrow$ \ref{vanishing_points_4} is \cite[Lemma 4.2]{GRZ}; alternatively, we can prove \ref{vanishing_points_1} $\Leftrightarrow$ \ref{vanishing_points_4} directly. Indeed, recall that $\delta:\bwt{M} \to (\meas{\bwt{M}},w^*)$ is a homeomorphic embedding and $\dual{\Phi}(\delta(\wt{M}))=\{m_{xy} \,:\, (x,y) \in \wt{M}\}$ is the set of molecules on $M$. By compactness and the $w^*$-$w^*$-continuity of $\dual{\Phi}$ we have
\begin{align}\label{eqn_w_closure}
\dual{\Phi}(\delta(\bwt{M}))\cap\lipfree{M} &= \dual{\Phi}\pare{\wscl{\delta(\wt{M})}}\cap\lipfree{M} \\
&= \cl{\dual{\Phi}(\delta(\wt{M}))}^{(\dual{\Lip_0(M)},w^*)}\cap\lipfree{M} = \cl{\dual{\Phi}(\delta(\wt{M}))}^{(\lipfree{M},w)}. \nonumber
\end{align}
Thus the existence of vanishing points is equivalent to the statement that $0$ belongs to the weak closure of the set of molecules in $\lipfree{M}$.
\end{proof}

Incidentally, the identity \eqref{eqn_w_closure} above, together with \cite[Theorem 3.43]{Weaver2}, tell us that $0$ is the only possible point in the weak closure of the set of molecules in $\lipfree{M}$ that doesn't belong to the set itself. This is a known fact, see \cite[Proposition 2.9]{GPPR}.

It is possible to find really extreme examples of vanishing points. Consider the ``reflection map'' $\rr:\bwt{M}\to\bwt{M}$ defined as the continuous extension of the self-map $(x,y)\mapsto (y,x)$ on $\wt{M}$. It is then easy to see that $\dual{\Phi}(\mu+\rr_\sharp\mu)=0$ for any $\mu\in\meas{\bwt{M}}$ \cite[Lemma 2]{Aliaga}. In particular, $\dual{\Phi}\delta_\zeta=0$ if $\zeta$ is a fixed point of $\rr$. In certain metric spaces these fixed points exist.

\begin{example}\label{ex_r_fixed_points}
Suppose that, given a non-empty finite set $F \subset C(\bwt{M})$, there exists $(x_F,y_F) \in \wt{M}$ such that $f(x_F,y_F)=f(y_F,x_F)$ for all $f \in F$. Then any accumulation point $\zeta \in \bwt{M}$ of the net $(x_F,y_F)$, indexed by finite subsets of $C(\bwt{M})$ and directed by inclusion, satisfies $f(\zeta)=f(\rr(\zeta))$ for all $f \in C(\bwt{M})$, and therefore $\zeta = \rr(\zeta)$.

This property is satisfied by any metric space $M$ that contains homeomorphic copies of the Euclidean sphere $S^n \subset \RR^{n+1}$ for all $n \in \NN$, e.g. any infinite-dimensional Banach space or the Hilbert cube (compare to Example \ref{ex:all_optimal}). Indeed, given $F=\{f_1,\dots,f_n\}$ as above, let $\pi:S^n \to M$ be a homeomorphic embedding and define the continuous map $h:S^n \to \RR^n$ by $h(x)_i=f_i(\pi(x),\pi(-x))$, $i=1,\dots,n$. By the Borsuk-Ulam theorem, there exists $x \in S^n$ such that $h(x)=h(-x)$, i.e. $f_i(\pi(x),\pi(-x))=f_i(\pi(-x),\pi(x))$ for $i=1,\dots,n$. Therefore the point $(\pi(x),\pi(-x)) \in \wt{M}$ satisfies the above requirement.
\end{example}

\subsection{Those that live in between}\label{subsec_in_between}

We close this section by showing that there can exist elements $\zeta \in \bwt{M}$ satisfying $0 < \|\Phi^*\delta_\zeta\| < 1$. Again we will make use of Proposition \ref{prop_de_Leeuw_relation}. Given $\xi,\eta \in \rcomp{M}$, Proposition \ref{prop_de_Leeuw_relation} \ref{d_bar_less_than_d} provides the lower bound $\bar{d}(\xi,\eta)$ for $d(\zeta)$, $\zeta \in \pp^{-1}(\xi,\eta)$. 
 There exists an upper bound as well, for if $(x_i,y_i) \subset \wt{M}$ is a net converging to $\zeta \in \pp^{-1}(\xi,\eta)$, then
\[
 d(x_i,y_i) \leq d(x_i,0) + d(y_i,0)
\]
whence $d(\zeta) \leq \bar{d}(0,\xi) + \bar{d}(0,\eta)$ after taking limits and applying Proposition \ref{pr:bar-d_cont} \ref{bar-d_cont_3}. Example \ref{ex_all_values} below will show that sometimes it is possible to find $(\xi,\eta) \in \rcomp{M} \times \rcomp{M}$ such that all values within these bounds are attained. Before constructing our example, we need two combinatorial lemmas.

\begin{lemma}
\label{lm:I_kn}
There exists a family $(I_{k,n})_{k,n\in\NN}$ of infinite subsets of $\NN$ satisfying
\begin{enumerate}[label={\upshape{(\roman*)}}]
 \item\label{prop_min} $k < \min I_{k,n}$;
 \item\label{prop_dis} given $k \in \NN$, the sequence $(I_{k,n})_{n \in \NN}$ is pairwise disjoint;
 \item\label{prop_ref} given $k,k',n,n' \in \NN$ with $k \leq k'$, we have 
 \[
  I_{k',n'} \subset I_{k,n} \text{ whenever } I_{k,n} \cap I_{k',n'} \neq \varnothing;
 \]
 \item\label{prop_sub} if $p \in I_{k,n}$ and $q \in \NN$ then $I_{p,q} \subset I_{k,n}$.
\end{enumerate}
\end{lemma}

\begin{proof}
The sets $I_{k,n}$ are constructed by recursion on $k$. To begin, let $(I_{1,n})$ be a sequence of pairwise disjoint infinite subsets having union $\NN\setminus\{1\}$. Now assume that $(I_{k,n})$, $k \leq m$, have been constructed and satisfy properties \ref{prop_min} -- \ref{prop_sub} up to and including $m$ (i.e in all cases where $k,k',p \leq m$). Since $m+1 \in I_{1,n}$ for some $n \in \NN$, there exists maximal $k_0 \leq m$ for which $m+1 \in I_{k_0,n_0}$ for some (unique) $n_0 \in \NN$. Then let $k_1 \leq m$ be maximal, subject to the condition that $ I_{k_1,n_1} \subset I_{k_0,n_0}$ for some (unique) $n_1 \in \NN$. Finally, let $(I_{m+1,n})$ be a sequence of pairwise disjoint infinite subsets of $I_{k_1,n_1}\setminus\{1,\dots,m+1\}$. Clearly \ref{prop_min} and \ref{prop_dis} above hold; it remains to verify \ref{prop_ref} and \ref{prop_sub}.

To verify \ref{prop_ref}, let $k,n,n' \in \NN$ such that $k < m+1$ and $I_{k,n} \cap I_{m+1,n'} \neq \varnothing$. Then $I_{k,n} \cap I_{k_1,n_1} \neq \varnothing$, so by maximality of $k_1$ we get $k \leq k_1$. Indeed, if $k_1 < k$ then by \ref{prop_ref} (applied to $k_1,k \leq m$), we would have $I_{k,n} \subset I_{k_1,n_1} \subset I_{k_0,n_0}$. Consequently, again by \ref{prop_ref} (applied to $k,k_1 \leq m$), we have $I_{m+1,n'} \subset I_{k_1,n_1} \subset I_{k,n}$, as required. To verify \ref{prop_sub}, let $m+1 \in I_{k,n}$ for some $k \leq m$. By maximality of $k_0$ and \ref{prop_ref} (applied to $k,k_0 \leq m$), we obtain $I_{k_0,n_0} \subset I_{k,n}$. Therefore $I_{m+1,q} \subset I_{k_1,n_1} \subset I_{k_0,n_0} \subset I_{k,n}$ for all $q \in \NN$, as required.
\end{proof}

\begin{lemma}
\label{lm:Gamma_n}
There exists a sequence $(\Gamma_n)$ of pairwise disjoint subsets of $\set{(x,y):x<y}\subset\NN\times\NN$ with the following property: if
\[
\mathscr{H} = \set{A \subset \NN \,:\, (A \times A) \cap \Gamma_n = \varnothing \text{ for some }n \in \NN},
\]
then $\NN$ is not the union of finitely many elements of $\mathscr{H}$.
\end{lemma}

\begin{proof}
Let $(I_{k,n})$ be the family of sets from Lemma \ref{lm:I_kn} and set
\[
\Gamma_n = \set{(k,p) \,:\, k \in \NN,\, p \in I_{k,n}} . 
\]
We shall refer to Lemma \ref{lm:I_kn} \ref{prop_min} -- \ref{prop_sub} above. By \ref{prop_min} all $(x,y)\in\Gamma_n$ satisfy $x<y$. To see that the $\Gamma_n$ are pairwise disjoint, let $n,n' \in \NN$ be distinct and let $(x,y) \in \Gamma_n$. Then $y \in I_{x,n}$, giving $y \notin I_{x,n'}$ by \ref{prop_dis} and hence $(x,y) \notin \Gamma_{n'}$. To see that $\NN$ is not the union of finitely many elements of $\mathscr{H}$, let $A \in \mathscr{H}$ and $k,n \in \NN$. We find $p,q \in \NN$ such that $I_{p,q} \subset I_{k,n}$ and $A \cap I_{p,q} = \varnothing$. Indeed, if $A \cap I_{k,n} = \varnothing$ then set $p=k$ and $q=n$. Otherwise, let $p \in A \cap I_{k,n}$. As $A \in \mathscr{H}$, there exists $q \in \NN$ such that $(A \times A) \cap \Gamma_q = \varnothing$. It follows that $A \cap I_{p,q}=\varnothing$, for if $\ell \in A \cap I_{p,q}$ then $(p,\ell) \in (A \times A) \cap \Gamma_q$. By \ref{prop_sub}, we have $I_{p,q} \subset I_{k,n}$. Now, given finitely many elements $A_1,\dots,A_j \in \mathscr{H}$, we can iterate the above process to obtain some $I_{p,q}$ that is disjoint from all the $A_i$. As $I_{p,q}$ is an infinite subset of $\NN$, it follows that $\NN \neq \bigcup_{i=1}^j A_i$.
\end{proof}

\begin{example}\label{ex_all_values}
There exists a bounded and uniformly discrete metric $d$ on $\NN$ (with base point $1$) and $\xi,\eta \in \rcomp{\NN}$, such that $\bar{d}(\xi,\eta)=\bar{d}(1,\xi)=\bar{d}(1,\eta)=\frac{1}{2}$ and, given $r \in [\frac{1}{2},1]$, there exists $\zeta \in \pp^{-1}(\xi,\eta)$ such that $d(\zeta)=r$. Consequently, $\norm{\Phi^*\delta_\zeta} = \frac{1}{2r} \in [\frac{1}{2},1]$.
\end{example}

\begin{proof}
Let $(\Gamma_n)$ be the sequence from Lemma \ref{lm:Gamma_n}, let $(q_n)$ be an enumeration of $[\frac{1}{2},1] \cap \QQ$ and define a metric $d$ on $\NN$ by
\[
 d(x,y) = \begin{cases}
           0 &\text{if $x=y$}\\
           q_n &\text{if $(x,y)=(2p,2q+1)$ or $(2q+1,2p)$ for some $(p,q)\in\Gamma_n$}\\
          \frac{1}{2} &\text{otherwise}.
           \end{cases}
\]
The properties of the $\Gamma_n$ ensure that $d$ is well defined and symmetric, and the fact that $d(x,y) \in [\frac{1}{2},1]$ whenever $x\neq y$ guarantees that the triangle inequality is satisfied. Since $d$ is bounded and uniformly discrete, we have $\rcomp{\NN} = \ucomp{\NN} = \beta\NN$. Recall that $\beta\NN$ can be interpreted as the family of all ultrafilters on $\NN$, with the topology generated by the sets $\set{\xi\in\beta\NN:U\in\xi}$ for $U\subset\NN$. Following that interpretation, we have $\cl{U}^{\beta\NN}=\set{\xi\in\beta\NN:U\in\xi}$ for $U\subset\NN$.

Given $A\subset\NN$ and $s>0$, we set $sA=\set{n \in \NN \,:\, \frac{n}{s} \in A}$ and $A+1=\set{n+1:n \in A}$. Using the assumption on $(\Gamma_n)$ and $\mathscr{H}$, by compactness there exists
\[
\xi \in \bigcap_{A \in \mathscr{H}} \cl{2\NN}^{\beta \NN}\setminus\cl{2A}^{\beta \NN}.
\]
Then, for all $U\in\xi$, $U \cap 2\NN \in \xi$ and the set $\frac{1}{2}U \times \frac{1}{2}U$ intersects all $\Gamma_n$ (else $\frac{1}{2}U \in \mathscr H$, giving $U \cap 2\NN = 2(\frac{1}{2}U) \notin \xi$, which is false).
In addition, define 
$$
\eta = \set{V \subset \NN \,:\, U+1 \subset V \text{ for some }U \ni \xi} \in \beta \NN .
$$
We now prove our three claims:
\begin{enumerate}
 \item given $r \in [\frac{1}{2},1]$, there exists $\zeta \in \pp^{-1}(\xi,\eta)$ such that $d(\zeta)=r$;
 \item $\bar{d}(\xi,\eta)=\frac{1}{2}$;
 \item $\bar{d}(1,\xi)=\bar{d}(1,\eta)=\frac{1}{2}$.
\end{enumerate}

To prove claim (1), given $n \in \NN$, we find $\zeta_n \in \pp^{-1}(\xi,\eta)$ satisfying $d(\zeta_n)=q_n$. Indeed, we know that $(\frac{1}{2}U \times \frac{1}{2}U) \cap \Gamma_n \neq \varnothing$ whenever $U \in \xi$ and $U \subset 2\NN$, so by the definition of $d$ and compactness there exists
\[
\zeta_n \in \bigcap_{U \in \xi,\, U \subset 2\NN} \cl{(U \times (U+1)) \cap d^{-1}(q_n)}^{\bwt{\NN}}.
\]
Clearly $\pp(\zeta_n)=(\xi,\eta)$, and $d(\zeta_n)=q_n$ by continuity of $d$. The claim now follows by compactness of $\pp^{-1}(\xi,\eta)$ and the density of the $q_n$. 

To prove claim (2), define $f \in B_{\Lip_0(\NN)}$ by $f(n)=\frac{1}{2}$ for even $n$ and $f(n)=0$ for odd $n$. Evidently $\bar{d}(\xi,\eta) \geq \rcomp{f}(\xi)-\rcomp{f}(\eta)=\frac{1}{2}$. Since $\bar{d}(\xi,\eta)\leq d(\zeta_n) = q_n$ for all $n$, the assertion follows.

For claim (3), we simply observe that $\rho(x)=d(1,x)=\frac{1}{2}$ whenever $x>1$.
\end{proof}

\section{Optimal representations and optimal transport}\label{sect_optimal_rep_and_optimal_trans}

In this section we generalise the results obtained in \cite{APS23} relating optimal De Leeuw representations to optimal transport theory. These include the characterisation of certain optimal representations in terms of cyclical monotonicity (see Theorem \ref{th:cyclic_monotone_optimal} below), and of measure-induced functionals in terms of optimal representations with additional finiteness constraints (see \cite[Section 3]{APS23}). Throughout this section, if a measure $\mu$ is said to be concentrated on a set $C$, then the assumption is that $C$ is $\mu$-measurable at least; as $\mu$ is inner regular, if necessary we can replace $C$ by a $K_\sigma$ and thus Borel measurable subset.

\subsection{Optimal transport in \texorpdfstring{$\rcomp{M}$}{MR}}\label{sec:optimal_trans_MR}

A standard reference for optimal transport theory is \cite{Villani}, although it focusses on separable metric spaces. We will now introduce the necessary preliminaries and extend some of the basic results, in particular the Kantorovich-Rubinstein theorem, to $\rcomp{M}$. In fact, our arguments are valid in much more general topological spaces, so we will present the general setting.

The Monge-Kantorovich problem in optimal transport theory can be roughly stated as follows: \textit{find a way to transform a given distribution into another one at the least possible cost}. Our setting is a Hausdorff space $X$, and our ``distributions'' are Radon probability measures $\mu,\nu\in\meas{X}$. The possible ways of transforming or ``redistributing'' $\mu$ into $\nu$ are represented by Radon \emph{couplings}, that is, probability measures $\pi\in\meas{X\times X}$ whose marginals are $\mu$ and $\nu$, i.e. $(\pp_1)_\sharp\pi=\mu$ and $(\pp_2)_\sharp\pi=\nu$. Here $\pp_i$ stands for the projection from a product space into its $i$-th coordinate; this is essentially consistent with the notation in the preceding sections. The set of Radon couplings from $\mu$ to $\nu$ will be denoted by $\Pi(\mu,\nu)$. The cost associated with a coupling $\pi\in\Pi(\mu,\nu)$ is given as
$$
\int_{X\times X} d(x,y)\,d\pi(x,y)
$$
where $d$ is some predetermined cost function. We shall only be concerned with the case where $d$ is a metric on $X$, although not necessarily one that is compatible with the topology of $X$. In order to ensure that the optimal cost is finite, it is customary to assume that $\mu$ and $\nu$ have finite first $d$-moment, that is
$$
\int_X d(x,x_0)\,d\mu(x) < \infty
$$
for some (and hence all) $x_0\in X$, and likewise for $\nu$.

We recall here two classical notions in this setting. The first definition is given differently in \cite{APS23} but the two are equivalent because every finite permutation can be written as a product of disjoint cycles.

\begin{definition} Let $(X,d)$ be a metric space and $C$ be a subset of $X\times X$.
\begin{enumerate}[label={\upshape{(\alph*)}}]
\item We say that $C$ is \emph{cyclically monotonic} if, given $(x_1,y_1),\ldots,(x_n,y_n) \in C$ and a permutation $\sigma$ of $\set{1,\ldots,n}$, we have
$$
\sum_{k=1}^n d(x_k,y_k) \leq \sum_{k=1}^n d(x_k,y_{\sigma(k)}) .
$$
\item We say that a function $f:X\to\RR$ is \emph{extremal on $C$} if it is $1$-$d$-Lipschitz and satisfies
$$
f(x)-f(y) = d(x,y) \qquad\text{for all }(x,y) \in C.
$$
Accordingly, we say that $C$ admits an extremal function if there exists a function $f:X\to\RR$ that is extremal on $C$.
\end{enumerate}
\end{definition}

The next proposition shows that both notions are, in fact, almost equivalent, and we can moreover impose a measurability constraint on the extremal function under additional conditions. Its proof can be extracted from that of \cite[Theorem 5.10]{Villani}; given the different circumstances and the length of the foregoing proof we provide a direct argument for completeness.

\begin{proposition}\label{pr:cm_extremal}
Let $X$ be a Hausdorff space and $d$ be a lower semicontinuous metric on $X$. Given $\pi\in\meas{X\times X}$, the following are equivalent:
\begin{enumerate}[label={\upshape{(\roman*)}}]
\item\label{cm_extremal-cm} $\pi$ is concentrated on a $d$-cyclically monotonic set,
\item\label{cm_extremal-ext} $\pi$ is concentrated on a set that admits a $d$-extremal function,
\item\label{cm_extremal-borel} $\pi$ is concentrated on a set that admits a Borel $d$-extremal function.
\end{enumerate}
\end{proposition}

\begin{proof}
Implications \ref{cm_extremal-borel}$\Rightarrow$\ref{cm_extremal-ext}$\Leftrightarrow$\ref{cm_extremal-cm} hold for general sets (i.e. without considering a measure) and even without semicontinuity of $d$. Indeed, \ref{cm_extremal-borel}$\Rightarrow$\ref{cm_extremal-ext} is trivial, 
and \ref{cm_extremal-cm}$\Leftrightarrow$\ref{cm_extremal-ext} is \cite[Proposition 2.10]{APS23}. Note that \cite[Proposition 2.10]{APS23} is stated for subsets of $\wt{X}$ rather than $X\times X$, but points in the diagonal $(X\times X)\setminus\wt{X}$ are inconsequential to either \ref{cm_extremal-cm} or \ref{cm_extremal-ext}.

Let us now prove that \ref{cm_extremal-cm}$\Rightarrow$\ref{cm_extremal-borel}. Suppose that $\pi\neq 0$ is concentrated on a $d$-cyclically monotonic set $C\subset X\times X$. By Lusin's theorem (see \cite[Theorem 7.1.13]{Bogachev}) and the regularity of $\pi$, we can take an increasing sequence $(K_n)_{n=0}^\infty$ of non-empty compact subsets of $C$ such that $d\restrict_{K_n}$ is continuous and $\pi$ is concentrated on $\bigcup_{n=0}^\infty K_n$, and thus assume that $C=\bigcup_{n=0}^\infty K_n$. Using these sets, we will construct a Borel function $f:X\to\RR$ that is $d$-extremal on $C$, based on the classical formula (cf. \cite[(5.17) and pp. 69--71]{Villani}).

Fix some $(x_0,y_0)\in K_0$ and, for $n\geq 1$, define $\psi_n:X\times C^n\to\RR$ by
$$
\psi_n(x,\omega) = \sum_{k=0}^n d(x_k,y_k) - \sum_{k=0}^{n-1} d(x_k,y_{k+1}) - d(x_n,x),
$$
where $\omega=((x_1,y_1),\dots,(x_n,y_n))\in C^n$. Since $d$ is lower semicontinuous and $d\restrict_{K_n}$ is continuous, $\psi_n\restrict_{X\times K_n^n}$ is upper semicontinuous.  By $d$-cyclical monotonicity of $C$ we have
$$
\psi_n(x,\omega) \leq d(x_n,y_0) - d(x_n,x) \leq d(y_0,x) .
$$
Hence we can define $f_n:X\to\RR$ by
$$
f_n(x) = \sup\set{\psi_n(x,\omega) \,:\, \omega \in K_n^n} = \max\set{\psi_n(x,\omega) \,:\, \omega \in K_n^n} .
$$
A standard argument (see e.g. \cite[Exercise IV.6.9(a) on p. 396]{Bourbaki}) shows that each $f_n$ is also upper semicontinuous. Indeed, let $x\in X$ and $\varepsilon>0$. By the upper semicontinuity of $\psi_n$, given $\omega \in K_n^n$, take sets $U_\omega \ni x$ and $V_\omega \ni \omega$, open in $X$ and $K_n^n$, respectively, such that
$$
\psi_n(x',\omega') < \psi_n(x,\omega) + \varepsilon
$$
whenever $(x',\omega') \in U_\omega \times V_\omega$. Using compactness of $K_n^n$, find $\omega_1,\ldots,\omega_m \in K_n^n$ such that $K_n^n = \bigcup_{i=1}^m V_{\omega_i}$, and set $U=\bigcap_{i=1}^m U_{\omega_i}$. Then $U$ is a neighbourhood of $x$ in $X$ and, given $x' \in U$, there exists $\omega' \in K_n^n$ satisfying $f_n(x')=\psi_n(x',\omega')$; if $i\leq m$ is such that $\omega' \in V_{\omega_i}$, then $(x',\omega') \in U_{\omega_i} \times V_{\omega_i}$ and
$$
f_n(x') = \psi_n(x',\omega') < \psi_n(x,\omega_i) + \varepsilon \leq f_n(x) + \varepsilon.
$$
It follows that $f_n$ is upper semicontinuous.

The $f_n$ form a pointwise increasing sequence of functions. Indeed, given $\omega \in K_n^n$, one can verify that $\psi_n(x,\omega)=\psi_{n+1}(x,\omega')$, where $\omega'=(\omega,(x_n,y_n)) \in K_{n+1}^{n+1}$. As $f_n(x) \leq d(y_0,x)$ for all $x \in X$ and $n \in \NN$, we can define a function $f:X\to\RR$ by
$$
f(x) = \lim_{n\to\infty} f_n(x) = \sup\set{\psi_n(x,\omega) \,:\, \omega \in C^n, n \in \NN} \leq d(y_0,x) .
$$
Above, we have exploited the fact that if $\omega\in C^n$ then $\psi_n(x,\omega) = \psi_{n'}(x,\omega') \leq f_{n'}(x)$ for some $\omega' \in K_{n'}^{n'}$ and $n' \geq n$. Being a pointwise limit of a sequence of upper semicontinuous functions, $f$ is Borel. Also, $f$ is the pointwise supremum of the functions $\psi_n(\cdot,\omega)\in B_{\Lip(X,d)}$, and thus $f\in B_{\Lip(X,d)}$. Finally, we verify that $f$ is $d$-extremal on $C$. Fix $(x,y)\in C$. Given $\varepsilon>0$, let $n\in\NN$ and $\omega\in C^n$ such that $f(y) < \psi_n(y,\omega)+\varepsilon$. If we let $\omega'=(\omega,(x,y))\in C^{n+1}$, then
$$
f(x) \geq \psi_{n+1}(x,\omega') = \psi_n(y,\omega) + d(x,y) > f(y) + d(x,y) - \varepsilon.
$$
Thus $f(x)-f(y)\geq d(x,y)$, and the converse inequality is clear because $f \in B_{\Lip(X,d)}$.
\end{proof}

The classical Kantorovich-Rubinstein theorem relates the optimal cost of the transportation problem to 1-$d$-Lipschitz functions, and also shows that optimal couplings are characterised by being concentrated on cyclically monotonic sets. This is well known for metric spaces when $d$ is the associated metric \cite[Theorem 4.1]{Edwards} but also when $d$ is merely lower semicontinuous \cite[Theorem 1.14]{Villani03}. We will now generalise the theorem to normal Hausdorff spaces. In exchange for that, lower semicontinuity of $d$ is replaced by condition \eqref{eq:krm0} below. This condition implies lower semicontinuity (see \cite[p. 214]{Matouskova}) but it is not much stronger than that; in fact, it is equivalent to lower semicontinuity when $X$ is compact \cite[Corollary 2.5]{Matouskova}.

\begin{theorem}\label{th:kr_lsc}
Let $X$ be a normal Hausdorff space and $d$ be a metric on $X$ with the property that
\begin{equation}\label{eq:krm0}
\text{$\set{x\in X \,:\, d(x,A)\leq r}$ is closed for every closed $A\subset X$ and $r>0$}
\end{equation}
(i.e. $d(\cdot,A)$ is lower semicontinuous whenever $A \subset X$ is closed).
Let $\mu$, $\nu$ be Radon probability measures on $X$ with finite first $d$-moment. Then
\begin{equation}\label{eq:krm1}
\inf_{\pi\in\Pi(\mu,\nu)}\int_{X\times X} d\,d\pi = \sup \set{ \int_X f\,d(\mu - \nu) \,:\, f\in C(X)\cap B_{\Lip(X,d)} }
\end{equation}
and the infimum is attained. Moreover, given $\pi\in\Pi(\mu,\nu)$, the infimum of \eqref{eq:krm1} is attained at $\pi$ if and only if $\pi$ is concentrated on a $d$-cyclically monotonic subset of $X\times X$.
\end{theorem}

\begin{proof}
Let $\mu$ and $\nu$ be as above. First note that the infimum in \eqref{eq:krm1} is finite. Indeed, the product measure $\mu\times\nu$ extends uniquely to a Radon measure $\varpi$ on $X\times X$ by \cite[Theorem 7.6.2]{Bogachev}, which will belong to $\Pi(\mu,\nu)$. Fixing some $x_0\in X$, its cost can be estimated as
\begin{align*}
\int_{X\times X}d(x,y)\,d\varpi(x,y) &\leq \int_{X\times X}\pare{d(x,x_0)+d(x_0,y)}\,d\varpi(x,y) \\
&= \int_X d(x,x_0)\,d\mu(x) + \int_X d(y,x_0)\,d\nu(y) < \infty
\end{align*}
by the assumption of finite first $d$-moments.

Condition \eqref{eq:krm0} implies that $d$ is lower semicontinuous by the remarks in \cite[p. 214]{Matouskova}. We can thus appeal to \cite[Theorem 3.1]{Edwards} to conclude that the infimum in \eqref{eq:krm1} is attained and that it equals
\begin{equation}\label{eq:krm2}
\min_{\pi\in\Pi(\mu,\nu)}\int_{X\times X} d\,d\pi = \sup_{(u,v)\in\Theta} \set{ \int_X u\,d\mu + \int_X v\,d\nu }
\end{equation}
where $\Theta$ is the set of all pairs $(u,v)$ such that $u,v:X\to\RR$ are Borel, $u\in L_1(\mu)$, $v\in L_1(\nu)$, and moreover they satisfy
$$
u(x)+v(y)\leq d(x,y) \qquad\text{for all } x,y\in X .
$$
For every $f\in C(X)\cap B_{\Lip(X,d)}$ we have $(f,-f)\in\Theta$: indeed, $f(x)+(-f)(y)\leq d(x,y)$ because $f \in B_{\Lip(X,d)}$, and
$$
\int_X \abs{f(x)}\,d\mu(x) \leq \abs{f(x_0)} + \int_X d(x,x_0)\,d\mu(x) < \infty
$$
and similarly for $\nu$. Thus, the value in \eqref{eq:krm2} is greater than or equal to the right-hand side of \eqref{eq:krm1}. We will prove the reverse inequality.

Fix a coupling $\pi\in\Pi(\mu,\nu)$ where the value of \eqref{eq:krm2} is attained. We will prove that $\pi$ is concentrated on a $d$-cyclically monotonic set. Pick a sequence $(u_k,v_k)$ in $\Theta$ such that
$$
\lim_{k\to\infty}\pare{\int_X u_k\,d\mu + \int_X v_k\,d\nu} = \int_{X\times X} d\,d\pi ,
$$
that is
$$
\lim_{k\to\infty}\int_{X\times X} \pare{d(x,y)-u_k(x)-v_k(y)}\,d\pi(x,y) = 0 .
$$
The integrand is non-negative for every $k$, so it converges to $0$ in $L_1(\pi)$. Therefore we can choose a subsequence, still denoted $(u_k,v_k)$, such that the integrand converges pointwise to $0$ $\pi$-a.e., that is, $\pi$ is concentrated on a set $C$ such that $u_k(x)+v_k(y)\to d(x,y)$ for all $(x,y)\in C$. Now for each choice of $n\in\NN$, points $(x_i,y_i)\in C$, $i=1,\ldots,n$, and permutation $\sigma$ of $\set{1,\ldots,n}$, we have
$$
\sum_{i=1}^n d(x_i,y_{\sigma(i)}) \geq \sum_{i=1}^n \pare{u_k(x_i)+v_k(y_{\sigma(i)})} = \sum_{i=1}^n \pare{u_k(x_i)+v_k(y_i)}
$$
for all $k$, and taking the limit $k\to\infty$ we obtain
$$
\sum_{i=1}^n d(x_i,y_{\sigma(i)}) \geq \sum_{i=1}^n d(x_i,y_i) .
$$
Thus $C$ is $d$-cyclically monotonic as claimed.

Now suppose that $\pi\in\Pi(\mu,\nu)$ is concentrated on a $d$-cyclically monotonic set $C$. We will prove that $\int_{X\times X}d\,d\pi$ is bounded above by the right-hand side of \eqref{eq:krm1}, and this will complete the proof. By Proposition \ref{pr:cm_extremal}, we may assume that there exists a Borel function $g:X\to\RR$ that is $d$-extremal on $C$. Note that
$$
\int_X g\,d\mu - \int_X g\,d\nu = \int_{X\times X} g(x)\,d\pi(x,y) - \int_{X\times X} g(y)\,d\pi(x,y) = \int_{X\times X} d\,d\pi .
$$
Fix $\varepsilon>0$. Then we can find a bounded Borel 1-$d$-Lipschitz function $h:X\to\RR$ such that
$$
\int_X h\,d\mu - \int_X h\,d\nu \geq \int_X g\,d\mu - \int_X g\,d\nu - \varepsilon = \int_{X\times X} d\,d\pi - \varepsilon ;
$$
indeed, by the dominated convergence theorem it suffices to take $h=\mathrm{mid}\set{-k,g,k}$ for $k$ large enough.
By Lusin's theorem and regularity, we can choose a compact set $K\subset X$ such that $h\restrict_K$ is continuous and
$$
\pare{ \mu(X\setminus K)+\nu(X\setminus K) }\cdot \norm{h}_\infty \leq \varepsilon .
$$
Condition \eqref{eq:krm0} allows us to apply Matou\v{s}kov\'a's extension theorem \cite[Theorem 2.4]{Matouskova} to find $f\in C(X)\cap B_{\Lip(X,d)}$ such that $f\restrict_K=h\restrict_K$ and $\norm{f}_\infty\leq\norm{h}_\infty$. So we get
$$
\int_X f\,d\mu - \int_X f\,d\nu \geq \int_X h\,d\mu - \int_X h\,d\nu - 2\varepsilon \geq \int_{X\times X} d\,d\pi - 3\varepsilon .
$$
Since $\varepsilon>0$ was arbitrary, this shows that the right-hand side of \eqref{eq:krm1} is greater than or equal to $\int_{X\times X}d\,d\pi$, and this finishes the proof.
\end{proof}

Note that the space $(\rcomp{M},\tau)$ together with the metric $\bar{d}$ satisfies the hypotheses of Theorem \ref{th:kr_lsc}, as proved in Propositions \ref{prop_rcomp_observations} \ref{rcomp_normal} and \ref{pr:bar-d_cont} \ref{bar-d_cont_1}. Moreover, a real-valued function on $\rcomp{M}$ is $\bar{d}$-Lipschitz and $\tau$-continuous if and only if it is the continuous extension of a function in $\Lip(M)$. Thus we have the following particular case.

\begin{corollary}\label{cr:kr_rcomp}
Let $\mu$, $\nu$ be Radon probability measures on $\rcomp{M}$ with finite first $\bar{d}$-moment. Then
\begin{equation}\label{eq:kr1}
\inf_{\pi\in\Pi(\mu,\nu)}\int_{\rcomp{M}\times\rcomp{M}} \bar{d}\,d\pi = \sup_{f\in B_{\Lip(M)}} \int_{\rcomp{M}}\rcomp{f}\,d(\mu - \nu) .
\end{equation}
The infimum is attained, and $\pi\in\Pi(\mu,\nu)$ minimises \eqref{eq:kr1} if and only if it is concentrated on a $\bar{d}$-cyclically monotonic subset of $\rcomp{M}\times\rcomp{M}$.
\end{corollary}

In Corollary \ref{cr:kr_rcomp}, as well as the sequel, the topology considered for $\rcomp{M}$ will be $\tau$ unless specified otherwise; in particular, Borel functions and measures on $\rcomp{M}$ or $\rcomp{M}\times\rcomp{M}$ will be assumed to be $\tau$-Borel.

\subsection{Optimal representations and cyclical monotonicity}\label{sec:cyc_mon}

In \cite[Section 2]{APS23}, the optimality of De Leeuw representations is related to cyclical monotonicity for the particular case of representations concentrated on $\wt{M}$. The main result is as follows.

\begin{theorem}[{\cite[Theorem 2.9]{APS23}}]
\label{th:cyclic_monotone_optimal}
Suppose that $\mu\in\meas{\bwt{M}}$ is positive and concentrated on $\wt{M}$. Then $\mu$ is optimal if and only if it is concentrated on a cyclically monotonic set.
\end{theorem}

From this we obtain a necessary condition for the shape of the support of such representations, expressed in terms of metric alignment.

\begin{corollary}[{\cite[Corollary 2.11]{APS23}}]
\label{cr:metric_triples_pm}
Let $\mu\in\opr{\bwt{M}}$ be concentrated on $\wt{M}$ and suppose that $(x_0,x_1), (x_1,x_2), \ldots,$ $(x_{n-1},x_n)\in\supp(\mu)\cap\wt{M}$. Then $\sum_{k=1}^n d(x_{k-1},x_k) = d(x_0,x_n)$. In particular, $x_1,\ldots,x_{n-1} \in [x_0,x_n]$.
\end{corollary}

We wish to generalise both of these results to measures that are no longer concentrated on $\wt{M}$, but we need a suitable metric on the coordinates of $\bwt{M}$ to be able to do that. The metric $\bar{d}$ introduced in Section \ref{subsec_realcompact} will satisfy this requirement. Recall from \eqref{subsec_realcompact} the dense open subset $\bwtf = \pp^{-1}(\rcomp{M} \times \rcomp{M})$ of $\bwt{M}$. In this section, we consider sets $C \subset \bwtf$ such that $\pp(C) \subset \rcomp{M} \times \rcomp{M}$ is $\bar{d}$-cyclically monotonic and satisfies the additional condition (cf. Remark \ref{rem_bar_d_d_equality}) that $\bar{d} \circ \pp\restrict_C = d\restrict_C$. First, we present our main tool.

\begin{lemma}\label{lemma_minimal_1}
Let $\zeta_1,\dots,\zeta_n \in \bwtf$ and $\sigma$ be a permutation of $\{1,\dots,n\}$. Suppose that, given $\varepsilon>0$, there exists $f \in B_{\Lip_0(M)}$ such that $\Phi f(\zeta_k) \geq 1-\varepsilon$ for $k=1,\dots,n$. Then
\begin{equation}\label{eqn_ext_cyc}
 \sum_{k=1}^n d(\zeta_k) \leq \sum_{k=1}^n \bar{d}(\pp_1(\zeta_k),\pp_2(\zeta_{\sigma(k)})).
\end{equation}
\end{lemma}

\begin{proof}
Let $\varepsilon > 0$ and pick $f \in B_{\Lip_0(M)}$ as above. Then, using Proposition \ref{prop_de_Leeuw_relation} \ref{delta_diff}, we have
\begin{align*}
(1-\varepsilon)\sum_{k=1}^n d(\zeta_k) &\leq \sum_{k=1}^n d(\zeta_k)\Phi f(\zeta_k)\\
&= \sum_{k=1}^n \rcomp{f}(\pp_1(\zeta_k)) - \rcomp{f}(\pp_2(\zeta_k)) \\
&= \sum_{k=1}^n \rcomp{f}(\pp_1(\zeta_k)) - \rcomp{f}(\pp_2(\zeta_{\sigma(k)})) \\
&\leq \sum_{k=1}^n \bar{d}(\pp_1(\zeta_k),\pp_2(\zeta_{\sigma(k)})) .
\end{align*}
Since this holds for all $\varepsilon>0$, the result follows.
\end{proof}

Now we explore some circumstances under which the hypotheses of Lemma \ref{lemma_minimal_1} can hold.

\begin{lemma}\label{lem_pointwise}
 Let $\mu \in \opr{\bwt{M}}$. Then there exists a $\mu$-null set $N \subset \supp(\mu)\setminus \wt{M}$ such that, given $\varepsilon>0$ and $\zeta_1,\dots,\zeta_n \in \supp(\mu)\setminus N$, there is $f \in B_{\Lip_0(M)}$ satisfying $\Phi f(\zeta_k) \geq 1-\varepsilon$ for $k=1,\dots,n$.
\end{lemma}

\begin{proof}
Assume $\|\mu\|=1$. For $m \in \NN$ take $f_m \in B_{\Lip_0(M)}$ such that $\duality{f_m,\Phi^*\mu} \geq 1-2^{-m}$. Then
\[
 \int_{\bwt{M}} \abs{\mathbf{1}_{\bwt{M}} - \Phi f_m} \,d\mu = \int_{\bwt{M}} (\mathbf{1}_{\bwt{M}} - \Phi f_m) \,d\mu = \norm{\mu} - \duality{f_m,\Phi^*\mu} \leq 2^{-m},
\]
whence $\sum_{m=1}^\infty \norm{\mathbf{1}_{\bwt{M}} - \Phi f_m}_{L_1(\mu)} < \infty$. It follows that $(\Phi f_m)$ converges to $\mathbf{1}_{\bwt{M}}$ $\mu$-a.e.

Define the $\mu$-null set $N=\set{\zeta \in \supp(\mu) \,:\, \Phi f_m(\zeta) \not\to 1}$. To finish the proof, we need to show that $N \cap \wt{M} = \varnothing$. For a contradiction, suppose that there exists $(x,y) \in N \cap \wt{M}$. Let
$$
U = \set{(x',y')\in\wt{M} \,:\, \tfrac{1}{2}d(x,y)<d(x',y')<2d(x,y)} ,
$$
which is open in $\wt{M}$. Next, we observe that there exists $L>0$ such that $\Phi f_m$ is $L$-Lipschitz on $U$ for all $m \in \NN$ with respect to the metric $\wt{d}$ on $\wt{M}$, given by 
\[
\wt{d}((x_1,y_1),(x_2,y_2)) = d(x_1,x_2) + d(y_1,y_2).
\]
Indeed, on $U$ the maps $1/d$ and $(x',y') \mapsto f_m(x')-f_m(y')$, $m \in \NN$, are bounded by $2/d(x,y)$ and $2d(x,y)$, respectively, and are $4/d(x,y)^2$-Lipschitz and $1$-Lipschitz with respect to $\wt{d}$, respectively. Thus, given the ``Leibniz inequality'' for products of bounded Lipschitz functions \cite[{Proposition 1.30 (i)}]{Weaver2}, our claim holds with $L=10/d(x,y)$.

By taking a subsequence of the $f_m$ if necessary, let $\varepsilon>0$ such that $\Phi f_m(x,y) \leq 1-2\varepsilon$ for all $m$. Now define the set
\[
V=\set{(x',y') \in U \,:\, \wt{d}((x,y),(x',y')) < \textstyle\frac{\varepsilon}{L}},
\]
which is again open in $\wt{M}$. Then, given $(x',y') \in V$ and $m \in \NN$, we have
\[
%\textstyle
\Phi f_m(x',y') < \Phi f_m(x,y) + \varepsilon \leq 1 - \varepsilon.
\]
By continuity, $\Phi f_m(\zeta) \leq 1 - \varepsilon$ whenever $\zeta \in \cl{V}$. Thus $\supp(\mu) \cap \cl{V} \subset N$. However, since $(x,y) \in \supp(\mu) \cap \cl{V}$ and $\cl{V}$ contains a neighbourhood of $(x,y)$ open in $\bwt{M}$, we have $\mu(N) \geq \mu(\supp(\mu) \cap \cl{V}) = \mu(\cl{V})>0$, which is a contradiction.
\end{proof}

Using Lemma \ref{lem_pointwise}, we can make a general statement about the geometry of part of the support of any $\mu \in \opr{\bwt{M}}$. In particular, it considerably strengthens the assertion, made in Remark \ref{rem_bar_d_d_equality}, that optimal representations concentrated on $\bwtf$ are concentrated moreover on a set on which $\bar{d} \circ \pp$ agrees with $d$: there exists such a set that is also cyclically monotonic, in the following sense.

\begin{theorem}\label{thm_minimal_1}
Let $\mu \in \opr{\bwt{M}}$. Then there exists a $\mu$-null set $N \subset \supp(\mu)\setminus \wt{M}$, such that if we set $C = \bwtf \cap \supp(\mu)\setminus N$, then $\pp(C)$ is $\bar{d}$-cyclically monotonic and $\bar{d} \circ \pp\restrict_C = d\restrict_C$.
\end{theorem}

\begin{proof}
Let $N$ be the set from Lemma \ref{lem_pointwise}. Then, given $\zeta_1,\dots,\zeta_n \in C$ and $\varepsilon>0$, we can find $f \in B_{\Lip_0(M)}$ such that $\Phi f(\zeta_k) \geq 1-\varepsilon$ for $k=1,\dots,n$.
By Lemma \ref{lemma_minimal_1}, \eqref{eqn_ext_cyc} holds. Because $\bar{d}(\pp(\zeta)) \leq d(\zeta)$ for all $\zeta \in \bwtf$ by Proposition \ref{prop_de_Leeuw_relation} \ref{d_bar_less_than_d}, it follows that $\pp(C)$ is $\bar{d}$-cyclically monotonic. Finally, by applying the identity permutation to \eqref{eqn_ext_cyc}, we obtain $\bar{d} \circ \pp\restrict_C = d\restrict_C$.
\end{proof}

One implication of Theorem \ref{th:cyclic_monotone_optimal} is implied by the next result.

\begin{corollary}\label{cor_minimal_3}
Let $\mu \in \opr{\bwt{M}}$. Then $\supp(\mu) \cap \wt{M}$ is cyclically monotonic.
\end{corollary}

\begin{proof}
A consequence of Theorem \ref{thm_minimal_1} as $\wt{M} \subset \bwtf \setminus N$ and $\bar{d}$ agrees with $d$ on $M$.
\end{proof}

\begin{remark}
Note that Corollary \ref{cor_minimal_3} cannot be obtained from Theorem \ref{th:cyclic_monotone_optimal} simply by considering the restriction $\mu\restrict_{\wt{M}}$ of $\mu \in \opr{\bwt{M}}$, because the inclusion $\supp(\mu)\cap\wt{M}\subset\supp(\mu\restrict_{\wt{M}})$ is not valid in general. To see this, let $M=B_{\ell_1}$ with the norm metric, let $(e_n)$ denote the standard basis of $\ell_1$ and set $q = \frac{1}{2}e_1$. For each $m \in \NN$, let $\zeta_m \in \bwt{M}\setminus\wt{M}$ be an accumulation point of the sequence $(q+2^{-m}e_n,2^{-m}e_n)_{n=1}^\infty$. Then $(\zeta_m)$ converges to $(q,0)$. Indeed, given $U \ni (q,0)$ open in $\bwt{M}$, take $V \ni (q,0)$ open in $\bwt{M}$, such that $\cl{V} \subset U$, and then $m_0 \in \NN$ such that $(q+2^{-m_0}B_{\ell_1}) \times 2^{-m_0}B_{\ell_1} \subset V \cap \wt{M}$. It follows that $\zeta_m \in \cl{V} \subset U$ for $m \geq m_0$. Now set $\mu=\sum_{m=1}^\infty 2^{-m}\delta_{\zeta_m}$. Then $(q,0)\in\supp(\mu) \cap \wt{M}$ and $\mu\restrict_{\wt{M}}=0$. Furthermore, for $\rho(x)=\norm{x}$, $x\in M$, we have $\Phi \rho(q+2^{-m}e_n,2^{-m}e_n)=1$ for all $m,n \in \NN$, so $\Phi \rho(\zeta_m)=1$ for all $m \in \NN$. It follows that $\mu \in \opr{\bwt{M}}$ by Lemma \ref{lm:norming_phi1}. 
\end{remark}

We can use Theorem \ref{thm_minimal_1} to get a natural extension of Corollary \ref{cr:metric_triples_pm}.

\begin{corollary}\label{cor_minimal_4}
Let $\mu \in \opr{\bwt{M}}$ and let $C$ be defined as in Theorem \ref{thm_minimal_1}. Then given $(\xi_0,\xi_1),$ $(\xi_1,\xi_2),\dots,(\xi_{n-1},\xi_n) \in \pp(C)$, we have $\sum_{k=1}^n \bar{d}(\xi_{k-1},\xi_k) = \bar{d}(\xi_0,\xi_n)$. In particular, $\xi_1,\ldots,\xi_{n-1} \in \rcomp{[\xi_0,\xi_n]}$.
\end{corollary}

\begin{proof}
Consider the cyclic permutation $\sigma$ on $\set{1,\ldots,n}$ defined by $\sigma(1)=n$ and $\sigma(k)=k-1$ for $k=2,\ldots,n$. By Theorem \ref{thm_minimal_1}, we have
\[
\sum_{k=1}^n \bar{d}(\xi_{k-1},\xi_k) \leq \sum_{k=1}^n \bar{d}(\xi_{k-1},\xi_{\sigma(k)}) = \bar{d}(\xi_0,\xi_n) + \sum_{k=2}^n \bar{d}(\xi_{k-1},\xi_{k-1}) = \bar{d}(\xi_0,\xi_n).
\]
The other inequality follows obviously from the triangle inequality.
\end{proof}

If we impose an additional assumption on $\mu \in \opr{\bwt{M}}$, namely that $\Phi^*\mu$ avoids infinity, then we can draw a correspondingly stronger conclusion about the geometry of a set on which $\mu$ is concentrated. We know from Section \ref{functionals_avoid_infinity} that functionals that avoid infinity possess optimal De Leeuw representations concentrated on $\bwtf$. We see now that, furthermore, such representations are concentrated on $\bar{d}$-cyclically monotonic sets.

\begin{corollary}\label{cor_minimal_5}
Let $\mu \in \opr{\bwt{M}}$ such that $\Phi^*\mu$ avoids infinity. Then $\mu$ is concentrated on a set $C \subset \bwtf$ such that $\pp(C)$ is $\bar{d}$-cyclically monotonic and $\bar{d} \circ \pp\restrict_C = d\restrict_C$.
\end{corollary}

\begin{proof}
Let $C$ be the set from Theorem \ref{thm_minimal_1}. Because $\Phi^*\mu$ avoids infinity, Corollary \ref{cor_avoid_infinity} ensures that $\mu$ is concentrated on $C$ and the conclusion follows.
\end{proof}

\subsection{Representations avoiding the diagonal}

The converse of Corollary \ref{cor_minimal_5} fails in general, in the sense that if $\mu$ is positive and concentrated on a set $C \subset \bwtf$ such that $\pp(C)$ is $\bar{d}$-cyclically monotonic and $\bar{d} \circ \pp\restrict_C = d\restrict_C$, then it is not necessarily true that $\mu \in \opr{\bwt{M}}$. The gremlins in this case are vanishing points.

\begin{example}\label{ex_converse_fail}
Let $M$ be a compact metric space that cannot be bi-Lipschitz embedded into Euclidean space and let $\zeta \in \bwt{M}$ be a vanishing point, furnished by Proposition \ref{pr:vanishing_points}. By compactness, $\bwtf=\bwt{M}$ and $d(\zeta)=0$. Then $\delta_\zeta$ is concentrated on $C=\{\zeta\}$, and evidently $\pp(C)$ is $\bar{d}$-cyclically monotonic and $\bar{d} \circ \pp\restrict_C = d\restrict_C$. However, $\delta_\zeta \notin \opr{\bwt{M}}$ because $\Phi^*\delta_\zeta=0$.
\end{example}

Despite Example \ref{ex_converse_fail}, it turns out that we really have equivalence if we restrict proceedings to the set
\begin{equation}
\bwtfp := \pp^{-1}(\wt{\rcomp{M}}) = \set{\zeta\in\bwt{M} \,:\, \pp_1(\zeta)\neq\pp_2(\zeta)\in\rcomp{M} }
\end{equation}
where $\wt{\rcomp{M}}\subset\rcomp{M}\times\rcomp{M}$ is defined analogously to $\wt{M}\subset M\times M$. Recall that for $\zeta\in\bwtfp$, $\bar{d}(\pp(\zeta))=d(\zeta)$ if and only if $\zeta$ is an optimal point (Proposition \ref{prop_de_Leeuw_relation} \ref{d_bar_less_than_d}).

\begin{theorem}\label{thm_equiv_d_bar_cyclical_monotonicity}
Let $\mu \in \meas{\bwt{M}}$ be positive and concentrated on $\bwtfp$. Then $\mu \in \opr{\bwt{M}}$ if and only if it is concentrated on a set $C \subset \bwtfp\cap\opt$ such that $\pp(C)$ is $\bar{d}$-cyclically monotonic.
\end{theorem}

Observe that, as $\wt{M} \subset \bwtfp$ and $\bar{d}$ agrees with $d$ on $M$, Theorem \ref{th:cyclic_monotone_optimal} is a corollary of Theorem \ref{thm_equiv_d_bar_cyclical_monotonicity}.

\begin{proof}
The forward implication follows immediately from Corollaries \ref{cor_avoid_infinity} and \ref{cor_minimal_5}. For the converse, suppose that $\mu$ is concentrated on a set $C \subset \bwtfp\cap\opt$ such that $\pp(C)$ is $\bar{d}$-cyclically monotonic. By regularity, we may assume that $C$ is $\sigma$-compact and therefore $\pp(C)$ is $\tau$-Borel. Then $\pp_\sharp\mu\in\meas{\rcomp{M}\times\rcomp{M}}$ is concentrated on $\pp(C)$. By Proposition \ref{pr:cm_extremal}, $\pp_\sharp\mu$ is also concentrated on a set that admits a $\tau$-Borel $\bar{d}$-extremal function $\psi:\rcomp{M}\to\RR$. Thus, by substituting $C$ with an appropriate subset, we may assume that $\psi$ is $\bar{d}$-extremal on $\pp(C)$. We may moreover assume that $\psi(0)=0$, as extremality is invariant to the addition of constants.

Assume $\norm{\mu}=1$. It suffices to show that, given $\varepsilon>0$, there exists $f \in B_{\Lip_0(M)}$ such that $\duality{f,\Phi^*\mu} > 1 - \varepsilon$. To this end, fix $\varepsilon>0$. Define $\nu = (\pp_1)_\sharp \mu + (\pp_2)_\sharp \mu \in \meas{\ucomp{M}}$. Since $C \subset \bwtf$, we have that $\nu$ is concentrated on $\rcomp{M}$. By Lusin's theorem, there exists a compact set $K \subset \rcomp{M}$ such that $\nu(\ucomp{M}\setminus K)=\nu(\rcomp{M}\setminus K)<\frac{1}{2}\varepsilon$ and $\psi\restrict_K$ is $\tau$-continuous. Without loss of generality we can and do assume that $0 \in K$ (otherwise $0$ is isolated in $K\cup\set{0}$ and thus $\psi\restrict_{K\cup\set{0}}$ is $\tau$-continuous as well). By Proposition \ref{pr:Mat_app}, there exists a bounded $1$-Lipschitz function $f:M \to \RR$ such that $\psi\restrict_K = \rcomp{f}\restrict_K$. As $0 \in K$, we have $f \in B_{\Lip_0(M)}$. Set $L = C \cap \pp^{-1}(K \times K)$. Because $C \subset \bwtfp\cap\opt$, given $\zeta \in L$ we have $d(\zeta) = \bar{d}(\pp(\zeta)) \in (0,\infty)$ and $\pp(\zeta) \in K \times K$, thus by Proposition \ref{prop_de_Leeuw_relation} \ref{delta_diff}
\[
\Phi f(\zeta) = \frac{\rcomp{f}(\pp_1(\zeta)) - \rcomp{f}(\pp_2(\zeta))}{d(\zeta)} = \frac{\psi(\pp_1(\zeta)) - \psi(\pp_2(\zeta))}{\bar{d}(\pp(\zeta))} = 1.
\]
Furthermore,
\[
\mu(C\setminus L) \leq \mu(\pp_1^{-1}(\ucomp{M}\setminus K) \cup \pp_2^{-1}(\ucomp{M}\setminus K)) \leq \nu(\ucomp{M}\setminus K) < \textstyle\frac{1}{2}\varepsilon,
\]
whence $\mu(L)>1-\frac{1}{2}\varepsilon$ and
\[
\duality{f,\Phi^*\mu}  > \int_L \Phi f\,d\mu - {\textstyle\frac{1}{2}\varepsilon} = \mu(L) - {\textstyle\frac{1}{2}\varepsilon} > 1 - \varepsilon, 
\]
as required.
\end{proof}

The preceding proof highlights a connection between measures on $\bwtfp$ and measures on $\wt{\rcomp{M}}$. To any positive $\mu\in\meas{\bwtfp}$ (not necessarily optimal) we can associate the measure $\nu=\pp_\sharp\mu\in\meas{\wt{\rcomp{M}}}$, which satisfies $\norm{\nu}=\norm{\mu}$. This measure has the same marginals as $\mu$, in the sense that
\begin{equation}\label{eq:same marginals}
(\pp_1)_\sharp\mu(E) = \mu(\pp_1^{-1}(E)) = \mu(\pp^{-1}(E\times\rcomp{M})) = \nu(E\times\rcomp{M}) = \nu(\pp_1^{-1}(E)) = (\pp_1)_\sharp\nu(E)
\end{equation}
for every Borel $E\subset\rcomp{M}$, and similarly for $\pp_2$. Note that here we are using the same notation for the coordinate mappings $\pp_i:\bwt{M}\to\ucomp{M}$ and the projections $\pp_i:\rcomp{M}\times\rcomp{M}\to\rcomp{M}$; context should dispel any confusion. Conversely, every $\nu\in\meas{\wt{\rcomp{M}}}$ can be lifted to $\mu\in\meas{\bwtfp}$ such that $\pp_\sharp\mu=\nu$ and $\norm{\mu}=\norm{\nu}$ as a consequence of the Hahn-Banach theorem. A less obvious fact is that we can choose such $\mu$ to be concentrated on optimal points.

\begin{proposition}\label{pr:lifting_to_S0}
If $\nu\in\meas{\rcomp{M}\times\rcomp{M}}$ is concentrated on $\wt{\rcomp{M}}$, then there exists $\mu\in\meas{\bwt{M}}$ concentrated on $\bwtfp\cap\opt$ such that $\pp_\sharp\mu=\nu$ and $\norm{\mu}=\norm{\nu}$. If $\nu$ is positive then $\mu$ can also be chosen to be positive.
\end{proposition}

\begin{proof}
By Lusin's theorem and regularity, $\nu$ is concentrated on a set $C=\bigcup_{n=1}^\infty K_n$, where $K_n$ is an increasing sequence of compact subsets of $\wt{\rcomp{M}}$ such that $\bar{d}\restrict_{K_n}$ is continuous. Let us see that $\pp^{-1}(K_n)\cap\opt$ is compact for each $n$. Suppose that $(\zeta_i)$ is a net in $\pp^{-1}(K_n)\cap\opt$. Note that $\pp^{-1}(K_n)$ is compact, so there is a subnet (still denoted $(\zeta_i)$) that converges to $\zeta\in\pp^{-1}(K_n)$. Since $\bar{d}\restrict_{K_n}$ is continuous and $\zeta_i\in\opt$ for all $i$, we get
$$
d(\zeta) = \lim_i d(\zeta_i) = \lim_i \bar{d}(\pp(\zeta_i)) = \bar{d}(\pp(\zeta)) .
$$
Thus $\zeta\in\opt$, and this proves our claim. Moreover, by Proposition \ref{pr:optimals_everywhere} we have $\pp(\pp^{-1}(K_n)\cap\opt)=K_n$. Therefore we can apply \cite[Theorem 9.1.9]{Bogachev} to get the conclusion.
\end{proof}

Define the map $\Psi:\meas{\wt{\rcomp{M}}}\to\Lip_0(M)^*$ by
\begin{equation}\label{eq:Psi}
\duality{f,\Psi\nu} = \int_{\wt{\rcomp{M}}} \frac{\rcomp{f}(\xi)-\rcomp{f}(\eta)}{\bar{d}(\xi,\eta)} \,d\nu(\xi,\eta)
\end{equation}
for $f\in\Lip_0(M)$. It is straightforward to check that $\Psi$ is a linear operator with norm $1$, although it is not $w^*$-$w^*$-continuous (precisely because $\bar{d}$ is not continuous on $\wt{\rcomp{M}}$ in general; this is witnessed e.g. by points $\zeta\in\bwtfp$ such that $0<\norm{\Phi^*\delta_\zeta}<1$, as in Example \ref{ex_all_values}). If $\mu\in\meas{\bwt{M}}$ is concentrated on $\bwtfp\cap\opt$, then $\Phi^*\mu=\Psi(\pp_\sharp\mu)$, as
\begin{align*}
\duality{f,\Psi(\pp_\sharp\mu)} &= \int_{\wt{\rcomp{M}}} \frac{\rcomp{f}(\xi)-\rcomp{f}(\eta)}{\bar{d}(\xi,\eta)} \,d(\pp_\sharp\mu)(\xi,\eta) = \int_{\bwtfp} \frac{\rcomp{f}(\pp_1(\zeta))-\rcomp{f}(\pp_2(\zeta))}{\bar{d}(\pp(\zeta))} \,d\mu(\zeta) \\
&= \int_{\bwtfp} \frac{\rcomp{f}(\pp_1(\zeta))-\rcomp{f}(\pp_2(\zeta))}{d(\zeta)} \,d\mu(\zeta) = \int_{\bwtfp} \Phi f(\zeta)\,d\mu(\zeta) = \duality{f,\Phi^*\mu}
\end{align*}
for $f\in\Lip_0(M)$, using Proposition \ref{prop_de_Leeuw_relation} \ref{delta_diff}. Thus, for measures concentrated on $\bwtfp\cap\opt$, in particular for optimal measures concentrated on $\bwtfp$, the De Leeuw transform $\Phi^*$ factors through the pushforward operator $\pp_\sharp$. All relevant information about such measures is therefore already codified in $\wt{\rcomp{M}}$, with no need to consider the (bigger) space $\bwtfp\subset\bwt{M}$. In particular, the existence of optimal measures concentrated on $\bwtfp$ can be determined directly on $\wt{\rcomp{M}}$. Indeed, the following is an immediate consequence of Theorem \ref{thm_equiv_d_bar_cyclical_monotonicity} and Proposition \ref{pr:lifting_to_S0}.

\begin{corollary}\label{cr:optimal_S0}
Let $\phi\in\Lip_0(M)^*$. Then the following are equivalent:
\begin{enumerate}[label={\upshape{(\roman*)}}]
\item $\phi$ admits an optimal representation $\mu\in\opr{\bwt{M}}$ concentrated on $\bwtfp$,
\item $\phi=\Psi\nu$ for some positive $\nu\in\meas{\wt{\rcomp{M}}}$ concentrated on a $\bar{d}$-cyclically monotonic set.
\end{enumerate}
If the above hold, then $\norm{\phi}=\norm{\nu}$ and we can choose $\mu,\nu$ to satisfy $\pp_\sharp\mu=\nu$.
\end{corollary}

\begin{remark}
We finish this subsection by drawing attention to the similitude between images under $\Psi$ of Radon measures on $\wt{\rcomp{M}}$ and convex integrals of molecules on the metric space $(\rcomp{M},\bar{d})$. The first setting is strictly more general than the second because we are considering ($\tau$-)Radon measures versus $\bar{d}$-Radon measures on $\wt{\rcomp{M}}$, respectively. For this reason the results in this paper having counterparts in \cite{APS23} are strictly more general.
Let us justify these remarks. Any $\bar{d}$-Radon measure $\nu$ on $\wt{\rcomp{M}}$ yields a convex integral of molecules $m \in \lipfree{\rcomp{M},\bar{d}}$ via Proposition \ref{pr:wt_bochner}, which can be regarded as an element of $\Lip_0(M)^*$ by Remark \ref{rm:lipfree_rcomp}. As $\tau$-Borel sets are $\bar{d}$-Borel and $\bar{d}$-compact sets are $\tau$-compact, $\nu$ is $\tau$-Radon as well, and one can verify that $\Psi\nu = m$. The converse fails because there can exist Radon measures $\nu$ on $\wt{\rcomp{M}}$ such that $\Psi\nu \notin \lipfree{\rcomp{M},\bar{d}}$; one can construct an example by considering $M=\NN$ with the discrete metric.
%To see this, let $M=\NN$ with the discrete metric $d$ (and base point $1$). Then $\rcomp{M}=\beta\NN$ and $\bar{d}$ is the discrete metric on $\rcomp{M}$. Fix a non-zero positive dispersed measure $\lambda \in \meas{\rcomp{M}}$, define the continuous map $\pi:\rcomp{M}\to\rcomp{M}\times\rcomp{M}$ by $\pi(\xi)=(\xi,1)$, and consider the pushforward $\nu = \pi_\sharp \lambda$. As $\lambda$ is dispersed, $\nu$ is concentrated on $(\rcomp{M}\setminus M) \times \set{1} \subset \wt{\rcomp{M}}$. We calculate
%\[
 %\duality{f,\Psi\nu} = \int_{\wt{\rcomp{M}}} \frac{\rcomp{f}(\xi)-\rcomp{f}(\eta)}{\bar{d}(\xi,\eta)} \,d\nu(\xi,\eta) = \int_{\rcomp{M}} \rcomp{f}(\xi)\,d\lambda(\xi), \quad f \in \Lip_0(M).
%\]
%To see that $\Psi\nu \notin \lipfree{\rcomp{M},\bar{d}}$, take points $\xi_1,\ldots,\xi_n \in \rcomp{M}$ and scalars $a_1,\ldots,a_n$, and set $m=\sum_{k=1}^n a_k \delta(\xi_k) \in \lipfree{\rcomp{M},\bar{d}}$. As $\lambda$ is dispersed, given $\varepsilon>0$, there exists a clopen set $U \ni 1,\xi_1,\ldots,\xi_n$ such that $\lambda(U) < \varepsilon$. It is easy to see that $f:=\mathbf{1}_{M\setminus U} \in B_{\Lip_0(M)}$, and thus
%\[
 %\norm{\Psi\nu-m} \geq \duality{f,\Psi\nu} - \duality{\rcomp{f},m} = \lambda(\rcomp{M}\setminus U) > \norm{\lambda} - \varepsilon.
%\]
%Because this holds for all such $m$ and $\varepsilon$, we obtain $\norm{\Psi\nu-m} \geq \norm{\lambda}>0$ for all $m \in \lipfree{\rcomp{M},\bar{d}}$.
\end{remark}

\subsection{Representation of measure-induced functionals}

In this section we establish connections between measure-induced functionals in $\Lip_0(M)^\ast$ and functionals admitting optimal representations on $\bwtfp$, and in so doing provide plenty of examples of the latter. Given the close relation to convex integrals of molecules, we will do so by generalising the results in \cite[Section 3]{APS23}, which show that measure-induced functionals in $\lipfree{M}$ are convex integrals of molecules.

Given a Borel measure $\lambda$ on $\ucomp{M}$, we can formally define a functional $\widehat{\lambda}$ on $\Lip_0(M)$ by
\begin{equation}\label{eq:induced_functional}
\duality{f,\widehat{\lambda}} = \int_{\ucomp{M}}\ucomp{f}\,d\lambda .
\end{equation}
This yields an element of $\Lip_0(M)^*$ precisely when $\lambda$ is concentrated on $\rcomp{M}$ and has finite first $\bar{d}$-moment \cite[Proposition 4.3]{AP_measures}. We then call $\widehat{\lambda}$ the functional \emph{induced by $\lambda$}. Observe that $\delta_0$ induces the null functional, so the content of $\lambda$ at $0$ is inconsequential. Another relevant observation is that under appropriate regularity assumptions (e.g. $\lambda$ being Radon) we have $\widehat{\lambda}\in\lipfree{M}$ if and only if $\lambda$ is concentrated on $M$ \cite[Theorem 4.11]{AP_measures}; this will imply that the results in \cite[Section 3]{APS23} are particular cases of those in this subsection, for functionals belonging to $\lipfree{M}$.

The link between measure-induced functionals in $\Lip_0(M)^*$ and optimal De Leeuw representations on $\bwtfp$ is provided by the Kantorovich-Rubinstein theorem for $\rcomp{M}$ (Corollary \ref{cr:kr_rcomp}). In this context, two finiteness conditions become relevant:
\begin{equation}\label{eq:bwt weighted finite}
\int_{\bwtfp}\frac{1}{d(\zeta)}\,d\mu(\zeta) < \infty
\end{equation}
and
\begin{equation}\label{eq:bwt_marginals_finite_first_moment}
\int_{\bwtfp}\frac{\bar{d}(\pp_1(\zeta),0)}{d(\zeta)}\,d\mu(\zeta) < \infty
\end{equation}
(compare to \cite[(3.1) and (3.2)]{APS23}). We will characterise functionals that admit optimal representations $\mu\in\meas{\bwtfp}$ satisfying one or both of these conditions. Of course, $d(\zeta)$ may be replaced by $\bar{d}(\pp(\zeta))$ in any of those integrals, as $\bar{d}\circ\pp=d$ $\mu$-a.e. if $\mu$ is optimal.

\begin{theorem}[cf. {\cite[Theorem 3.1]{APS23}}]\label{th:bidual_induced}
Let $\phi\in\Lip_0(M)^*$. Then the following are equivalent:
\begin{enumerate}[label={\upshape{(\roman*)}}]
\item\label{bidual_induced-1} $\phi$ has an optimal representation $\mu$ concentrated on $\bwtfp$ satisfying \eqref{eq:bwt weighted finite} and \eqref{eq:bwt_marginals_finite_first_moment},
\item\label{bidual_induced-2} $\phi=\widehat{\lambda}$ for some Radon measure $\lambda\in\meas{\rcomp{M}}$ with finite first $\bar{d}$-moment.
\end{enumerate}
If the above hold, then $\mu$ can be chosen to satisfy $(\pp_1)_\sharp\mu\ll\lambda^+$ and $(\pp_2)_\sharp\mu\ll\lambda^-$, and in particular $(\pp_1)_\sharp\mu \perp (\pp_2)_\sharp\mu$.
\end{theorem}

\begin{proof}
\ref{bidual_induced-1}$\Rightarrow$\ref{bidual_induced-2}: Define a measure $\nu$ on $\bwt{M}$ by $d\nu=d\mu/d$. This measure is finite by \eqref{eq:bwt weighted finite}, and since $\nu\ll\mu$, it is a Radon measure concentrated on $\bwtfp$ (see e.g. \cite[Lemma 7.1.11]{Bogachev}). Let $\lambda_i=(\pp_i)_\sharp \nu$ for $i=1,2$, then $\lambda_i\in\meas{\rcomp{M}}$ have finite first $\bar{d}$-moment. Indeed,
$$
\int_{\rcomp{M}} \bar{d}(\xi,0)\,d\lambda_1(\xi) = \int_{\bwtfp} \bar{d}(\pp_1(\zeta),0)\,d\nu(\zeta) = \int_{\bwtfp} \frac{\bar{d}(\pp_1(\zeta),0)}{d(\zeta)}\,d\mu(\zeta) < \infty
$$
by \eqref{eq:bwt_marginals_finite_first_moment}, and similarly for $\lambda_2$, given that
$$
\int_{\bwtfp}\frac{\bar{d}(\pp_2(\zeta),0)}{d(\zeta)}\,d\mu(\zeta) \leq \int_{\bwtfp}\frac{\bar{d}(\pp_1(\zeta),0)+\bar{d}(\pp(\zeta))}{d(\zeta)}\,d\mu(\zeta) \leq \int_{\bwtfp}\frac{\bar{d}(\pp_1(\zeta),0)}{d(\zeta)}\,d\mu(\zeta) + \norm{\mu} < \infty
$$
as well by Proposition \ref{prop_de_Leeuw_relation} \ref{d_bar_less_than_d}. For every $f\in\Lip_0(M)$ we have
\begin{align*}
\duality{f,\phi} = \int_{\bwt{M}}\Phi f\,d\mu = \int_{\bwtfp}\frac{\rcomp{f}\circ\pp_1-\rcomp{f}\circ\pp_2}{d}\,d\mu = \int_{\ucomp{M}}\ucomp{f}\,d\lambda_1 - \int_{\ucomp{M}}\ucomp{f}\,d\lambda_2 ,
\end{align*}
where both integrals on the right-hand side are finite by finite first $\bar{d}$-moment. Therefore $\phi$ is induced by $\lambda=\lambda_1-\lambda_2\in\meas{\rcomp{M}}$.

\medskip
\ref{bidual_induced-2}$\Rightarrow$\ref{bidual_induced-1}: We assume that $\norm{\lambda^+}=\norm{\lambda^-}=1$ after adding a multiple of $\delta_0$ and rescaling. By \cite[Proposition 4.3]{AP_measures}, $\lambda^+$ and $\lambda^-$ have finite first $\bar{d}$-moment and are concentrated on $\rcomp{M}$. We may then apply Corollary \ref{cr:kr_rcomp} to find an optimal coupling $\pi\in\Pi(\lambda^+,\lambda^-)$ such that
$$
\int_{\rcomp{M}\times\rcomp{M}} \bar{d}\,d\pi = \sup_{f\in B_{\Lip_0(M)}}\int_{\rcomp{M}}\rcomp{f}\,d\lambda = \big\|\widehat{\lambda}\big\|_{\Lip_0(M)^*} = \norm{\phi} ,
$$
and $\pi$ is concentrated on a $\bar{d}$-cyclically monotonic subset of $\rcomp{M}\times\rcomp{M}$.
The restriction of $\pi$ to the diagonal $(\rcomp{M}\times\rcomp{M})\setminus\wt{\rcomp{M}}$ has equal marginals, and these are majorised by the full marginals $(\pp_1)_\sharp\pi=\lambda^+$ and $(\pp_2)_\sharp\pi=\lambda^-$, which are mutually singular. Thus $\pi$ is concentrated on $\wt{\rcomp{M}}$.
Define $\nu\in\meas{\wt{\rcomp{M}}}$ by $d\nu=\bar{d}\,d\pi$, then $\nu$ is concentrated on a $\bar{d}$-cyclically monotonic set. If $\Psi\nu$ is defined as in \eqref{eq:Psi}, then we have for $f\in\Lip_0(M)$
\begin{align*}
\duality{f,\Psi\nu} &= \int_{\rcomp{M}\times\rcomp{M}}\frac{\rcomp{f}(\xi)-\rcomp{f}(\eta)}{\bar{d}(\xi,\eta)}\,d\nu(\xi,\eta) \\
&= \int_{\rcomp{M}\times\rcomp{M}}(\rcomp{f}(\xi)-\rcomp{f}(\eta))\,d\pi(\xi,\eta) \\
&= \int_{\rcomp{M}\times\rcomp{M}}(\rcomp{f}\circ\pp_1)\,d\pi - \int_{\rcomp{M}\times\rcomp{M}}(\rcomp{f}\circ\pp_2)\,d\pi \\
&= \int_{\rcomp{M}}\rcomp{f}\,d\lambda^+ - \int_{\rcomp{M}}\rcomp{f}\,d\lambda^- = \int_{\rcomp{M}}\rcomp{f}\,d\lambda ,
\end{align*}
therefore $\Psi\nu=\widehat{\lambda}=\phi$ (both terms $\int_{\rcomp{M}}\rcomp{f}\,d\lambda^{\pm}$ are finite, which makes the computation valid). By Corollary \ref{cr:optimal_S0}, $\phi$ admits an optimal representation $\mu\in\opr{\bwt{M}}$ concentrated on $\bwtfp$ such that $\pp_\sharp\mu=\nu$. Since $\bar{d}\circ\pp=d$ $\mu$-almost everywhere, we have
$$
\int_{\bwtfp}\frac{d\mu(\zeta)}{d(\zeta)} = \int_{\bwtfp}\frac{d\mu(\zeta)}{\bar{d}(\pp(\zeta))} = \int_{\wt{\rcomp{M}}}\frac{d\nu(\xi,\eta)}{\bar{d}(\xi,\eta)} = \norm{\pi} < \infty , 
$$
and
\begin{align*}
\int_{\bwtfp}\frac{\bar{d}(\pp_1(\zeta),0)}{d(\zeta)}\,d\mu(\zeta) &= \int_{\bwtfp}\frac{\bar{d}(\pp_1(\zeta),0)}{\bar{d}(\pp(\zeta))}\,d\mu(\zeta) = \int_{\wt{\rcomp{M}}}\frac{\bar{d}(\xi,0)}{\bar{d}(\xi,\eta)}\,d\nu(\xi,\eta) \\
&= \int_{\wt{\rcomp{M}}}\bar{d}(\xi,0)\,d\pi(\xi,\eta) = \int_{\rcomp{M}}\bar{d}(\xi,0)\,d\lambda^+(\xi) < \infty 
\end{align*}
is the first $\bar{d}$-moment of $\lambda^+$. That is, the values of \eqref{eq:bwt weighted finite} and \eqref{eq:bwt_marginals_finite_first_moment} are both finite.

For the last statement, observe that \eqref{eq:same marginals} implies $(\pp_1)_\sharp\mu=(\pp_1)_\sharp\nu\ll(\pp_1)_\sharp\pi=\lambda^+$, and similarly $(\pp_2)_\sharp\mu=(\pp_2)_\sharp\nu\ll\lambda^-$. Since $\lambda^+\perp\lambda^-$, we also have $(\pp_1)_\sharp\mu\perp(\pp_2)_\sharp\mu$.
\end{proof}

If we drop condition \eqref{eq:bwt_marginals_finite_first_moment}, we get functionals that are induced by Radon measures in a weak sense. This corresponds to measures on $\rcomp{M}$ with (possibly) infinite first $\bar{d}$-moment.

\begin{theorem}\label{th:bidual_meas_non_ffm}
Let $\phi\in\Lip_0(M)^*$. Then the following are equivalent:
\begin{enumerate}[label={\upshape{(\roman*)}}]
\item\label{bidual_meas_non_ffm-1} $\phi$ has an optimal representation $\mu$ concentrated on $\bwtfp$ satisfying \eqref{eq:bwt weighted finite},
\item\label{bidual_meas_non_ffm-2} $\phi$ avoids infinity, and there exists a Radon measure $\lambda\in\meas{\rcomp{M}}$ such that
\begin{equation}\label{eq:bidual_meas_non_ffm}
\duality{f,\phi} = \int_{\rcomp{M}} \rcomp{f}\,d\lambda
\end{equation}
for every $f\in\Lip_0(M)$ with bounded support.
\end{enumerate}
If the above hold, then \eqref{eq:bidual_meas_non_ffm} holds for every $f\in\Lip_0(M)$ such that $\rcomp{f}$ is $\abs{\lambda}$-integrable, in particular for all bounded $f\in\Lip_0(M)$; furthermore, $\mu$ can be chosen to satisfy
$$
(\pp_1)_\sharp\mu\restrict_{\rcomp{M}\setminus\set{0}}\ll\lambda^+
\qquad\text{and}\qquad
(\pp_2)_\sharp\mu\restrict_{\rcomp{M}\setminus\set{0}}\ll\lambda^-
$$
and $(\pp_1)_\sharp\mu\perp(\pp_2)_\sharp\mu$.
\end{theorem}

Condition \eqref{eq:bidual_meas_non_ffm} implies that $\phi$ avoids 0 but not necessarily infinity, as witnessed e.g. by \cite[Example 4.14]{AP_measures}. Thus, the condition that $\phi$ avoids infinity cannot be removed from \ref{bidual_meas_non_ffm-2}.

\begin{proof}
\ref{bidual_meas_non_ffm-1}$\Rightarrow$\ref{bidual_meas_non_ffm-2}: Corollary \ref{cor_avoid_infinity} implies that $\phi$ avoids infinity. Define a measure $\nu$ on $\bwt{M}$ by $d\nu=d\mu/d$. This measure is finite by \eqref{eq:bwt weighted finite} and, since $\nu\ll\mu$, it is a Radon measure concentrated on $\bwtfp$. Let $\lambda_i=(\pp_i)_\sharp \nu$ for $i=1,2$ and $\lambda=\lambda_1-\lambda_2\in\meas{\rcomp{M}}$. For every $f\in\Lip_0(M)$ with $\rcomp{f}\in L_1(\abs{\lambda})=L_1(\lambda_1)\cap L_1(\lambda_2)$ we have
$$
\duality{f,\phi} = \int_{\bwt{M}}\Phi f\,d\mu = \int_{\bwtfp}\frac{\rcomp{f}\circ\pp_1-\rcomp{f}\circ\pp_2}{d}\,d\mu = \int_{\rcomp{M}}\rcomp{f}\,d\lambda_1 - \int_{\rcomp{M}}\rcomp{f}\,d\lambda_2 .
$$
Since both integrals in the final expression are finite, \eqref{eq:bidual_meas_non_ffm} holds whenever $\rcomp{f}$ is $\abs{\lambda}$-integrable.

\medskip
\ref{bidual_meas_non_ffm-2}$\Rightarrow$\ref{bidual_meas_non_ffm-1}: We assume without loss of generality that $\lambda$ is concentrated on $\rcomp{M}\setminus\set{0}$. Since $\phi$ avoids infinity, it is the norm limit of the functionals $\phi_n=W^*_{\daleth_n}(\phi)\in\Lip_0(M)^*$ defined by
$$
\duality{f,\phi_n} = \duality{f\cdot\daleth_n,\phi} = \int_{\rcomp{M}} \rcomp{f}\cdot\rcomp{\daleth_n}\,d\lambda, 
$$
where $\daleth_n$ is as in \eqref{eq:daleth}. Then $\phi_n=\widehat{\lambda}_n$ where $d\lambda_n=\rcomp{\daleth_n}\,d\lambda$ has $\bar{d}$-bounded support, and therefore finite first $\bar{d}$-moment. We obviously have $\norm{\lambda_n}\leq\norm{\lambda}$. Let
$$
\nu_n=\lambda_n-\lambda_n(\rcomp{M})\cdot\delta_0 ,
$$
then we also have $\phi_n=\widehat{\nu}_n$ and moreover $\nu_n(\rcomp{M})=0$, that is $\norm{\nu_n^+}=\norm{\nu_n^-}$. By Theorem \ref{th:bidual_induced}, $\phi_n$ has an optimal representation $\mu_n\in\opr{\bwt{M}}$ that is concentrated on $\bwtfp$. In fact, the details of the proof show that $d(\pp_\sharp\mu_n)=\bar{d}\cdot d\pi_n$ where $\pi_n\in\meas{\wt{\rcomp{M}}}$ is positive and satisfies $(\pp_1)_\sharp\pi_n=\nu_n^+$ and $(\pp_2)_\sharp\pi_n=\nu_n^-$; in particular, $\norm{\pi_n}=\norm{\nu_n^+}=\norm{\nu_n^-}$.

Since $\norm{\mu_n}=\norm{\phi_n}$ converges to $\norm{\phi}$, there is a $w^*$-cluster point $\mu \in \meas{\bwt{M}}$ of $(\mu_n)$ satisfying $\norm{\mu} \leq \norm{\phi}$. We will show that $\mu$ witnesses \ref{bidual_meas_non_ffm-1}. For any $f \in \Lip_0(M)$, the numbers $\duality{f,\phi_n}=\duality{f,\Phi^*\mu_n}$ both converge to $\duality{f,\phi}$ and, by continuity of $\Phi^*$, cluster at $\duality{f,\Phi^*\mu}$, meaning that $\duality{f,\phi}=\duality{f,\Phi^*\mu}$. Hence $\Phi^*\mu=\phi$ (which implies $\norm{\mu} \geq \norm{\phi}$). It is also clear that $\mu$ is positive. Therefore $\mu$ is an optimal representation of $\phi$ and, by Corollary \ref{cor_avoid_infinity}, $\mu$ is concentrated on $\bwtf$.

To establish the remaining properties of $\mu$, we interrupt proceedings to prove the following claim.

\begin{claim}
Let $U,V \subset \rcomp{M}$ be open and let $r>0$ such that $\bar{d}(0,\xi) \leqslant r$ whenever $\xi \in U \cup V$. Then
\begin{equation}\label{eq:claim}
 \mu(\pp^{-1}(U \times V)) \leq 2r\pare{\lambda^+(U) + \max\set{0,-\lambda(\rcomp{M})}\delta_0(U)}.
\end{equation}
\end{claim}

Indeed, we estimate
\[
 \mu_n(\pp^{-1}(U \times V)) = \int_{U \times V} \bar{d}(\xi,\eta)\,d\pi_n(\xi,\eta) \leq 2r\pi_n(U \times V) \leq 2r\pi_n(\pp_1^{-1}(U)) = 2r\nu_n^+(U).
\]
Because $\lambda_n$ is concentrated on $\rcomp{M}\setminus\set{0}$, we calculate $\nu_n^+ = \lambda_n^+ + \max\set{0,-\lambda_n(\rcomp{M})}\delta_0$. Hence
\[
 \mu_n(\pp^{-1}(U \times V)) \leq 2r\pare{\lambda_n^+(U) + \max\set{0,-\lambda_n(\rcomp{M})}\delta_0(U)}.
\]
By the dominated convergence theorem, the right-hand side of the line above converges to that of \eqref{eq:claim}. On the other hand, as $\pp^{-1}(U \times V)$ is open, the map $\omega \in \meas{\bwt{M}} \mapsto \omega(\pp^{-1}(U \times V))$ is $w^*$-lower semicontinuous on $\meas{\bwt{M}}$ by \cite[Corollary 8.2.4 (b)]{Bogachev}. Therefore, as $\mu$ is a $w^*$-cluster point of the $\mu_n$, \eqref{eq:claim} must hold. This completes the proof of the Claim.

Now we return to the main proof and use the Claim to show that $(\pp_1)_\sharp\mu\restrict_{\rcomp{M}\setminus\set{0}}\ll\lambda^+$. To this end, let $K \subset \rcomp{M}$ be compact and satisfy $\lambda^+(K)=0$, and fix $\varepsilon>0$. Using inner regularity, Proposition \ref{prop_rcomp_observations} \ref{compact_bounded} and Proposition \ref{pr:bar-d_cont} \ref{bar-d_cont_3}, find $r>0$ such that $V:=\set{\xi \in \rcomp{M} \,:\, \bar{d}(0,\xi) < r}$ contains $K$ and
\[
\mu(\pp^{-1}(K\times V)) \geq \mu(\pp_1^{-1}(K))-\varepsilon.
\]
As $\lambda^+(K)=0$ and $V$ is open, we may choose an open set $U\subset V$ containing $K$, such that $\lambda^+(U)\leq \varepsilon/2r$. Hence by \eqref{eq:claim} in the Claim, we obtain
\begin{align*}
 \mu(\pp^{-1}(K\times V)) \leq \mu(\pp^{-1}(U\times V)) &\leq 2r\pare{\lambda^+(U) + \max\set{0,-\lambda(\rcomp{M})}\delta_0(U)}\\
 &\leq \varepsilon + \max\set{0,-2r\lambda(\rcomp{M})}\delta_0(U),
\end{align*}
whence
\begin{equation}\label{eqn:bidual_meas_non_ffm}
\mu(\pp_1^{-1}(K)) \leq 2\varepsilon + \max\set{0,-2r\lambda(\rcomp{M})}\delta_0(U). 
\end{equation}
Now assume $0 \not\in K$. By considering $U\setminus\set{0}$ if necessary, \eqref{eqn:bidual_meas_non_ffm} implies $\mu(\pp_1^{-1}(K)) \leq 2\varepsilon$. As this holds for all $\varepsilon>0$ we conclude $(\pp_1)_\sharp\mu(K)=\mu(\pp_1^{-1}(K))=0$. That $(\pp_1)_\sharp\mu\restrict_{\rcomp{M}\setminus\set{0}}\ll\lambda^+$ now follows by appealing to inner regularity. A similar argument shows that $(\pp_2)_\sharp\mu\restrict_{\rcomp{M}\setminus\set{0}}\ll\lambda^-$.

To see that $(\pp_1)_\sharp\mu$ and $(\pp_2)_\sharp\mu$ are mutually singular, it is now enough to check that only one of them can have a non-zero mass at $0$. Suppose first that $\lambda(\rcomp{M})\geq 0$. Let $\varepsilon>0$. As $\lambda^+(\set{0})=0$, we can apply the argument above to $K=\set{0}$ and, by considering suitable $r,V$ and $U$, by \eqref{eqn:bidual_meas_non_ffm} we have $\mu(\pp_1^{-1}\set{0}) \leq 2\varepsilon$, so we conclude that $(\pp_1)_\sharp\mu(\set{0})=0$. If $\lambda(\rcomp{M})\leq 0$ instead, a similar argument yields $(\pp_2)_\sharp\mu(\set{0})=0$.

We can now show that $\mu$ is concentrated on $\bwtfp$. Because $(\pp_1)_\sharp\mu\perp(\pp_2)_\sharp\mu$, there exists a Borel set $C\subset\rcomp{M}$ such that $(\pp_1)_\sharp\mu(C)=0$ and $(\pp_2)_\sharp\mu(\rcomp{M}\setminus C)=0$. Using the notation
$$
\Delta(E) := \set{\zeta\in\bwt{M} \,:\, \pp_1(\zeta)=\pp_2(\zeta)\in E}
$$
for $E\subset\rcomp{M}$, we have
$$
\bwtf\setminus\bwtfp = \Delta(\rcomp{M}) = \Delta(C) \cup \Delta(\rcomp{M}\setminus C).
$$
But $\mu(\Delta(C)) \leq \mu(\pp_1^{-1}(C)) = (\pp_1)_\sharp \mu(C) =0$, and similarly $\mu(\Delta(\rcomp{M}\setminus C))\leq \mu(\pp_2^{-1}(\rcomp{M}\setminus C)) = (\pp_2)_\sharp \mu(\rcomp{M}\setminus C) =0$, which implies that $\mu(\bwtf\setminus\bwtfp)=0$.

Finally, we check \eqref{eq:bwt weighted finite}. Put $\psi(\zeta) = 1/d(\zeta)$ for $\zeta\in\bwtfp$. For all $k,n\in\NN$ we have
$$
\int_{\bwtfp}\min\set{\psi,k}\,d\mu_n \leq \int_{\bwtfp}\psi(\zeta)\,d\mu_n(\zeta) = \int_{\wt{\rcomp{M}}}\,d\pi_{n} = \norm{\pi_n} = \norm{\nu_n^+} \leq \norm{\lambda_n} \leq \norm{\lambda} .
$$
Above, again we have used the fact that, as $\mu_n$ is optimal, $\bar{d} \circ \pp = d$ $\mu_n$-a.e. on $\bwtfp$.
Because $\min\set{\psi,k}$ is a continuous function on $\bwt{M}$, the numbers $\int_{\bwt{M}}\min\set{\psi,k}\,d\mu_n \leq \norm{\lambda}$ cluster at $\int_{\bwt{M}}\min\set{\psi,k}\,d\mu$ and, by Lebesgue's monotone convergence theorem,
$$
\int_{\bwtfp}\psi\,d\mu = \lim_k\int_{\bwtfp}\min\set{\psi,k}\,d\mu \leq \lim_k\int_{\bwt{M}}\min\set{\psi,k}\,d\mu \leq \norm{\lambda} < \infty
$$
as required.
\end{proof}

The counterpart to Theorem \ref{th:bidual_meas_non_ffm} for functionals in $\lipfree{M}$ is not included in \cite{APS23}, so we state it explicitly here for completeness. It follows straightforwardly from Theorem \ref{th:bidual_meas_non_ffm} and its proof.

\begin{corollary}
Let $m\in\lipfree{M}$. Then the following are equivalent:
\begin{enumerate}[label={\upshape{(\roman*)}}]
\item $m$ is a convex integral of molecules with a representation $\mu\in\opr{\bwt{M}}$ concentrated on $\wt{M}$ such that
$$
\int_{\wt{M}}\frac{1}{d(x,y)}\,d\mu(x,y)<\infty .
$$
\item There exists a Radon measure $\lambda\in\meas{M}$ such that
\begin{equation}\label{eq:meas_non_ffm}
\duality{f,m} = \int_M f\,d\lambda
\end{equation}
for every $f\in\Lip_0(M)$ with bounded support.
\end{enumerate}
If the above hold, then \eqref{eq:meas_non_ffm} holds for every $f\in\Lip_0(M)$ that is $\abs{\lambda}$-integrable, in particular for all bounded $f\in\Lip_0(M)$; furthermore, $\mu$ can be chosen to satisfy
$$
(\pp_1)_\sharp\mu\restrict_{M\setminus\set{0}}\ll\lambda^+
\qquad\text{and}\qquad
(\pp_2)_\sharp\mu\restrict_{M\setminus\set{0}}\ll\lambda^-
$$
and $(\pp_1)_\sharp\mu\perp(\pp_2)_\sharp\mu$.
\end{corollary}

Finally, if we keep condition \eqref{eq:bwt_marginals_finite_first_moment} and drop \eqref{eq:bwt weighted finite} instead, we get \emph{majorisable} functionals, that is, functionals that can be expressed as the difference between two positive functionals on $\Lip_0(M)$. Positive functionals are also essentially measure-induced, in this case by a measure that may be $\sigma$-finite, with a singularity at $0$. Thus, majorisable functionals are induced by the difference of two such measures (which is not formally a measure); see \cite[Theorem 5.9]{AP_measures} for details.

\begin{theorem}[cf. {\cite[Theorem 3.3]{APS23}}]\label{th:bidual_majorisable}
Let $\phi\in\Lip_0(M)^*$. Then the following are equivalent:
\begin{enumerate}[label={\upshape{(\roman*)}}]
\item\label{th:bidual_majorisable-1} $\phi$ has an optimal representation $\mu$ concentrated on $\bwtfp$ satisfying \eqref{eq:bwt_marginals_finite_first_moment},
\item\label{th:bidual_majorisable-2} $\phi$ is majorisable and avoids 0 and infinity.
\end{enumerate}
If the above hold, then $\mu$ can be chosen to satisfy $(\pp_1)_\sharp\mu\restrict_{\rcomp{M}\setminus\set{0}}\perp(\pp_2)_\sharp\mu\restrict_{\rcomp{M}\setminus\set{0}}$.
\end{theorem}

\begin{proof}
Throughout the proof, we denote $A_n=\set{\xi\in\rcomp{M} \,:\, 2^{-n}\leq\rcomp{\rho}(\xi)\leq 2^n}$. Observe that $\bigcup_{n=1}^\infty A_n=\rcomp{M}\setminus\set{0}$.

\medskip
\ref{th:bidual_majorisable-1}$\Rightarrow$\ref{th:bidual_majorisable-2}: Define a positive, possibly infinite, Borel measure $\nu$ on $\bwt{M}$ by
$$
d\nu(\zeta) = \frac{1}{d(\zeta)}\,d\mu\restrict_{\pp_1^{-1}(\rcomp{M}\setminus\set{0})}(\zeta) \,.
$$
Then, for any $n\in\NN$
$$
\nu(\pp_1^{-1}(A_n)) \leq 2^n\int_{\pp_1^{-1}(A_n)}\frac{\bar{d}(\pp_1(\zeta),0)}{d(\zeta)}\,d\mu(\zeta) < \infty
$$
by \eqref{eq:bwt_marginals_finite_first_moment}, so $\nu$ is $\sigma$-finite. Let $\lambda=(\pp_1)_\sharp\nu$, that is $\lambda(E)=\nu(\pp_1^{-1}(E))$ for Borel $E\subset\rcomp{M}$; while the pushforward is usually defined for finite measures, it is clear that it defines a positive, countably additive set function even in this case. Clearly $\lambda$ is concentrated on $\rcomp{M}\setminus\set{0}$ and $\lambda(A_n)<\infty$ for any $n$, thus $\lambda$ is $\sigma$-finite, and so it follows (e.g. by Lebesgue's monotone convergence theorem) that the formula
$$
\int_{\rcomp{M}}\varphi\,d\lambda = \int_{\pp_1^{-1}(\rcomp{M}\setminus\set{0})}(\varphi\circ\pp_1)\,d\nu = \int_{\pp_1^{-1}(\rcomp{M}\setminus\set{0})}\frac{\varphi(\pp_1(\zeta))}{d(\zeta)}\,d\mu(\zeta)
$$
also holds in this case for any positive Borel function $\varphi:\rcomp{M}\to\RR$. In particular, \eqref{eq:bwt_marginals_finite_first_moment} is equal to
$$
\int_{\bwtfp}\frac{\bar{d}(\pp_1(\zeta),0)}{d(\zeta)}\,d\mu(\zeta) = \int_{\bwtfp} \bar{d}(\pp_1(\zeta),0) \,d\nu(\zeta) = \int_{\rcomp{M}}\bar{d}(\xi,0)\,d\lambda(\xi)
$$
so $\lambda$ has finite first $\bar{d}$-moment. By \cite[Proposition 4.3]{AP_measures}, $\lambda$ induces a positive functional $\widehat{\lambda}\in\Lip_0(M)^*$. If $f\in\Lip_0(M)$ is such that $f\geq 0$ pointwise, then $\rcomp{f}\geq 0$ as well and
$$
\duality{f,\phi} = \int_{\bwtfp}\Phi f\,d\mu \leq \int_{\bwtfp}\frac{\rcomp{f}(\pp_1(\zeta))}{d(\zeta)}\,d\mu(\zeta) = \int_{\rcomp{M}}\rcomp{f}\,d\lambda = \duality{f,\widehat{\lambda}} .
$$
Therefore $\widehat{\lambda}-\phi$ is a positive functional as well, and $\phi = \widehat{\lambda} - (\widehat{\lambda}-\phi)$ is majorisable. The fact that $\phi$ avoids 0 and infinity follows from Corollary \ref{cor_avoid_infinity} and Proposition \ref{pr:avoid_0}.

\medskip
\ref{th:bidual_majorisable-2}$\Rightarrow$\ref{th:bidual_majorisable-1}: By Theorem 5.14 and Proposition 4.3 in \cite{AP_measures}, there exist positive $\sigma$-finite Borel measures $\lambda^+,\lambda^-$ on $\rcomp{M}$ that have finite first $\bar{d}$-moments, are Radon when restricted to any closed $K\subset\rcomp{M}$ not containing $0$, are concentrated on disjoint Borel subsets $C^+,C^-$ of $\rcomp{M}$, and such that
$$
\duality{f,\phi} = \int_{\rcomp{M}}\rcomp{f}\,d\lambda^+ - \int_{\rcomp{M}}\rcomp{f}\,d\lambda^-
$$
for every $f\in\Lip_0(M)$. Since $\phi$ avoids 0 and infinity, it is the norm limit of the functionals $\phi_n=W^*_{\Pi_n}(\phi)\in\Lip_0(M)^*$ defined by
$$
\duality{f,\phi_n} = \duality{f\cdot\Pi_n,\phi} = \int_{\rcomp{M}} \rcomp{f}\rcomp{\Pi_n}\,d\lambda^+ - \int_{\rcomp{M}} \rcomp{f}\rcomp{\Pi_n}\,d\lambda^- ,
$$
where $\Pi_n$ is as in \eqref{eq:Pi}. This means that $\phi_n=\widehat{\lambda}_n$ where $\lambda_n$ is given by
$$
d\lambda_n=\rcomp{\Pi_n}\,d\lambda^+-\rcomp{\Pi_n}\,d\lambda^- .
$$
Note that $\lambda_n$ is well defined and belongs to $\meas{\rcomp{M}}$ because $0\notin\supp(\rcomp{\Pi_n})$. Clearly $\lambda_n$ has finite first $\bar{d}$-moment, and $d\lambda_n^\pm = \rcomp{\Pi_n}\,d\lambda^\pm$.

The rest of the proof is very similar to that of Theorem \ref{th:bidual_meas_non_ffm}, so we only sketch the argument. Taking $\nu_n=\lambda_n-\lambda_n(\rcomp{M})\cdot\delta_0$, we obtain optimal representations $\mu_n\in\opr{\bwt{M}}$ of $\phi_n$ that are concentrated on $\bwtfp$ and satisfy $d(\pp_\sharp\mu_n)=\bar{d}\cdot d\pi_n$, where $\pi_n\in\meas{\wt{\rcomp{M}}}$ is positive and satisfies $(\pp_1)_\sharp\pi_n=\nu_n^+$ and $(\pp_2)_\sharp\pi_n=\nu_n^-$.
The sequence $(\mu_n)$ has a weak$^*$ cluster point $\mu$ that is an optimal representation of $\phi$ concentrated on $\bwtf$. Following a similar argument as in Theorem \ref{th:bidual_meas_non_ffm} but with an open set $U\subset\supp(\rcomp{\Pi}_k)\subset A_{k+1}$ in this case, we can show that $(\pp_1)_\sharp\mu(K)=0$ whenever $K \subset \rcomp{M}\setminus\set{0}$ is compact and satisfies $\lambda^+(K)=0$. This implies $(\pp_1)_\sharp\mu\restrict_{\rcomp{M}\setminus\set{0}}\ll\lambda^+$ by inner regularity. Likewise $(\pp_2)_\sharp\mu\restrict_{\rcomp{M}\setminus\set{0}}\ll\lambda^-$. It follows that $(\pp_1)_\sharp\mu$ and $(\pp_2)_\sharp\mu$ are concentrated on $C^+\cup\set{0}$ and $C^-\cup\set{0}$, respectively, and that the restrictions of $(\pp_1)_\sharp\mu$ and $(\pp_2)_\sharp\mu$ to $\rcomp{M}\setminus\set{0}$ are mutually singular. Moreover, when combined with the fact that $\phi$ avoids $0$ and Proposition \ref{pr:avoid_0}, this also implies that $\mu$ is concentrated on $\bwtfp$.
Finally, to prove that $\mu$ satisfies \eqref{eq:bwt_marginals_finite_first_moment}, we put
$$
\psi(\zeta) = \frac{\bar{d}(\pp_1(\zeta),0)}{d(\zeta)}
$$
for $\zeta\in\bwtfp$. For all $k,n\in\NN$ we have
\begin{align*}
\int_{\bwtfp}\min\set{\psi,k}\,d\mu_n &\leq \int_{\bwtfp}\psi(\zeta)\,d\mu_n(\zeta) = \int_{\wt{\rcomp{M}}}\bar{d}(\xi,0)\,d\pi_{n}(\xi,\eta) \\
&= \int_{\rcomp{M}}\bar{d}(\xi,0)\,d\nu^+_n(\xi) 
= \int_{\rcomp{M}}\bar{d}(\xi,0)\,d\lambda^+_n(\xi) \\
&= \int_{\rcomp{M}}\bar{d}(\xi,0)\,\rcomp{\Pi}_n(\xi)\,d\lambda^+(\xi) \leq \int_{\rcomp{M}}\bar{d}(\xi,0)\,d\lambda^+(\xi) .
\end{align*}
Since $\min\set{\psi,k}$ is a continuous function on $\bwt{M}$ by Proposition \ref{pr:bar-d_cont} \ref{bar-d_cont_3}, we get
$$
\int_{\bwtfp}\psi\,d\mu \leq \int_{\rcomp{M}}\bar{d}(\xi,0)\,d\lambda^+(\xi) < \infty
$$
using the same argument as in Theorem \ref{th:bidual_meas_non_ffm}.
\end{proof}

To summarise, conditions \eqref{eq:bwt weighted finite} and \eqref{eq:bwt_marginals_finite_first_moment} are dual to each other. Both imply that the functional $\Phi^*\mu$ essentially acts on $\Lip_0(M)$ by integration against a measure $\lambda$ on $\rcomp{M}$. The former condition is equivalent to $\lambda$ being finite, i.e. ``no singularity at $0$'', and the latter is equivalent to $\lambda$ having finite first moment, i.e. ``no singularity at infinity''.

\section{Supports of functionals and De Leeuw representations}\label{sect_relations_between_supports}

In this section we extend the results in \cite{Aliaga} on minimal-support representations (Lemma 8 and Proposition 9), that applied to elements of $\lipfree{M}$ and subsets of $M$, to elements of $\Lip_0(M)^*$ and subsets of $\ucomp{M}$. We begin by recalling the definition of support of an element of a Lipschitz-free space.

Recall that the Lipschitz-free space over a subset $N\subset M$ containing $0$ can be identified with a subspace of $\lipfree{M}$, namely $\lipfree{N}=\cl{\lspan}\,\delta(N)$. The \emph{support} $\supp(m)$ of a functional $m\in\lipfree{M}$ is defined to be the smallest closed subset $N$ of $M$ such that $m\in\lipfree{N\cup\set{0}}$  \cite{AP_rmi,APPP_2020}. This set is always separable, and a point $x\in M$ belongs to $\supp(m)$ if and only if every neighbourhood of $x$ contains the support of some $f\in\Lip_0(M)$ such that $\duality{f,m}\neq 0$ \cite[Proposition 2.7]{APPP_2020}. According to \cite[Proposition 2.6]{APPP_2020}, the value of $\duality{f,m}$ for $f\in\Lip_0(M)$ depends only on the restriction of $f$ to $\supp(m)$. This means that we may define supports equivalently as
\[
\supp(m) = \bigcap\set{A: \text{$A\subset M$ is closed and $\duality{f,m}=0$ whenever $f\in\Lip_0(M)$ and $f\restrict_A=0$}} .
\]
It is intuitively clear that, in some sense, the support of any De Leeuw representation of a Lipschitz-free space element needs to cover the support of the functional it represents. Given $E \subset \bwt{M}$, we define the \emph{shadow} $\pp_s(E)$ of $E$ to be the set $\pp_1(E) \cup \pp_2(E)$.

\begin{proposition}[{\cite[Lemma 8]{Aliaga}}]
\label{pr:aliaga_lemma_8}
If $m\in\lipfree{M}$, then $\supp(m)\subset\pp_s(\supp(\mu))$ for any De Leeuw representation $\mu$ of $m$.
\end{proposition}

By Proposition \ref{pr:aliaga_lemma_8} we always have $\supp(\Phi^*\mu)\subset\pp_s(\supp(\mu))$ when $\Phi^*\mu\in\lipfree{M}$, which provides a ``lower bound'' on $\supp(\mu)$. This lower bound extends in full generality to elements of $\Lip_0(M)^*$; see Proposition \ref{pr:support_inclusion}. To see this, we need an extension of the notion of support for elements of $\lipfree{M}$ to elements of $\Lip_0(M)^*$.

\begin{definition}[{\cite[Definition 3.5]{AP_measures}}]\label{defn_extended_support}
For any $\phi\in\Lip_0(M)^*$ we define the \emph{extended support} $\esupp{\phi}$ as the following compact subset of $\ucomp{M}$:
\[
\esupp{\phi} = \bigcap\set{\ucl{A}: \text{$A\subset M$ and $\duality{f,\phi}=0$ whenever $f\in\Lip_0(M)$ and $f\restrict_A=0$}}.
\]
\end{definition}

Equivalently, $\zeta\in\esupp{\phi}$ if and only if, for every neighbourhood $U$ of $\zeta$, there exists $f\in\Lip_0(M)$ such that $\duality{f,\phi}\neq 0$ and $\supp(f)\subset U\cap M$ \cite[Proposition 3.6]{AP_measures}. The definition is consistent with that of the support in $\lipfree{M}$ in the sense that, if $\phi\in\lipfree{M}$, then $\supp(\phi)=\esupp{\phi}\cap M$ and $\esupp{\phi}=\ucl{\supp(\phi)}$ \cite[Corollary 3.7]{AP_measures}. Extended supports behave well with elements $\phi\in\Lip_0(M)^*$ that avoid infinity (and only with them): if $\phi$ avoids infinity, $U\subset\ucomp{M}$ is a neighbourhood of $\esupp{\phi}$ and $f\in\Lip_0(M)$ vanishes on $U\cap M$, then $\duality{f,\phi}=0$ \cite[Theorem 3.13]{AP_measures}.

This first proposition extends Proposition \ref{pr:aliaga_lemma_8}.

\begin{proposition}
\label{pr:support_inclusion}
If $\mu\in\meas{\bwt{M}}$ then $\esupp{\dual{\Phi}\mu}\subset\pp_s(\supp(\mu))$.
\end{proposition}

\begin{proof}
Suppose $\xi\notin\pp_s(\supp(\mu))$. Since $\supp(\mu)$ is compact, so is $\pp_s(\supp(\mu))$, hence there is an open neighbourhood $U$ of $\xi$ such that $\ucl{U}$ does not intersect $\pp_s(\supp(\mu))$. Let $f\in\Lip_0(M)$ be such that $\supp(f)\subset U\cap M$; we will show that $\duality{f,\Phi^*\mu}=0$ which will complete the proof by \cite[Proposition 3.6]{AP_measures}.
Indeed, the set
$$
V = \pp_1^{-1}(\ucl{U}) \cup \pp_2^{-1}(\ucl{U})
$$
is a compact subset of $\bwt{M}$ and clearly does not intersect $\supp(\mu)$. Thus, if $\zeta\in\supp(\mu)$ then there is a neighbourhood $W$ of $\zeta$ such that $W\cap V=\varnothing$, so if $(x,y)\in\widetilde{M}\cap W$ then $x,y\notin U$, hence $f(x)=f(y)=0$ and by continuity $\Phi f(\zeta)=0$. This shows that $\Phi f=0$ on $\supp(\mu)$ and therefore $\duality{f,\Phi^*\mu}=0$ as claimed.
\end{proof}

If $\phi$ avoids infinity, then there is at least one representing measure $\mu$ with ``minimal support'' in the sense that equality is attained in Proposition \ref{pr:support_inclusion}, up to possibly the base point.

\begin{proposition}
\label{pr:minimal_support_representation}
If $\phi\in\Lip_0(M)^*$ avoids infinity then there is $\mu\in\opr{\bwt{M}}$ such that $\dual{\Phi}\mu=\phi$ and $\esupp{\phi} \subset \pp_s(\supp(\mu)) \subset \esupp{\phi}\cup\set{0}$.
\end{proposition}

\begin{proof}
The first inclusion is given by Proposition \ref{pr:support_inclusion}. For the second one, set $A=\esupp{\phi}\cup\set{0}$. Let $U$ be an open subset of $\ucomp{M}$ such that $A\subset U$, and denote $W=\cl{U\cap M}$ (where the closure is taken in $M$). Define a functional $\phi'\in\Lip_0(W)^*$ by setting
\[
\text{$\duality{f,\phi'} := \duality{F,\phi}$ where $F\in\Lip_0(M)$ is an extension of $f\in\Lip_0(W)$.}
\]
By \cite[Theorem 3.13]{AP_measures}, this is a well-defined value, i.e. it is independent of the choice of $F$. Linearity of $\phi'$ is obvious, and we easily get $\norm{\phi'}=\norm{\phi}$ by forcing $\lipnorm{F}=\lipnorm{f}$. Choose an optimal representation $\mu'\in\opr{\bwt{W}}$ of $\phi'$, that is, such that $(\Phi\restrict_{\Lip_0(W)})^*\mu'=\phi'$. It is a standard fact that $\bwt{W}$ can be identified with the closure $\cl{\wt{W}}^{\bwt{M}}$ of $\wt{W}$ in $\bwt{M}$, and so we may identify $\mu'$ with a measure $\mu_U\in\meas{\bwt{M}}$ that is supported on
\[
\cl{\wt{W}}^{\bwt{M}} \subset \pp^{-1}(\ucl{W}\times\ucl{W}) = \pp^{-1}(\ucl{U}\times\ucl{U})
\]
such that $\norm{\mu_U}=\norm{\mu'}=\norm{\phi'}=\norm{\phi}$. Given $f\in\Lip_0(M)$, we have
\[
\duality{f,\Phi^*\mu_U} = \int_{\bwt{M}}(\Phi f)\,d\mu_U = \int_{\cl{\wt{W}}^{\bwt{M}}}(\Phi f)\,d\mu_U = \int_{\bwt{W}}\Phi(f\restrict_W)\,d\mu' = \duality{f\restrict_W,\phi'} = \duality{f,\phi}
\]
whence $\mu_U$ is an optimal representation of $\phi$.

Now consider the net $(\mu_U)$ where $U$ runs over the open subsets of $\ucomp{M}$ containing $A$, directed by reverse inclusion. Since $(\mu_U)$ is bounded, it has a weak$^\ast$ convergent subnet with limit $\mu$. Clearly, $\mu$ is also an optimal representation of $\phi$, and it is straightforward to conclude that
$$
\supp(\mu) \subset \bigcap_U \pp^{-1}(\ucl{U}\times\ucl{U}) = \pp^{-1}(A\times A)
$$
by the normality of $\ucomp{M}$. That is, $\pp_s(\supp(\mu))\subset A$.
\end{proof}

The appearance of the base point in Proposition \ref{pr:minimal_support_representation} might look like an artifact of the method of proof, but it is in fact necessary: there are functionals $\phi$ for which the equality $\pp_s(\supp(\mu))=\esupp{\phi}$ is unattainable, e.g. $\phi=\delta(p)$ for $p\in M\setminus\set{0}$. This follows from the next result (the implication (i)$\Rightarrow$(ii) more specifically), which shows that, under a mild condition, the shadow of the support is the same for all ``minimal'' $\mu$ (in the sense of Proposition \ref{pr:minimal_support_representation}) and depends only on $\phi$: either all such $\mu$ cover the base point or none of them do.

\begin{proposition}
\label{pr:min_shadow_base_point}
Let $\phi\in\Lip_0(M)^*$ avoid infinity, and let $\mu$ be an optimal representation of $\phi$ such that $\pp_s(\supp(\mu))\subset\esupp{\phi}\cup\set{0}$. Suppose that there are no $\xi,\eta\in\esupp{\phi}\cap\rcomp{M}$ such that $0\in \rcomp{[\xi,\eta]}$. Then the following are equivalent:
\begin{enumerate}[label={\upshape{(\roman*)}}]
\item\label{min_shadow_base_point-shadow} $0\notin\pp_s(\supp(\mu))$;
\item\label{min_shadow_base_point-funct} $\duality{f,\phi}=0$ for some (and thus all) $f\in\Lip_0(M)$ such that $f\restrict_{U\cap M}=1$ for some open $U\subset\ucomp{M}$ containing $\esupp{\phi}$.
\end{enumerate}
\end{proposition}

\begin{proof}
Notice first that the hypothesis implies $0\notin\esupp{\phi}$, and that $\duality{f,\phi}$ only depends on the values of $f$ on $U\cap M$ for any neighbourhood $U$ of $\esupp{\phi}$ by \cite[Theorem 3.13]{AP_measures}. Consequently, if \ref{min_shadow_base_point-funct} holds for some $U$ and $f$ then it holds for all $U$ and $f$ (which exists if and only if $0\notin\ucl{U}$).

\ref{min_shadow_base_point-shadow}$\Rightarrow$\ref{min_shadow_base_point-funct}: Assume \ref{min_shadow_base_point-shadow} and let $f$ be as in \ref{min_shadow_base_point-funct} for some $U$. Then $\Phi f$ vanishes on $\pp^{-1}(\esupp{\phi}\times\esupp{\phi})\supset\supp(\mu)$ and so $\duality{f,\phi}=\int_{\bwt{M}}(\Phi f)\,d\mu=0$.

\ref{min_shadow_base_point-funct}$\Rightarrow$\ref{min_shadow_base_point-shadow}: Suppose that \ref{min_shadow_base_point-funct} holds but \ref{min_shadow_base_point-shadow} fails, that is $\pp_s(\supp(\mu))=\esupp{\phi}\cup\set{0}$. Then $\mu$ is not concentrated on $\pp^{-1}(\esupp{\phi}\times\esupp{\phi})$. Since $0\notin\esupp{\phi}$, $\phi$ avoids 0 by \cite[Proposition 3.18]{AP_measures} and Proposition \ref{pr:avoid_0} yields $\mu(\pp^{-1}(0,0))=0$. As a consequence, at least one of the sets $\pp^{-1}(\esupp{\phi}\times\set{0})$, $\pp^{-1}(\set{0}\times\esupp{\phi})$ must have positive $\mu$-measure. Now recall that, by Corollary \ref{cor_avoid_infinity} and Theorem \ref{thm_minimal_1}, $\mu$ is concentrated on a set $C\subset\bwtf$ such that $\pp(C)$ is $\bar{d}$-cyclically monotonic. Thus, for $f$ as in \ref{min_shadow_base_point-funct}
\begin{align*}
0 = \duality{f,\phi} &= \int_{\supp({\mu})} (\Phi f)\,d\mu\\
&=\int_{\pp^{-1}(\esupp{\phi}\times\set{0})} (\Phi f)\,d\mu + \int_{\pp^{-1}(\set{0}\times\esupp{\phi})} (\Phi f)\,d\mu + \int_{\pp^{-1}(\esupp{\phi}\times\esupp{\phi})} (\Phi f)\,d\mu\\
&=\int_{C\cap\pp^{-1}(\esupp{\phi}\times\set{0})}\frac{1}{d(\zeta)}\,d\mu(\zeta) - \int_{C\cap\pp^{-1}(\set{0}\times\esupp{\phi})} \frac{1}{d(\zeta)}\,d\mu(\zeta),
\end{align*}
where the last equality follows from the choice of $f$ and Proposition \ref{prop_de_Leeuw_relation} \ref{delta_diff}. Moreover, in this equation we may replace $\esupp{\phi}$ with $\esupp{\phi}\cap\rcomp{M}$ because $\mu$ is concentrated on $\bwtf$. Since $d<\infty$ on $\bwtf$, we conclude that at least one of the integrals in the last line must be non-zero, thus both are. But then $C$ intersects both $\pp^{-1}((\esupp{\phi}\cap\rcomp{M})\times\set{0})$ and $\pp^{-1}(\set{0}\times(\esupp{\phi}\cap\rcomp{M}))$, that is, there exist $\xi,\eta\in\esupp{\phi}\cap\rcomp{M}$ with $(\xi,0),(0,\eta)\in\pp(C)$. By cyclical monotonicity we conclude $0\in\rcomp{[\xi,\eta]}$ (see Corollary \ref{cor_minimal_4}), which contradicts our assumption on $\esupp{\phi}$. This ends the proof.
\end{proof}

Using the results in this section, we can easily extend Proposition \ref{pr:optimals_everywhere} about existence of optimal points to the case where both coordinates are equal. In general, sets $\pp^{-1}(\xi,\xi)$ for $\xi\in\rcomp{M}$ can be empty (see Proposition \ref{prop_p_range}), or they can consist entirely of vanishing points (see the example at the beginning of Section \ref{subsec_vanishing}). But in all other cases, they must contain optimal points.

\begin{proposition}
\label{pr:optimals_everywhere_2}
Let $\xi\in\rcomp{M}$ and suppose that $\pp^{-1}(\xi,\xi)$ contains a non-vanishing point. Then it contains an optimal point.
\end{proposition}

\begin{proof}
By assumption, there exists $\zeta\in\pp^{-1}(\xi,\xi)$ such that $\phi=\Phi^*\delta_\zeta$ is non-zero. Then $\esupp{\phi}\subset\set{\xi}$ by Proposition \ref{pr:support_inclusion}. Since $\phi$ avoids infinity, e.g. by Corollary \ref{cor_avoid_infinity}, Proposition \ref{pr:minimal_support_representation} yields an optimal representation $\mu$ of $\phi$ such that $\pp_s(\supp(\mu))\subset\esupp{\phi}\cup\set{0}\subset\set{0,\xi}$. If $\xi\neq 0$, note that the hypotheses of Proposition \ref{pr:min_shadow_base_point} are satisfied and thus $0\notin\pp_s(\supp(\mu))$. Thus, in any case we get $\pp_s(\supp(\mu))\subset\set{\xi}$, i.e. $\mu$ is concentrated on $\pp^{-1}(\xi,\xi)$. By Proposition \ref{prop_conc_on_O}, $\mu$ is also concentrated on a set of optimal points, thus there are optimal points in $\pp^{-1}(\xi,\xi)$ as $\mu\neq 0$.
\end{proof}

As another application of Proposition \ref{pr:minimal_support_representation}, together with a recent paper on non-separable growths of $\NN$ \cite{B_NZ20}, we can provide an almost complete answer to \cite[Question 1]{AP_measures}, which asks whether extended supports are always separable subsets of $\ucomp{M}$ or, failing that, satisfy the countable chain condition (ccc). The answer is negative for the former, but positive for the latter if we are restricted to functionals avoiding infinity.

\begin{corollary}\label{co:ccc}
If $\phi\in\Lip_0(M)^*$ avoids infinity then $\esupp{\phi}$ satisfies the ccc.
\end{corollary}

\begin{proof}
The ccc is stable under countable unions and taking continuous images. Since the support of any $\mu\in\meas{\bwt{M}}$ always satisfies the ccc, so does the set $\pp_s(\supp(\mu)) = \pp_1(\supp(\mu)) \cup \pp_2(\supp(\mu))$. By Proposition \ref{pr:minimal_support_representation}, we may choose $\mu$ such that $\pp_s(\supp(\mu))$ is either $\esupp{\phi}$ or $\esupp{\phi}\cup\set{0}$. If it equals $\esupp{\phi}$ then we are done. Otherwise, we have $\esupp{\phi}\cup\set{0}\neq\esupp{\phi}$, hence $0$ is isolated in $\pp_s(\supp(\mu))$ and so $\esupp{\phi}=\pp_s(\supp(\mu))\setminus\set{0}$ satisfies the ccc as well.
\end{proof}

\begin{example}\label{ex:non-sep}
Let $M = \NN$ with the discrete metric and base point $1$. Then there exists $\phi\in\Lip_0(M)^*$ such that $\esupp{\phi}$ is non-separable. Indeed, according to \cite{B_NZ20}, there exists (in ZFC) a compactification $K$ of $\NN$ such that $K\setminus \NN$ is non-separable, together with a strictly positive measure $\nu \in \meas{K\setminus \NN}$, i.e. $\nu$ is positive and $\supp(\nu) = K\setminus \NN$. Let $q:\beta\NN = \ucomp{M} \to K$ denote the natural surjective quotient map obtained by continuously extending the identity embedding from $\NN$ into $K$. This induces an isometric embedding $Q:C(K)\to C(\beta\NN)$ given by $Qf=f \circ q$. By applying the Hahn-Banach theorem to the functional $f \mapsto \duality{Q^{-1}f,\nu}$, $f \in Q(C(K)) \subseteq C(\beta\NN)$, we obtain positive $\lambda \in \meas{\beta\NN}$ satisfying $Q^*\lambda = q_\sharp\lambda = \nu$. Then $q(\supp(\lambda))=\supp(\nu)=K\setminus\NN$ is non-separable, which implies $\supp(\lambda)$ is also non-separable. Finally, let $\phi=\widehat{\lambda}$ be the functional induced by $\lambda$ as in \eqref{eq:induced_functional}. By \cite[Proposition 4.7]{AP_measures}, $\esupp{\phi}=\supp(\lambda)$ is non-separable.
\end{example}

\section{De Leeuw representations of elements of \texorpdfstring{$\lipfree{M}$}{F(M)}}\label{sect_representations_of_F(M)}

In this final section we will focus on the representation of elements of $\lipfree{M}$. No simple way is known to completely distinguish the (not necessarily optimal) De Leeuw measures that represent weak$^\ast$ continuous functionals from those that do not. The simplest and most natural criterion, being concentrated on $\wt{M}$, is a sufficient condition by Proposition \ref{pr:wt_bochner} but it is not necessary unless $M$ is finite and $\bwt{M}=\wt{M}$. To see this, consider the mapping $\rr:\bwt{M}\to\bwt{M}$ from Section \ref{subsec_vanishing}. For any $\zeta\in\bwt{M}\setminus\wt{M}$, $\delta_\zeta+\delta_{\rr(\zeta)}$ is a non-zero measure on $\bwt{M}$, concentrated outside of $\wt{M}$, that represents the $0$ functional by \cite[Lemma 2]{Aliaga}.

With that example in mind, the next simplest criterion one could hope for would be the one described by the following property:

\begin{description}
\item[($*$)]\labeltext{($*$)}{P1} Any $\mu\in\meas{\bwt{M}}$ such that $\dual{\Phi}\mu\in\lipfree{M}$ satisfies $\dual{\Phi}\mu=\dual{\Phi}(\mu\restrict_{\wt{M}})$.
\end{description}

Using Proposition \ref{pr:wt_bochner} and the fact that any element of $\lipfree{M}$ can be written as a (non-convex) series of molecules, we obtain some equivalent formulations of the property:
\begin{itemize}
\item If $\mu\in\meas{\bwt{M}}$ and $\dual{\Phi}\mu=0$ then $\dual{\Phi}(\mu\restrict_{\wt{M}})=0$.
\item If $\mu\in\meas{\bwt{M}}$ is concentrated outside of $\wt{M}$ and $\dual{\Phi}\mu\in\lipfree{M}$ then $\dual{\Phi}\mu=0$.
\end{itemize}

The intuitive idea behind property \ref{P1} is that if a measure on $\bwt{M}\setminus\wt{M}$ represents an element of $\lipfree{M}$, it should be because ``it is composed of parts that cancel each other out somehow'' as in the example above. Before going further, we observe that this property is still not true for every $M$.

\begin{theorem}[{\cite[Theorem 4.1]{APS23}}]
\label{th:isometric_fat_cantor_example}
If $M$ contains an isometric copy of a subset of $\RR$ with positive Lebesgue measure, then there is a non-zero $m\in\lipfree{M}$ with the property that $\supp(\mu)\cap\wt{M}=\varnothing$ for any optimal De Leeuw representation $\mu$ of $m$.
\end{theorem}

\noindent However, we will see in Theorem \ref{th:p1} that it holds for a nontrivial class of metric spaces.

One of the reasons for our interest in property \ref{P1} is its relation with the open problem of determining which Lipschitz-free spaces are complemented in their biduals (see \cite{CKK19} for the current state of the art on that problem). Recall that a Banach space $X$ is \emph{L-embedded} when it is the range of an L-projection from $\ddual{X}$, i.e. there exists a subspace $Z$ of $\ddual{X}$ such that $\ddual{X}=X\oplus_1 Z$.

\begin{proposition}
\label{pr:p1_l_embedded}
If $M$ satisfies condition \ref{P1} then $\lipfree{M}$ is L-embedded. The L-projection $P:\Lip_0(M)^*\to\lipfree{M}$ is given by
$$
P(\phi)=\dual{\Phi}(\mu\restrict_{\wt{M}}) \quad\text{where}\quad \phi=\dual{\Phi}\mu\in\Lip_0(M)^*.
$$
\end{proposition}

\begin{proof}
Let us first check that $P$ is well defined, that is, $P(\phi)$ does not depend on the choice of the De Leeuw representation $\mu$ of $\phi\in\dual{\Lip_0(M)}$. Let $\mu,\mu'\in\meas{\bwt{M}}$ be such that $\dual{\Phi}\mu=\dual{\Phi}\mu'=\phi$, then $\dual{\Phi}(\mu-\mu')=0\in\lipfree{M}$ and so property \ref{P1} implies that $0=\dual{\Phi}((\mu-\mu')\restrict_{\wt{M}})=\dual{\Phi}(\mu\restrict_{\wt{M}})-\dual{\Phi}(\mu'\restrict_{\wt{M}})$ as well.

The linearity of $P$ is obvious, and the image of $P$ is contained in $\lipfree{M}$ by Proposition \ref{pr:wt_bochner}. Since every element of $\lipfree{M}$ can be written as a (non-convex) series of molecules and thus admits a (non-optimal) De Leeuw representation concentrated on $\wt{M}$, we also have $P(\phi)=\phi$ whenever $\phi\in\lipfree{M}$. Therefore $P$ is a projection onto $\lipfree{M}$.

Finally, let $\phi\in\dual{\Lip_0(M)}$ and let $\mu\in\opr{\bwt{M}}$ be such that $\dual{\Phi}\mu=\phi$. Then $\mu\restrict_{\wt{M}}$ and $\mu\restrict_{\bwt{M}\setminus\wt{M}}$ also belong to $\opr{\bwt{M}}$ by Proposition \ref{pr:opr_facts} (d), and we conclude
\begin{align*}
\norm{\phi} = \norm{\mu} &= \norm{\mu\restrict_{\wt{M}}} + \norm{\mu\restrict_{\bwt{M}\setminus\wt{M}}} \\
&= \norm{\dual{\Phi}(\mu\restrict_{\wt{M}})} + \norm{\dual{\Phi}(\mu\restrict_{\bwt{M}\setminus\wt{M}})} \\
&= \norm{\dual{\Phi}(\mu\restrict_{\wt{M}})} + \norm{\dual{\Phi}\mu-\dual{\Phi}(\mu\restrict_{\wt{M}})} = \norm{P(\phi)} + \norm{\phi-P(\phi)}
\end{align*}
i.e. $P$ is an L-projection onto $\lipfree{M}$.
\end{proof}

Proposition \ref{pr:p1_l_embedded} should be compared to Remark \ref{remark:decomp_bidual}. There, we showed that restricting De Leeuw representations to $\bwtf$ induces an L-projection from $\Lip_0(M)^*$ onto the space of functionals that avoid infinity. Here, restricting to $\wt{M}$ induces an L-projection onto $\lipfree{M}$, but the operation is only well-defined under assumption \ref{P1}.

Let us now identify the metric spaces that satisfy condition \ref{P1}. It is obvious that this property is hereditary, that is, it passes to subsets. A key fact is that it is also invariant under bi-Lipschitz mappings.

\begin{lemma}
\label{lm:p1_invariant_bilip}
Let $M$ and $N$ be bi-Lipschitz equivalent metric spaces. Then $M$ satisfies \ref{P1} if and only if $N$ satisfies \ref{P1}.
\end{lemma}

\begin{proof}
Suppose that $M$ does not satisfy \ref{P1}, so that there is $m\in\lipfree{M}$, $m\neq 0$ that is represented by a measure $\mu\in\meas{\bwt{M}}$ concentrated outside of $\wt{M}$. We may assume that $\mu$ is positive by replacing it with $\mu^++\rr_\sharp\mu^-$ where $\rr$ is the map defined after Proposition \ref{pr:vanishing_points} (see \cite[Lemma 2 and Proposition 3]{Aliaga}). We will show that a pushforward of $\mu$ to $\bwt{N}$ derived from the bi-Lipschitz equivalence of the metric spaces witnesses the failure of \ref{P1} on $N$ as well.

Let $\psi:M\to N$ be a bi-Lipschitz map, and assume without loss of generality that $0_N=\psi(0_M)$. Define a map $\wt{\psi}:\wt{M}\to\wt{N}$ by $\wt{\psi}(x,y)=(\psi(x),\psi(y))$ for every $(x,y)\in\wt{M}$. Then $\wt{\psi}$ is continuous and hence can be continuously extended to a map $\bwt{M}\to\bwt{N}$; we call the extended map $\wt{\psi}$ again. Similarly we can define a continuous map $\wt{\psi^{-1}}:\bwt{N}\to\bwt{M}$. Notice that $\wt{\psi^{-1}}\circ\wt{\psi}$ is the identity on $\wt{M}$, thus also on $\bwt{M}$ by continuity. It follows that $\wt{\psi}$ and $\wt{\psi^{-1}}$ are inverse bijections, and in particular $\wt{\psi}^{-1}(\wt{N})=\wt{M}$.

If we define
$$w(x,y)=\frac{d_N(\psi(x),\psi(y))}{d_M(x,y)},\quad (x,y)\in\wt{M},$$
where $d_M, d_N$ denote the metrics in $M$ and $N$, respectively,
then $w$ is a continuous map from $\wt{M}$ to $[\Lip(\psi^{-1})^{-1},\Lip(\psi)]\subset\RR$. Hence, it can also be continuously extended to a map from $\bwt{M}$ to $[\Lip(\psi^{-1})^{-1},\Lip(\psi)]$; we again keep the notation $w$. Now, consider the weighted measure $\mu_w\in\meas{\bwt{M}}$ given by $d\mu_w=w\,d\mu$ and its pushforward $$\nu=\wt{\psi}_\sharp\mu_w\in\meas{\bwt{N}}.$$
Then $\nu$ is concentrated outside of $\wt{N}$, as
$$\nu(\wt{N})=\mu_w(\wt{\psi}^{-1}(\wt{N}))=\mu_w(\wt{M})=0$$
because $\mu_w$ is absolutely continuous with respect to $\mu$.

Finally, let us verify that $\Phi^\ast\nu=\widehat{\psi}(m)\in\lipfree{N}$, where $\widehat{\psi}:\lipfree{M}\to\lipfree{N}$ is the isomorphism induced by $\psi$, i.e. such that $\widehat{\psi}\circ\delta_M=\delta_N\circ\psi$. Let $f\in\Lip_0(N)$. Then
$$
\duality{f,\Phi^\ast\nu}=\int_{\bwt{N}}(\Phi f) \,d\nu=\int_{\bwt{N}}(\Phi f)\, d(\wt{\psi}_{\sharp}\,\mu_w)
=\int_{\bwt{M}}(\Phi f\circ\wt{\psi})w\, d\mu.
$$
But for any $(x,y)\in\wt{M}$, we have
\begin{align*}
   (\Phi f\circ\wt{\psi})(x,y)\,w(x,y)&=(\Phi f)(\psi(x),\psi(y))\,\frac{d_N(\psi(x),\psi(y))}{d_M(x,y)}\\
   &=\frac{f(\psi(x))-f(\psi(y))}{d_N(\psi(x),\psi(y))}\,\frac{d_N(\psi(x),\psi(y))}{d_M(x,y)}\\
   &=\frac{(f\circ\psi)(x)-(f\circ\psi)(y)}{d_M(x,y)}=\Phi(f\circ\psi)(x,y),
\end{align*}
and hence $(\Phi f\circ\wt{\psi})\,w=\Phi(f\circ \psi)$ on all of $\bwt{M}$ by continuity. Therefore, 
$$
\duality{f,\Phi^\ast\nu}=\int_{\bwt{M}}\Phi(f\circ\psi)\,d\mu=\duality{f\circ\psi,m}=\duality{f,\widehat{\psi}(m)}
$$
and $\Phi^\ast\nu=\widehat{\psi}(m)$ as claimed. In particular, $\dual{\Phi}\nu\in\lipfree{N}\setminus\set{0}$. Thus property \ref{P1} fails on $N$ as well, and this ends the proof.
\end{proof}

Lemma \ref{lm:p1_invariant_bilip} allows us to obtain some strong necessary conditions for property \ref{P1} by considering the available counterexamples. Recall that $M$ is \emph{purely $1$-unrectifiable} if the image of every Lipschitz map from a subset of $\RR$ into $M$ has null $1$-Hausdorff measure. By a result of Kirchheim, $M$ is purely $1$-unrectifiable if and only if it contains no \emph{curve fragment}, i.e. no bi-Lipschitz copy of a compact subset of $\RR$ of positive Lebesgue measure (see \cite{Kirchheim} for the original proof and \cite[Section 1.3]{AGPP} for an expository treatment).

\begin{proposition}
\label{pr:p1_implies_proper_p1u}
If $M$ has property \ref{P1} then it is proper and purely 1-unrectifiable.
\end{proposition}

\begin{proof}
Combining Theorem \ref{th:isometric_fat_cantor_example} with Lemma \ref{lm:p1_invariant_bilip}, we see that curve fragments fail property \ref{P1}. Since property \ref{P1} is hereditary, it can therefore be satisfied only by purely $1$-unrectifiable $M$.

On the other hand, if $M$ is not proper then there exists $\xi\in\rcomp{M}\setminus M$. Fix $x\neq y\in M$ and, using Proposition \ref{prop_p_range}, let $\zeta \in \pp^{-1}(x,\xi)$ and $\omega \in \pp^{-1}(\xi,y)$, and set $\mu=d(\zeta)\delta_\zeta + d(\omega)\delta_\omega$. As $\zeta,\omega \in \bwtf$, by Proposition \ref{prop_de_Leeuw_relation} \ref{delta_diff},
\[
 \Phi^*\mu = d(\zeta)\Phi^*\delta_\zeta + d(\omega)\Phi^*\delta_\omega = \delta(\pp_1(\zeta)) - \delta(\pp_2(\zeta)) + \delta(\pp_1(\omega)) - \delta(\pp_2(\omega)) = \delta(x) - \delta(y)
\]
belongs to $\lipfree{M}$. However $\mu\restrict_{\wt{M}}=0$, so $M$ does not have property \ref{P1}.
\end{proof}

We will now complete the characterisation of property \ref{P1} by proving that the converse of Proposition \ref{pr:p1_implies_proper_p1u} is also true. To that end, we need to recall several facts. It is established in \cite{AGPP} that pure 1-unrectifiability characterises some important Banach space properties of Lipschitz-free spaces. In particular, when $M$ is proper, \cite[Theorem B]{AGPP} states that $\lipfree{M}$ is a dual space precisely when $M$ is purely 1-unrectifiable, and in that case one of its preduals is $\lip_0(M)$, the space of all functions in $\Lip_0(M)$ that are \emph{locally flat} and \emph{flat at infinity}. Rather than defining these two notions, we will use a characterisation of these functions given by \cite[Proposition 4.20]{Weaver2}: if $M$ is proper, then a function $f\in\Lip_0(M)$ belongs to $\lip_0(M)$ if and only if $\Phi f(\zeta)=0$ for all $\zeta\in\bwt{M}\setminus\wt{M}$.

\begin{theorem}
\label{th:p1}
$M$ has property \ref{P1} if and only if it is proper and purely 1-unrectifiable.
\end{theorem}

\begin{proof}
One implication is given by Proposition \ref{pr:p1_implies_proper_p1u}. For the other one, suppose that $M$ is purely 1-unrectifiable and proper. Fix $\mu\in\meas{\bwt{M}}$ such that $\dual{\Phi}\mu\in\lipfree{M}$, and let $f\in B_{\Lip_0(M)}$. By \cite[Theorem B]{AGPP}, $\lip_0(M)$ is an isometric predual of $\lipfree{M}$ and hence $B_{\lip_0(M)}$ is weak$^*$ dense in $B_{\Lip_0(M)}$, so we may choose a net $(f_i)$ in $B_{\lip_0(M)}$ that converges weak$^\ast$ to $f$. It follows that
\begin{align*}
\duality{f,\dual{\Phi}\mu} = \lim_i \duality{f_i,\dual{\Phi}\mu} &= \lim_i \int_{\bwt{M}}(\Phi f_i)\,d\mu \\
&= \lim_i \int_{\wt{M}}(\Phi f_i)\,d\mu = \lim_i \duality{f_i,\dual{\Phi}(\mu\restrict_{\wt{M}})} = \duality{f,\dual{\Phi}(\mu\restrict_{\wt{M}})}
\end{align*}
where we have used the facts that $\Phi f_i$ vanishes outside of $\wt{M}$ by \cite[Proposition 4.20]{Weaver2} and $\dual{\Phi}(\mu\restrict_{\wt{M}})\in\lipfree{M}$ by Proposition \ref{pr:wt_bochner}. Since $f$ was arbitrary, we conclude that $\dual{\Phi}\mu=\dual{\Phi}(\mu\restrict_{\wt{M}})$. This shows that $M$ satisfies \ref{P1}.
\end{proof}

From this we deduce the following result.

\begin{corollary}
\label{cr:p1_convex}
If $M$ is proper and purely 1-unrectifiable then all elements of $\lipfree{M}$ are convex integrals of molecules. If $M$ is proper and countable then all elements of $\lipfree{M}$ are convex series of molecules.
\end{corollary}

Taking Theorem \ref{th:p1} into account, Proposition \ref{pr:p1_l_embedded} only tells us that $\lipfree{M}$ is L-embedded when $M$ is proper and purely 1-unrectifiable. This is a known result, essentially contained in a side remark in Godefroy's paper \cite{Godefroy_survey} (see the comment before Lemma 5.2 therein) and expanded to its current form in \cite{AGPP}. However, Godefroy's argument is abstract and relies on the theory of M-ideals. Our proof of Proposition \ref{pr:p1_l_embedded}, on the other hand, is constructive and provides an explicit and natural description of the L-projection of $\ddual{\lipfree{M}}$ onto $\lipfree{M}$: in the De Leeuw domain, it is simply the restriction operator to $\wt{M}$. We also remark that condition \ref{P1} is not necessary for L-embeddability: for instance, $\lipfree{[0,1]}=L_1$ is L-embedded although $[0,1]$ is clearly not purely $1$-unrectifiable.

As a byproduct of the ideas behind Proposition \ref{pr:p1_l_embedded}, we also obtain an explicit description of the bidual $\ddual{\lipfree{M}}$ in some particularly simple cases by determining the L-complement of $\lipfree{M}$. We recall that $\phi\in\Lip_0(M)^*$ is a \emph{derivation at $x \in M$} if $\duality{f,\phi}=0$ for any $f\in\Lip_0(M)$ that is constant in a neighbourhood of $x$ (see \cite[Section 2.5]{AP_measures} and \cite[Section 7.5]{Weaver2}).

\begin{proposition}
\label{pr:countable_bidual}
Let $M$ be a countable proper metric space. Then
$$
\Lip_0(M)^* = \lipfree{M} \oplus_1 \pare{\bigoplus_{x\in M'} \mathcal{D}(x)}_1 \oplus_1 \Lambda_\infty(M)
$$
where $\mathcal{D}(x)$ is the space of all derivations at $x$, $M'$ is the derived set of $M$, $\Lambda_\infty(M)$ is the space of functionals concentrated at infinity, and all direct sums denote $\ell_1$-sums.
\end{proposition}

\begin{proof}
Recall from \eqref{eq:decomp_bidual} that $\Lip_0(M)^*$ is the $\ell_1$-sum of $\Lambda_\infty(M)$ and $\Lambda_0(M)$, the space of functionals that avoid infinity, so we have to prove that
\begin{equation}
\label{eq:countable_avoid_infty}
\Lambda_0(M) = \lipfree{M} \oplus_1 \pare{\bigoplus_{x\in M'} \mathcal{D}(x)}_1 .
\end{equation}
Let $\phi\in\Lambda_0(M)$ and fix an optimal representation $\mu\in\opr{\bwt{M}}$ of $\phi$. By Corollary \ref{cor_avoid_infinity}, $\mu$ is concentrated on $\bwtf$ and, since $M$ is proper, Proposition \ref{prop_p_range} ensures that
$$
\bwtf = \wt{M}\cup\bigcup_{x\in M'}\pp^{-1}(x,x) .
$$
Hence
$$
\phi = \dual{\Phi}\mu = \dual{\Phi}(\mu\restrict_{\widetilde{M}}) + \sum_{x\in M'}\dual{\Phi}(\mu\restrict_{\pp^{-1}(x,x)}) \,.
$$
Note that the sum makes sense because it is absolutely convergent and $M'$ is countable. The term $\dual{\Phi}(\mu\restrict_{\widetilde{M}})$ belongs to $\lipfree{M}$ by Proposition \ref{pr:wt_bochner}, and one can verify that each term in the rightmost sum is a derivation at $x\in M'$ (see e.g. the proof of \cite[Lemma 2.4]{APS23}). It is also clear that each $\mathcal{D}(x)$ is a closed subspace of $\Lambda_0(M)$, that $\mathcal{D}(x)\cap\mathcal{D}(x')=\set{0}$ for $x\neq x'$, and that $\lipfree{M}\cap\mathcal{D}(x)=\set{0}$ (see e.g. \cite[Corollary 3.24]{AP_measures}). All of this proves that the formula \eqref{eq:countable_avoid_infty} holds as a direct sum. Finally, the fact that it is an $\ell_1$-sum is precisely the statement of \cite[Proposition 2.11]{AP_measures}.
\end{proof}

\begin{remark}
In particular, if $M$ is compact and countable then all functionals in $\Lip_0(M)^*$ avoid infinity and we have simply
$$
\Lip_0(M)^* = \lipfree{M} \oplus_1 \pare{\bigoplus_{x\in M'} \mathcal{D}(x)}_1 .
$$
On the other hand, if $M$ is proper and (topologically) discrete, e.g. $M=\ZZ^n$, then there are no nontrivial derivations on points of $M$ and thus all functionals on $\Lip_0(M)$ that avoid infinity are weak$^*$ continuous, so $\Lip_0(M)^* = \lipfree{M} \oplus_1 \Lambda_\infty(M)$.
\end{remark}

We finish this section by formulating an open problem (or rather, several of them) related to the characterisation of De Leeuw representations of elements of $\lipfree{M}$. Property \ref{P1} is rather strong because it applies to \textit{all} De Leeuw representations. It is perhaps more natural to restrict our attention to optimal representations, and in that case we may ask the following questions.

\begin{problem}\label{problem}
Characterise the complete metric spaces $M$ such that
\begin{enumerate}
\item[(a)] every optimal representation $\mu\in\opr{\bwt{M}}$ of an element of $\lipfree{M}$ is concentrated on $\wt{M}$;
\item[(b)] every element of $\lipfree{M}$ is a convex integral of molecules;
\item[(c)] every element of $\lipfree{M}$ is a convex series of molecules.
\end{enumerate}
\end{problem}

Property (a) is the restriction of property \ref{P1} to just optimal representations. Also, property (b) requires that there is at least one representation as in (a), instead of requiring that all optimal representations have that form. Thus we have the implications \ref{P1}$\Rightarrow$(a)$\Rightarrow$(b). Also clearly (c)$\Rightarrow$(b), and in fact (b)$\Leftrightarrow$(c) if $M$ is scattered (\cite[Proposition 2.8]{APS23}). It is also obvious that property (a) is hereditary, but we do not know whether (b) or (c) are. Results in this paper and \cite{APS23} yield only partial information about those $M$ satisfying each property.
For instance, (b) holds when $M$ is proper and purely 1-unrectifiable (Theorem \ref{th:p1}), or uniformly discrete (\cite[Corollary 3.7]{APS23}).

A source of interest in this problem is the conjecture that all extreme points of $B_{\lipfree{M}}$ are elementary molecules. In {\cite[Theorem 6.1]{APS23}}, it is shown that an extreme point of $B_{\lipfree{M}}$ that admits an optimal De Leeuw representation with a positive mass on $\wt{M}$ (in particular, if it is a convex integral of molecules) must be an elementary molecule. The conjecture is therefore true for $M$ with any of the properties above.

\subsection*{Note added after revision}
After initial submission of this manuscript, the preprints \cite{Smith} and \cite{APS3} have been published, in which the theory developed in this paper is expanded upon and some of our results are strengthened:
\begin{itemize}
\item Condition (b) in Problem \ref{problem} is indeed hereditary (see \cite[Corollary 4.13]{Smith}).
\item All purely $1$-unrectifiable spaces satisfy condition (b) in Problem \ref{problem}. Thus, the assumption that $M$ is proper can be dropped from Corollary \ref{cr:p1_convex} (see \cite[Theorem 4.1 and Corollary 4.3]{APS3}).
\item In Theorem \ref{th:bidual_majorisable}, the conclusion $(\pp_1)_\sharp\mu\restrict_{\rcomp{M}\setminus\set{0}}\perp(\pp_2)_\sharp\mu\restrict_{\rcomp{M}\setminus\set{0}}$ can be replaced with $(\pp_1)_\sharp\mu\perp(\pp_2)_\sharp\mu$ (see \cite[Theorem 4.15]{Smith}).
\end{itemize}
In addition, the conjecture that all extreme points of $B_{\lipfree{M}}$ are molecules, mentioned in the paragraph above, is solved in the affirmative (see \cite[Theorem 3.1]{APS3}).

\section*{Acknowledgements}

The authors would like to thank Damian Sobota for bringing \cite{B_NZ20} to their attention, which enabled them to complete Example \ref{ex:non-sep}, and Marek C\'uth for discussions that inspired Proposition \ref{pr:countable_bidual}. Some of this research was carried out when the first and third authors visited the Faculty of Information Technology at the Czech Technical University in Prague in October 2023, and during the second author's visit to the School of Mathematics and Statistics at University College Dublin in November 2023. The authors would like to thank the respective hosts for their hospitality.

R. J. Aliaga was partially supported by Grant PID2021-122126NB-C33 funded by MICIU/AEI/ 10.13039/501100011033 and by ERDF/EU. E. Perneck\'a was supported by the grant GA\v CR 22-32829S of the Czech Science Foundation.

%On behalf of all authors, the corresponding author states that there is no conflict of interest.

% BIBLIOGRAPHY

\end{document}